\documentclass[11pt,a4paper]{article}
\usepackage{amsfonts}
\usepackage{amssymb}
\usepackage{amsthm}
\usepackage{amsmath}
\usepackage{dsfont}
\usepackage{cite}
\usepackage{mathrsfs}
\newtheorem{theorem}{\textbf{Theorem}}[section]
\newtheorem{lemma}{\textbf{Lemma}}[section]
\newtheorem{proposition}{\textbf{Proposition}}[section]
\newtheorem{corollary}{\textbf{Corollary}}[section]
\newtheorem{remark}{\textbf{Remark}}[section]
\newtheorem{definition}{\textbf{Definition}}[section]

\allowdisplaybreaks[4]

\def\be{\begin{equation}}
\def\ee{\end{equation}}
\def\bea{\begin{eqnarray}}
\def\eea{\end{eqnarray}}
\def\bt{\begin{theorem}}
\def\et{\end{theorem}}
\def\bl{\begin{lemma}}
\def\el{\end{lemma}}
\def\br{\begin{remark}}
\def\er{\end{remark}}
\def\bp{\begin{proposition}}
\def\ep{\end{proposition}}
\def\bc{\begin{corollary}}
\def\ec{\end{corollary}}
\def\bd{\begin{definition}}
\def\ed{\end{definition}}

\def\non{\nonumber }

\def\d{\mathrm{d}}

\begin{document}

\title{An Energetic Variational Approach for the Cahn--Hilliard Equation with Dynamic Boundary Condition:\\ Model Derivation and Mathematical Analysis}

\author{
{\sc Chun Liu}
\footnote{Department of Applied Mathematics, Illinois Institute of Technology,
Chicago, IL 60616, USA, Email: \texttt{cliu124@iit.edu}}
\ and
{\sc Hao Wu} \footnote{School of Mathematical Sciences; Key Laboratory of Mathematics for Nonlinear Sciences (Fudan University), Ministry of Education; and Shanghai
Key Laboratory for Contemporary Applied Mathematics, Fudan
University, Shanghai 200433, China, Email:
\texttt{haowufd@fudan.edu.cn, haowufd@yahoo.com}}
}
\date{\today}
\maketitle


\begin{abstract}
The Cahn--Hilliard equation is a fundamental model that describes phase separation processes of binary mixtures.
In recent years, several types of dynamic boundary conditions have been proposed in order to account for possible short-range interactions of the material with the solid wall.
Our first aim in this paper is to propose a new class of dynamic boundary conditions for the Cahn--Hilliard equation in a rather general setting.
The derivation is based on an energetic variational approach that combines the least action principle and Onsager's principle of maximum energy dissipation.
One feature of our model is that it naturally fulfills three important physical constraints such as conservation of mass, dissipation of energy and force balance relations.
Next, we provide a comprehensive analysis of the resulting system of partial differential equations.
Under suitable assumptions, we prove the existence and uniqueness of global weak/strong solutions to the initial boundary value problem with or without surface diffusion.
Furthermore, we establish the uniqueness of asymptotic limit as $t\to+\infty$ and characterize the stability of local energy minimizers for the system.

\medskip


\medskip

\end{abstract}

\tableofcontents

\section{Introduction}
\setcounter{equation}{0}

The hydrodynamics of mixtures of materials plays an increasingly important role in scientific and engineering applications \cite{AN97}.
Different approaches can be found in the literature for the modelling and simulation of multi-phase problems.
The conventional sharp interface model consists of separate hydrodynamic system for each component, together with free interfaces that separate them.
However, this interface-capturing method breaks down when the interfaces experience topological changes like merging and splitting \cite{JT98}.
On the other hand, the so-called diffuse interface models replace the hyper-surface description of interfaces with a thin layer, where microscopic mixing of the macroscopically distinct components of matter are allowed.
This not only yields systems of partial differential equations that are
better amenable to further analysis, but possible topological changes of interfaces can also be
handled in a natural way \cite{AN97,JT98}. The Cahn--Hilliard equation is a fundamental diffuse interface model for multi-phase systems.
It was first proposed in the materials science to describe the pattern formation evolution of microstructures
during the phase separation process in binary alloys \cite{CH58, CH61}.
Later on,  because of the interesting features of its dynamics with an early
linear spinodal regime followed by a late coarsening regime, the Cahn--Hilliard equation and its variants have been successfully applied in a wide variety of segregation-like phenomena in science, see for instance
\cite{BE95, LS03, FM97, FR98, K-etal, KS08, Gur93, Gur96, NC88, NC08, OP13, Q1} and the references therein.

Assume that $T\in (0,+\infty)$, $\Omega\subset \mathbb{R}^d$ ($d=2,3$) is a bounded domain with smooth boundary $\Gamma=\partial \Omega$ and $\mathbf{n}=\mathbf{n}(x)$ is the unit outer normal vector on $\Gamma$. The classical Cahn--Hilliard equation can be written in the following (dimensionless) form:
\begin{equation}
\left\{
\begin{aligned}
&\phi_t=\Delta \mu,&\text{in}\ \Omega \times (0,T),\\
&\mu=-\Delta \phi +F'(\phi),&\text{in}\ \Omega \times (0,T),
\end{aligned}
\right.\label{claCH}
\end{equation}
where the subscript $t$ denotes the partial derivative with respect to time and $\Delta$ represents the usual Laplace
operator acting on the spatial variables in $\Omega$.
The phase-field order parameter $\phi$ represents the the difference of local relative concentrations for two components of the binary mixture such that $\phi=\pm 1$ correspond to the pure phases of the material, while
$\phi\in (-1, 1)$ corresponds to the transition between them.
$\mu$ stands for the chemical potential that equals to the Fr\'echet derivative of certain bulk free energy given by
\begin{align}
E^\text{bulk}(\phi)=\int_\Omega \frac12 |\nabla \phi|^2+F(\phi) \, \d x.\label{eebulk}
\end{align}
  $F'$ denotes the first derivative of the ``bulk" potential $F$ that usually has a double-well structure, with two minima and a local unstable maximum in between.
 A typical thermodynamically relevant potential $F$ is the following logarithmic potential \cite{CH58}
\begin{align*}
&F(y)=\frac{\theta}{2}(1+y)\ln (1+y)+(1-y)\ln(1-y)-\frac{\theta_c}{2}y^2,\quad y\in[-1,1],\ \ 0<\theta<\theta_c,
\end{align*}
where $\theta$ is the temperature of the system and $\theta_c$ is the critical temperature of phase separation. Besides, the above singular potential is very often approximated by regular ones, typically $$F(y)=\frac14(y^2-1)^2,\quad y\in \mathbb{R}.$$
The Cahn--Hilliard equation \eqref{claCH} can be formally derived as the conserved dynamics generated by the variational derivative of the free energy with respect to the phase function $\phi$ (i.e., the $H^{-1}$-gradient flow) or based on the second law of thermodynamics (see, e.g., \cite{CH58,NC88,Gur96} and references therein).

Suitable boundary conditions should be taken into account for system \eqref{claCH} when the evolution is confined in a bounded domain $\Omega$.
Classical choices are the following homogeneous Neumann boundary conditions
\begin{align}
&\partial_\mathbf{n}\mu=0, \quad \text{on}\ \Gamma\times (0,T),\label{bdN1}\\
&\partial_\mathbf{n}\phi=0, \quad \text{on}\ \Gamma\times (0,T),\label{bdN2}
\end{align}
where $\partial_\mathbf{n}$ denotes the outward normal derivative on $\Gamma$.
The no-flux boundary condition \eqref{bdN1} for the chemical potential $\mu$ guarantees the conservation of mass in the bulk
$$
\int_\Omega \phi(t) \,\d x=\int_\Omega \phi(0)\,\d x,\quad \forall\, t\in [0,T],
$$
while the second boundary condition \eqref{bdN2} has the physical interpretation that the diffused interface, which separates the two phases of the
material, intersects the solid wall (i.e., the boundary $\Gamma$) at a perfect static contact angle of $\frac{\pi}{2}$. In other words, the interaction between the material and wall of the container is simply neglected (see \cite{BF93}). Another important consequence of \eqref{bdN1}--\eqref{bdN2} is that the bulk free energy $E^\text{bulk}(\phi)$ given by \eqref{eebulk} is decreasing in time,
namely,
$$
\frac{\d}{\d t} E^\text{bulk}(\phi(t)) + \int_\Omega |\nabla \mu|^2 \, \d x=0,\quad \forall\, t\in (0,T).
$$

We note that \eqref{bdN2} turns out to be a rather restrictive assumption on the boundary for many materials and possible influence of the boundary to the bulk dynamics is neglected. In order to describe certain effective short-range interactions between the solid wall and the mixture, suitable surface free energy functional was introduced by physicists into the system (see, e.g., \cite{FM97, FR98, K-etal})
\begin{align}
E^\text{surf}(\phi)=\int_\Omega \frac{\kappa}{2} |\nabla_\Gamma \phi|^2+G(\phi) \, \d S,\quad \quad \text{for some }\kappa\geq 0,
\label{eesurf}
\end{align}
where $\nabla_\Gamma$ stands for the tangential (surface) gradient operator defined on $\Gamma$ and $G$ is a certain surface potential function
that for instance, characterizes the possible preferential attraction (or repulsion) of one of the components
by the wall. The coefficient $\kappa$ is related to the surface diffusion. When $\kappa=0$, the model is closely related to the evolution of an interface in contact with the solid boundary, i.e., the well-known moving contact line problem \cite{TR89}.

For the Cahn--Hilliard system \eqref{claCH}, the following dynamic boundary condition was proposed (see e.g., \cite{K-etal}) to replace the homogeneous Neumann one \eqref{bdN2}:
\begin{align}
\phi_t-\kappa \Delta_\Gamma \phi +\partial_\mathbf{n}\phi +G'(\phi)=0, \quad \text{on}\ \Gamma\times (0,T),\label{bdN2d}
\end{align}
where $\Delta_\Gamma$ denotes the Laplace--Beltrami operator on $\Gamma$ and the term $\Delta_\Gamma \phi$ corresponds to the possible surface diffusion phenomenon along the boundary (see also \cite{Q1,MH13} for the more complicated case with fluid interactions). Formally speaking, the dynamic boundary condition \eqref{bdN2d} can be viewed as an $L^2$-gradient flow of the surface free energy $E^\text{surf}$ on $\Gamma$ (with interaction from the bulk presented by the term $\partial_\mathbf{n}\phi$).
It is sometimes referred to as a variational boundary condition, since the total free energy (i.e., sum of the bulk and surface free energies \eqref{eebulk}, \eqref{eesurf}) is decreasing in time under the choice of \eqref{bdN2d} (combined with \eqref{bdN1}):
\begin{align}
\frac{\d}{\d t} \Big[E^\text{bulk}(\phi(t))+E^\text{surf}(\phi(t))\Big] + \int_\Omega |\nabla \mu|^2 \,\d x+\int_\Gamma |\phi_t|^2\, \d S=0,\quad \forall\, t\in (0,T).
\label{bbela}
\end{align}

Other types of dynamic boundary conditions can be derived under different considerations in physics.
For instance, if mass exchange between the bulk and the boundary is allowed, the following boundary condition was proposed in \cite{GMS11} in place of the no-flux boundary condition \eqref{bdN1}:
\begin{align}
\phi_t-\sigma \Delta_\Gamma \mu + \partial_\mathbf{n} \mu=0, \quad \text{on}\ \Gamma\times (0,T),\quad \text{for some }\sigma \geq 0.
\label{bdN1a}
\end{align}
When $\sigma =0$, \eqref{bdN1a} reduces to the so-called Wentzell boundary boundary that was introduced in \cite{Ga06, GaW08} for permeable walls.
It easily follows from \eqref{bdN1a} that the total (i.e., bulk plus boundary) mass is conserved
$$
\int_\Omega \phi(t) \,\d x+ \int_\Gamma \phi(t) \,\d S =\int_\Omega \phi(0)\,\d x+\int_\Gamma \phi(0) \,\d S, \quad \forall\, t\in [0,T],
$$
in contrast with what happens under the usual boundary condition \eqref{bdN1}.
As a complement to \eqref{bdN1a}, a different type of variational boundary condition was derived in \cite{GMS11},
by taking derivative of the total free energy in the product space $L^2(\Omega)\times L^2(\Gamma)$ such that
\begin{align}
\mu=-\kappa \Delta_\Gamma \phi +\partial_\mathbf{n}\phi +G'(\phi),\quad \text{on}\ \Gamma\times (0,T).\label{bdN2a}
\end{align}
From \eqref{bdN1a} and \eqref{bdN2a}, we infer that the total energy of the Cahn--Hilliard system \eqref{claCH} still decreases in time
\begin{align}
\frac{\d}{\d t} \Big[E^\text{bulk}(\phi(t))+E^\text{surf}(\phi(t))\Big] + \int_\Omega |\nabla \mu|^2 \,\d x + \sigma \int_\Gamma |\nabla_\Gamma \mu|^2\, \d S=0,\quad \forall\, t\in (0,T).
\label{bbelb}
\end{align}
Comparing with the dynamic boundary condition \eqref{bdN2d} that yields a relaxation dynamics for $\phi$ on the boundary $\Gamma$, \eqref{bdN2a} turns out to be a boundary equation for $\phi$ with a given source term, i.e., the trace of the bulk chemical potential $\mu$ on $\Gamma$. Hence, in some sense \eqref{bdN2a} can also be viewed as a compatibility condition for $\mu$ on the boundary. In the literature, \eqref{bdN1a} (with $\sigma>0$) together with \eqref{bdN2a} is also referred to as a dynamic boundary condition of Cahn--Hilliard type, since it has a very similar structure like the Cahn--Hilliard system \eqref{claCH} in the bulk, after neglecting those two normal derivatives for $\mu$ and $\phi$ (cf. \cite{CF15a} and see also \cite{CGS17} for possible generalizations).

In summary, we see that under both choices of boundary conditions (\eqref{bdN1} with \eqref{bdN2d}, or \eqref{bdN1a} with \eqref{bdN2a}), the Cahn--Hilliard system \eqref{claCH} satisfies two important physical constraints, namely, the mass conservation and energy dissipation.
Among them, \eqref{bdN1}, or respectively \eqref{bdN1a} is proposed to keep suitable mass conservation property in the physical domain, while the so-called variational boundary conditions (i.e., \eqref{bdN2d} or \eqref{bdN2a}) are chosen in a phenomenological way so that the validity of some specific energy dissipation relation is guaranteed (see \eqref{bbela} or \eqref{bbelb}). Thus,  \eqref{bdN2d} or \eqref{bdN2a} can be viewed as sufficient conditions for the energy dissipation of the system. However, such choice may not be unique.

Our first aim in this paper is to introduce a systematic scheme for deriving hydrodynamic boundary conditions for (isothermal) complex systems based on an energetic variational approach,
which combines the least action principle and Onsager's principle of maximum energy dissipation in continuum mechanics \cite{Ray, On31, On31a}.
Within this general framework, one can easily include different physical processes by choosing specific  bulk/surface free energies as well as bulk/surface dissipations for the system.
Then based on suitable kinematic and energetic assumptions, the corresponding partial differential equations in the bulk and their associated boundary conditions will be uniquely determined by force balance relations.
Therefore, one novelty of our approach is that the resulting models naturally fulfill three important physical constraints such as
 (A) \emph{conservation of mass}, (B) \emph{dissipation of energy} and (C) \emph{force balance}. In particular, since the force balance relation arising from the dynamics of energetic and dissipative processes is ensured, the derived differential equations together with their boundary conditions can lead to hydrodynamics that predicts the most probable course of motion \cite{QQS08}. This approach has been successfully applied to derive complex hydrodynamic systems in fluid dynamics, electrokinetics, visco-elasticity, liquid crystals and so on, we refer to \cite{DLRW, FJ13, KLG, LS03, EHL10, HKL10, XSL14, WXL} and the references therein.

In Section 2, we apply the energetic variational approach to derive a new class of dynamic boundary conditions for the Cahn--Hilliard equation in a general setting (see the full system \eqref{full}). It is worth mentioning that there are some subtle issues in our derivation, since we have to calculate the domain variation of action functionals not only in the bulk but also on the boundary. Furthermore, the variation on the boundary should be performed only in the tangential directions and one has to use a Riemannian metric expression, which turns out to be quite involved (cf. \cite{KLG} for the derivation of fluid systems on an evolving surface). As an illustrating example, for system \eqref{claCH} under the same choice of bulk/surface free energies like above, we keep the no-flux boundary condition \eqref{bdN1} and then derive the following dynamic boundary condition
\begin{equation}
\phi_t=\Delta_\Gamma \mu_\Gamma \quad \text{with} \quad \mu_\Gamma =-\kappa \Delta_\Gamma \phi  +\partial_\mathbf{n}\phi + G'(\phi),\quad \text{on}\ \Gamma\times (0,T).
\label{bdN2b}
\end{equation}
In this case, we preserve the mass conservation both in bulk and on the boundary (cf. \eqref{massbulk} and \eqref{massbd})
$$
\int_\Omega \phi(t) \,\d x=\int_\Omega \phi(0)\,\d x \quad \text{and} \quad \int_\Gamma \phi(t) \,\d S=\int_\Gamma \phi(0) \,\d S, \quad \forall\, t\in [0,T],
$$
as well as the energy dissipation property (cf. \eqref{aBEL})
\begin{align}
\frac{\d}{\d t} \Big[E^\text{bulk}(\phi(t))+E^\text{surf}(\phi(t))\Big] + \int_\Omega |\nabla \mu|^2\,\d x +  \int_\Gamma |\nabla_\Gamma \mu_\Gamma|^2\, \d S=0,\quad \forall\, t\in (0,T).
\label{bbelc}
\end{align}
The new dynamic boundary condition \eqref{bdN2b} turns out to be a Cahn--Hilliard type equation for $\phi$ defined on $\Gamma$ (provided that $\kappa>0$), comparing with the previous Allen--Cahn type like \eqref{bdN2d}. On the other hand, it is also different from \eqref{bdN1a} with \eqref{bdN2a} in the sense that here we impose \eqref{bdN1} for the bulk chemical potential $\mu$ instead of \eqref{bdN1a} and the ``boundary chemical potential" $\mu_\Gamma$ defined in \eqref{bdN2b} is no longer necessary to coincide with the trace of $\mu$ on $\Gamma$ as indicated by \eqref{bdN2a} (cf. \cite{GMS11,CF15a}). Moreover, although the energy dissipation relation \eqref{bbelc} looks very similar to \eqref{bbelb}, it actually follows from a rather different mechanism. Besides these differences, one important and interesting issue is that, as we shall see in Section 2, \eqref{bdN2b} can be uniquely determined through a force balance relation on the boundary. This physical property is not explicitly fulfilled in all the other dynamical boundary conditions mentioned above.

In the second part of this paper, we focus on the mathematical analysis for an initial boundary value problem of the Cahn--Hilliard system \eqref{claCH} subject to the boundary conditions \eqref{bdN1} and \eqref{bdN2b}:
\begin{equation}
\left\{
\begin{aligned}
&\phi_t=\Delta \mu, &\text{in}\ \Omega \times (0,T),\\
&\mu=-\Delta \phi +F'(\phi), &\text{in}\ \Omega \times (0,T),\\
&\partial_\mathbf{n}\mu=0, &\text{on}\ \Gamma\times (0,T),\\
&\phi_t=\Delta_\Gamma \mu_\Gamma,&\text{on}\ \Gamma\times (0,T),\\
&\mu_\Gamma =-\kappa \Delta_\Gamma \phi  +\partial_\mathbf{n}\phi + G'(\phi), &\text{on}\ \Gamma\times (0,T),\\
&\phi|_{t=0}=\phi_0(x),&\text{in}\ \Omega.
\end{aligned}
\right.
\label{newCH}
\end{equation}
More precisely, we provide a comprehensive treatment of problem \eqref{newCH} by proving
\begin{itemize}
\item[(i)] existence, uniqueness and regularity of global weak/strong solutions under suitable assumptions on the nonlinear potentials $F$ and $G$, for both physically relevant cases with or without surface diffusion (see Theorems \ref{main1}, \ref{main1a});
\item[(ii)] long-time behavior of the system, i.e., the uniqueness of asymptotic limit as $t\to+\infty$ for any bounded weak/strong solution (see Theorem \ref{main2}), and a characterization for the stability of local energy minimizers (see Theorem \ref{main3}).
\end{itemize}

The Cahn--Hilliard equation has attracted noteworthy interests of mathematicians for long time (see \cite{CMZ11, Mi2017} for recent reviews on this subject).
For instance, the system \eqref{claCH} with standard boundary conditions \eqref{bdN1}--\eqref{bdN2} has been well-studied from various viewpoints in the literature since the 1980s, see \cite{AW07, BF93, Chen, EZ86, GGM17, GN95, HR99, pego, CMZ11, Z86} and the references cited therein.
We also refer to \cite{BE95, ES96, NC88} for results on the so-called viscous Cahn--Hilliard equation, in which there is an additional regularizing term $\alpha \phi_t\ (\alpha>0)$ in the bulk chemical potential such that $\mu=-\Delta \phi+F'(\phi)+\alpha \phi_t$. As far as the dynamic boundary condition \eqref{bdN2d} is concerned, the analysis turns out to be more involved. The first well-posedness result (with $\kappa>0$) was proven in \cite{RZ03} by employing a suitable approximating phase-field equations of Caginalp type (a similar argument was used in \cite{CWX14} for the case $\kappa=0$ related to the problem of moving contact lines, neglecting the fluid interaction). Since then rigorous mathematical analysis for the Cahn--Hilliard system and its variants subject to the same dynamic boundary condition \eqref{bdN2d} have been performed. We refer to \cite{CGW14, CF15, CGS14, PRZ06, Ga07, MH15} for well-posedness results via different approaches, to \cite{CGG, Ga08, GM1, GMS09, MZ05, MZ10} for regularity properties and long-time behavior in terms of global or exponential attractors, and to \cite{CGW14, WZ04, CFJ06, GM2, GMS} for global asymptotic stability of single trajectories as time goes to infinity.
Next, for the Cahn--Hilliard system subject to the second type of dynamic boundary conditions \eqref{bdN1a}--\eqref{bdN2a}, well-posedness and long-time behavior was first investigated in \cite{GMS11} (see also \cite{CGM13}). Later in \cite{CF15a}, well-posedness of the same system with more general potentials was analyzed and a unified treatment was provided recently in \cite{CGS17} for both viscous and pure Cahn--Hilliard equations. We also refer to \cite{Ga06,GaW08,WH07} for analysis results on the special case with $\sigma=0$ in the boundary condition \eqref{bdN1a} (indeed, the fact $\sigma=0$ yields a weaker dissipative mechanism of the system, see \eqref{bbelb}). Quite recently, in \cite{NK17} the author obtained the strong well-posedness in maximal $L^p$-regularity spaces.

The new feature of our problem \eqref{newCH} is that the dynamic boundary condition \eqref{bdN2b} (with $\kappa>0$) turns out to be a surface Cahn--Hilliard type equation for the trace of $\phi$ on $\Gamma$ (cf. \cite{RV06}), which is further coupled with the bulk evolution in terms of the normal derivative $\partial_\mathbf{n} \phi$. When $\kappa=0$, the situation seems even worse since \eqref{bdN2b} becomes an evolution equation on $\Gamma$ that is not of standard parabolic type and may be ill-posed, with the highest spatial order of three. The linearized system of problem \eqref{newCH} does not satisfy the abstract framework discussed in \cite{DPZ} (e.g., the Lopatinskii--Shapiro condition therein), in which a general $L^p$-theory for parabolic problems with boundary dynamics of relaxation type was developed. Thus, it is not clear whether the maximal $L^p$-regularity results as in \cite{PRZ06,NK17} can be extended to our current problem. On the other hand, we recall that for the Cahn--Hilliard system \eqref{claCH} subject to boundary conditions \eqref{bdN1} and \eqref{bdN2d}, by exploiting the weak formulation of the system, a Faedo--Galerkin discretization scheme was introduced in \cite{GMS09} (indeed for more general cases with irregular potentials) to prove the existence of global weak solutions (cf. \cite{Ga08} for a slightly different approximation). Later, in \cite{GMS11} a similar idea was used to deal with the dynamic boundary conditions \eqref{bdN1a}--\eqref{bdN2a}. Unfortunately, it seems that a suitable Galerkin approximating scheme is not available for the current problem \eqref{newCH} due to its different variational structure.

Inspired by a third approach introduced in \cite{MZ05} for the dynamic boundary condition \eqref{bdN2d}, we treat the boundary condition \eqref{bdN2b} as a separate evolution equation on $\Gamma$  and try to solve our problem \eqref{newCH} as a bulk/surface coupled system. Our strategy is as follows. We first regularize the original problem \eqref{newCH} with surface diffusion (i.e., $\kappa>0$) by adding viscous terms $\alpha \phi_t$ ($\alpha> 0$) in both of the bulk and surface chemical potentials $\mu$ and $\mu_\Gamma$. This leads to a viscous Cahn--Hilliard equation (cf. \cite{NC88, BE95}) subject to a dynamic boundary condition of viscous Cahn--Hilliard type (see \eqref{ACH}). The solvability of its corresponding linearized system (see \eqref{Lpa}) can be obtained by transforming this fourth-order problem into second-order parabolic equations (see Lemma \ref{exLpa}). Then one can prove local well-posedness of the regularized system \eqref{ACH} by using the contraction mapping theorem (see Proposition \ref{exeACH}). Afterwards, the key step is to obtain uniform global-in-time a priori estimates that are independent of the parameter $\alpha$, which allow us to pass to the limit as $\alpha\to 0^+$ to obtain the global existence of weak (or strong) solutions of the original problem \eqref{newCH} with $\kappa>0$. The uniqueness of solutions can be proved by a standard energy method. In order to prove well-posedness results for the case without surface diffusion, i.e., $\kappa=0$, we view $\kappa\Delta_\Gamma \phi$ as a regularizing term in the surface equation \eqref{bdN2b}. The conclusion can be drawn by taking the limit as $\kappa \to 0^+$ provided that proper uniform estimates independent of $\kappa$ are available, which turn out to be more delicate.
Next, we study the global regularity of solutions and prove for arbitrary large initial datum, every global bounded weak/strong solution to problem \eqref{newCH} will converge to a single steady state as $t\to +\infty$ and an estimate on the convergence rate is also given. This issue is nontrivial, since both the bulk/surface free energies of the Cahn--Hilliard system are in general non-convex and thus the set of solutions to the stationary problem (see \eqref{sta}--\eqref{staL}) may have a quite complicated structure. The goal is achieved by employing the \L ojasiewicz--Simon approach (see \cite{S83}, and \cite{AW07, GA16, Ch03, CFJ06, FS00, GaW08, GGW18, GMS, HR99, HJ01, HT01, WH07, WZ04,ZWH} for its various extensions). At last, by applying the \L ojasiewicz--Simon approach in a different way, we further give a characterization on the Lyapunov stability of steady states of the system (in particular, the local energy minimizers) that are allowed to be non-isolated.

 The rest of this paper is organized as follows.
 In Section 2, we derive a new class of dynamic boundary conditions for the Cahn--Hilliard equation by using the energetic variational approach.
 Next, in Section 3, we introduce the function spaces, problem settings and state the main results on mathematical analysis of problem \eqref{newCH}.
 Section 4 is devoted to prove the well-posedness of a regularized Cahn--Hilliard system with viscous terms.
 Then in Section 5, we first derive uniform a priori estimates that are independent of the regularizing parameter $\alpha$ as well as the surface diffusion coefficient $\kappa$.
 Then we give the proof for our main results Theorems \ref{main1}, \ref{main1a} on global well-posedness of the original problem \eqref{newCH}. The convergence of global solutions to a single equilibrium as $t\to +\infty$ (Theorem \ref{main2}) and the stability criterion (Theorem \ref{main3}) are proved in Section 6.
 In the Appendix, we provide detailed calculations for the model derivation that are omitted in Section 2 and prove a fundamental result on the well-posedness of an auxiliary linear fourth-order parabolic equation subject to a fourth-order dynamic boundary condition.

\section{Model Derivation}
\setcounter{equation}{0}

In this section, we derive a general class of Cahn--Hilliard type equations subject to dynamic boundary conditions that fulfill several important physical properties:
\begin{itemize}
\item[(A)] Kinematics: conservation of mass both in the bulk $\Omega$ and on the boundary $\Gamma$;
\item[(B)] Energetics: dissipation of the total free energy;
\item[(C)] Force balance: both in the bulk $\Omega$ and on the boundary $\Gamma$.
\end{itemize}
Hereafter, the (smooth) boundary $\Gamma$ is assumed to be either a closed surface ($d=3$) or a closed curve ($d=2$). Let $\phi$ be the state variable, which for instance, stands for the difference of relative concentrations for the components of binary mixtures \cite{CH58} or a labeling function presenting the smooth transition between different phases \cite{AN97, LS03}.

\subsection{The energetic variational approach}

\textbf{(A) Mass conservation.} In the bulk $\Omega$, $\phi$ is assumed to be a locally
conserved quantity that satisfies the continuity equation
\begin{equation}
\phi_t+\nabla\cdot(\phi \mathbf{u})=0,
\quad (x, t)\in \Omega\times (0,T),
\label{con1}
\end{equation}
where $\mathbf{u}: \Omega \to \mathbb{R}^d$ stands for the microscopic effective velocity (e.g., due to the diffusion process).
We assume that $\mathbf{u}$  satisfies the no-flux boundary condition
\begin{equation}
\mathbf{u}\cdot \mathbf{n}=0,
\quad (x, t)\in \Gamma\times (0,T).
\label{bd1}
\end{equation}
The boundary condition \eqref{bd1} implies the simple fact that there is no mass exchange between the bulk $\Omega$ and the boundary $\Gamma$. In the remaining part of this paper, we shall always confine ourselves to this special case for the sake of simplicity. Next, we consider the possible evolution on the boundary $\Gamma$. Instead of taking the usual $L^2$-gradient flow for certain surface energy (cf. \eqref{bdN2d}), we assume that the boundary dynamics is characterized by a local mass conservation law analogous to \eqref{con1} such that (see, e.g., \cite{DE07})
\begin{equation}
\phi_t+\nabla_{\Gamma}\cdot (\phi \mathbf{v})=0,
\quad (x, t)\in \Gamma\times (0,T),
\label{con2}
\end{equation}
where $\mathbf{v}: \Gamma \to \mathbb{R}^d$ denotes the microscopic effective tangential velocity field on the boundary $\Gamma$.
We note that there is no need to impose any boundary condition on $\mathbf{v}$, since $\Gamma$ is assumed to be a closed manifold.

\begin{remark}
\label{Rem1}
Within this section, we regard the phase function $\phi$ to be regular enough (e.g., $\phi\in C^m(\overline{\Omega})$ for sufficiently large $m\in \mathbb{N}$).
Then by the continuity of $\phi$ to the boundary, the surface equation \eqref{con2} should be understood for its trace denoted by $\phi|_\Gamma$, namely,
\begin{equation}
\partial_t \phi|_\Gamma+\nabla_{\Gamma}\cdot (\phi|_\Gamma \mathbf{v})=0,
\quad (x, t)\in \Gamma\times (0,T).
\label{con2a}
\end{equation}
  Without ambiguity, on the boundary $\Gamma$ we simply use the surface equation in the form of \eqref{con2}.
  Nevertheless, the formulation \eqref{con2a} will turn out to be helpful when we perform mathematical analysis on the related partial differential equations.
\end{remark}

It easily follows from the kinematic relations \eqref{con1}--\eqref{con2} that the mass is conserved in the bulk $\Omega$ and on the boundary $\Gamma$, respectively.
 To this end, integrating \eqref{con1} over $\Omega$ and using the no-flux boundary condition \eqref{bd1}, we have
\begin{equation}
\frac{\d}{\d t}\int_\Omega \phi(\cdot,t)\, \d x =0,
\quad \forall\, t\in(0,T),
\label{massbulk}
\end{equation}
while integrating \eqref{con2} over $\Gamma$ gives
\begin{equation}
\frac{\d}{\d t}\int_{\Gamma} \phi(\cdot,t)\, \d S =0,
\quad \forall\, t\in(0,T).
\label{massbd}
\end{equation}

\medskip

\noindent \textbf{(B) Energy dissipation.}
We are interested in the case that on the boundary there is certain non-trivial dynamics induced by a surface potential, e.g., the short-range interactions between the material and the wall \cite{K-etal, Q1, TR89}. This can also be formally justified by the surface layer scaling approach within a thin layer next to the solid wall (cf. \cite{QQS08}).

For an isothermal closed system, evolution of the  binary mixtures is characterized by the following energy dissipation law (cf. \cite{HKL10})
\begin{align}
 \frac{\d}{\d t}E^{\text{total}}(t)=-\mathcal{D}^{\text{total}}(t),
 \label{ABEL}
\end{align}
which is a consequence of the first and second laws of thermodynamics.
$E^{\text{total}}$ is the total Helmholtz free energy of the system.
For the sake of simplicity, macroscopic kinetic energy of the binary mixture is neglected throughout the paper.
Then we assume that
\begin{equation}
E^{\text{total}}(t)= E^{\text{bulk}}(t)+E^{\text{surf}}(t),
\label{fenergy}
\end{equation}
which is the sum of free energies in the bulk $\Omega$ and on the boundary $\Gamma$:
\begin{equation}
E^{\text{bulk}}(t)
=\int_\Omega W_\text{b}(\phi, \nabla \phi)\, \d x,
\quad E^{\text{surf}}(t)=\int_\Gamma W_\text{s}(\phi, \nabla_\Gamma \phi)\, \d S.
\label{bsenergy}
\end{equation}
The energy density functions $W_\text{b}$ and $W_\text{s}$ are scalar functions that can take different forms under various physical considerations.
Next, the rate of energy dissipation $\mathcal{D}^{\text{total}}$ (related to the entropy production in thermodynamics) is chosen as
\begin{align}
\mathcal{D}^{\text{total}}(t)
=\mathcal{D}^{\text{bulk}}(t)+\mathcal{D}^{\text{surf}}(t),
\label{diss}
\end{align}
which also consists of contributions from the bulk $\Omega$ and the boundary $\Gamma$. Here, we assume that
\begin{align}
\mathcal{D}^{\text{bulk}}(t)= \int_\Omega \phi^2 (\mathbb{M}_\text{b}^{-1}\mathbf{u})\cdot \mathbf{u} \, \d x,
\quad \mathcal{D}^{\text{surf}}(t)= \int_{\Gamma} \phi^2 (\mathbb{M}_\text{s}^{-1} \mathbf{v})\cdot \mathbf{v} \, \d S,
\label{diss1}
\end{align}
where $\mathbb{M}_\text{b}$, $\mathbb{M}_\text{s}$ are $d\times d$ mobility matrices that are assumed to be symmetric and positive definite.
Their entries may depend on $x$, $t$ as well as $\phi$.

In summary, the basic energy law we are going to consider reads as follows:
\begin{align}
&\frac{\d}{\d t}\left(\int_\Omega W_\text{b}(\phi, \nabla \phi)\, \d x
+\int_{\Gamma} W_\text{s}(\phi, \nabla_{\Gamma}\phi)\, \d S\right)\nonumber\\
&\quad = -\left(\int_\Omega \phi^2 (\mathbb{M}_\text{b}^{-1}\mathbf{u})\cdot \mathbf{u} \, \d x
+\int_{\Gamma} \phi^2 (\mathbb{M}_\text{s}^{-1} \mathbf{v})\cdot \mathbf{v} \, \d S\right).
\label{BEL}
\end{align}

\medskip

\noindent \textbf{(C) Force balance.} In order to derive a closed system of partial differential equations, it remains to determine the microscopic velocities $\mathbf{u}$, $\mathbf{v}$ in equations \eqref{con1} and \eqref{con2}. The derivation will be carried out through a unified energetic variational approach \cite{HKL10} motivated by the seminal work \cite{Ray, On31, On31a} that ensures certain force balance relations from the dynamics of conservative and dissipative processes, based on the energy dissipation relation \eqref{BEL}.

To apply the variational principle, we need to calculate the variation of the action integral with respect to the flow maps, as well as the variation of the dissipation functional with respect to the velocities. To this end, let $\Omega_0^X, \Omega_t^x\subset\mathbb{R}^d$ be bounded domains with smooth boundaries $\Gamma_0^X, \Gamma_t^x$, respectively.
 Then we introduce the (bulk) \emph{flow map}  $x(X,t): \Omega_0^X\to \Omega_t^x$, which is defined as a solution to the system of ordinary differential equations
\begin{equation}
\begin{cases}
\displaystyle{\frac{\d}{\d t}}x(X,t)=\mathbf{w}(x(X,t),t),\quad t>0,\\
x(X,0)=X,
\end{cases}
\label{flowmap}
\end{equation}
where $X=(X_1,...,X_d)^T\in \Omega_0^X$, $x=(x_1,...,x_d)^T\in \Omega_t^x$, and  $\mathbf{w}(x,t)\in \mathbb{R}^d$ is a (smooth) velocity field.
The coordinate system $X$ is called the Lagrangian coordinate system and refers
to $\Omega_0^X$ that we call the \emph{reference configuration}; the coordinate system
$x$ is called the Eulerian coordinate system and refers to $\Omega_x^t$
that we call the \emph{deformed configuration}. In a similar manner, we can introduce the \emph{surface flow map} $x_\text{s}=x_\text{s}(X_\text{s},t): \Gamma_0^X\to \Gamma_t^x$ (cf. \cite[Definition 2.9]{KLG}).
Within this section, we shall denote $\nabla_x$ and $\nabla_\Gamma^x$ the gradient in $\Omega$ and the tangential gradient on $\Gamma$ under the Eulerian coordinate system, respectively. In the Lagrangian coordinate system, we shall use the notations $\nabla_X$ and $\nabla^X_\Gamma$ correspondingly.

The Lagrangian framework of continuum mechanics writes
energies of the system in terms of the motions $x(X,t)$, $x_\text{s}(X_\text{s},t)$ using action functionals.
To this end, we introduce the Lagrangians in the bulk $\Omega$ and on the boundary $\Gamma$ that are associated with the bulk/surface free energies \eqref{bsenergy}, respectively:
\begin{align}
L^{\text{bulk}}(x(t))=-\int_{\Omega_t^x} W_\text{b}(\phi, \nabla_x \phi)\, \d x,\quad L^{\text{surf}}(x_\text{s}(t))=-\int_{\Gamma_t^x} W_\text{s}(\phi, \nabla_\Gamma^x \phi)\, \d S_x.\nonumber
\end{align}
Then the corresponding action functionals are given by
\begin{align}
\mathcal{A}^{\text{bulk}}(x(X,t))  = \int_0^T  L^{\text{bulk}}(x(t))\,\d t,
\quad \mathcal{A}^{\text{surf}}(x_\text{s}(X,t))  = \int_0^T L^{\text{surf}}(x_\text{s}(t))\,\d t.
\label{ac}
\end{align}
For the total action functional $\mathcal{A}^{\text{total}}=\mathcal{A}^{\text{bulk}}+\mathcal{A}^{\text{surf}}$, taking the variation with respect to the spatial variables $x$ in $\Omega$ and $x_\text{s}$ on $\Gamma$, respectively, we deduce that (see Appendix for detailed calculations)
\begin{align}
\delta_{(x, x_\text{s})} \mathcal{A}^{\text{total}}
&  = -\int_0^T\int_{\Omega_t^x} (\phi \nabla_x \mu_\text{b}) \cdot \tilde{y}\, \d x \d t
  \nonumber\\
&\quad -\int_0^T\int_{\Gamma_t^x} \left[\phi \nabla_\Gamma^x \left(\mu_\text{s}+ \frac{\partial W_\text{b}}{\partial \nabla_x \phi}\cdot \mathbf{n}\right)\right] \cdot \tilde{y}_\text{s}\, \d S_x\d t,
\label{LAP}
\end{align}
for arbitrary smooth vectors $\tilde{y}(x,t)$ ($x\in \Omega_t^x$) and $\tilde{y}_\text{s}(x_\text{s},t)$ ($x_\text{s}\in \Gamma_t^x$) satisfying $\tilde{y}\cdot \mathbf{n}=\tilde{y}_\text{s}\cdot \mathbf{n}=0$ on $\Gamma_t^x$.
The notions $\mu_\text{b}$ and $\mu_\text{s}$ in \eqref{LAP} stand for the chemical potentials in the bulk and on the boundary such that
\begin{align}
&\mu_\text{b}=\frac{\delta W_\text{b}(\phi, \nabla_x\phi)}{\delta \phi}
=-\nabla_x\cdot \frac{\partial W_\text{b}}{\partial \nabla_x \phi}
+\frac{\partial W_\text{b}}{\partial\phi},
\label{bumu}\\
&\mu_\text{s} =\frac{\delta W_\text{s}(\phi, \nabla_\Gamma^x\phi)}{\delta \phi}
=-\nabla_\Gamma^x\cdot \frac{\partial W_\text{s}}{\partial \nabla_\Gamma^x \phi}
+\frac{\partial W_\text{s}}{\partial\phi}.
\label{bdmu}
\end{align}

The principle of least action (LAP) optimizes the action functional $\mathcal{A}$ with respect to all trajectories $x(X,t)$ by setting the variation
with respect to $x$ to zero, i.e., $\delta_x \mathcal{A}=0$. From this, one can obtain the conservative force balance equation
of classical Hamiltonian mechanics, i.e., the conservation of momentum.
In other words, the LAP gives the Hamiltonian part of a mechanical system that corresponds to its conservative
forces, formally speaking (see e.g., \cite[Section 2.2.1]{SV}),
$$
\delta_x L(x,x_t)=(F_{\text{inertial}}+F_{\text{conv}})\cdot \delta x.
$$
Since the macroscopic kinetic energy is neglected in our current system, then the inertial forces in the bulk $\Omega$ and on the boundary $\Gamma$ simply vanish, i.e., $F_{\text{inertial}}=0$.
Thus in view of \eqref{LAP}--\eqref{bdmu}, we deduce the generalized conservative
forces in the bulk and on the boundary for $t\in (0,T)$ (written in the strong form):
\begin{align}
F_{\text{conv}}^{\text{bulk}}= -\phi \nabla_x \mu_\text{b},
\quad
F_{\text{conv}}^{\text{surf}}= -\phi \nabla_\Gamma^x \left(\mu_\text{s}+ \frac{\partial W_\text{b}}{\partial \nabla_x \phi}\cdot \mathbf{n}\right).
\label{FOCO}
\end{align}

On the other hand, energy dissipation functionals will be treated by extending the classical
treatment of the Hamiltonian to include dissipative forces, through Onsager's maximum dissipation principle (MDP) that describes the irreversible dissipative processes in the regime of linear response.
This can be done by taking variation of the Rayleigh dissipation functional $\mathcal{R}=\frac12\mathcal{D}$ with respect to the rate functions, for instance, the velocity $x_t$
(see e.g., \cite[Section 2.2.1]{SV}): $$
\delta_{x_t}\mathcal{R}=-F_{\text{diss}}\cdot \delta x_t.
$$
We note that the factor $\frac12$ is used since the energy dissipation $\mathcal{D}$ is usually chosen to be quadratic in ``rates" within the linear response theory for long-time near equilibrium dynamics (cf. \cite{Ray}).
If $\delta_{x_t} \mathcal{R}$ is set to zero, the resulting equation gives a weak variational form
of the dissipative force balance law that is again equivalent to the conservation of momentum.
Recalling the energy dissipation functions introduced in \eqref{diss1}, we have for $t\in (0,T)$
\begin{align}
\delta_{(\mathbf{u},\mathbf{v})}\left(\frac12 \mathcal{D}^{\text{total}} \right)
&= \frac12\delta_{\mathbf{u}} \mathcal{D}^{\text{bulk}}+\frac12 \delta_{\mathbf{v}} \mathcal{D}^{\text{surf}}\non\\
&= \frac12\left.\frac{\d}{\d \epsilon} \right|_{\epsilon=0}  \mathcal{D}^{\text{bulk}}(\mathbf{u}+\epsilon \tilde{\mathbf{u}})
+\frac12\left.\frac{\d}{\d \epsilon} \right|_{\epsilon=0}  \mathcal{D}^{\text{surf}}(\mathbf{v}+\epsilon \tilde{\mathbf{v}})\non\\
&= \frac12 \left.\frac{\d}{\d \epsilon} \right|_{\epsilon=0} \int_{\Omega_t^x} \phi^2\big[\mathbb{M}_\text{b}^{-1}(\mathbf{u}+\epsilon \tilde{\mathbf{u}})\big] \cdot (\mathbf{u}+\epsilon \tilde{\mathbf{u}})\,\d x\nonumber\\
&\quad +\frac12 \left.\frac{\d}{\d\epsilon} \right|_{\epsilon=0} \int_{\Gamma_t^x} \phi^2 \big[\mathbb{M}_\text{s}^{-1}(\mathbf{v}+\epsilon \tilde{\mathbf{v}})\big]\cdot (\mathbf{v}+\epsilon \tilde{\mathbf{v}})\,\d S_x\non\\
&= \int_{\Omega_t^x} \phi^2 (\mathbb{M}_\text{b}^{-1}\mathbf{u})\cdot \tilde{\mathbf{u}}\, \d x
   +\int_{\Gamma_t^x} \phi^2 (\mathbb{M}_\text{s}^{-1}\mathbf{v})\cdot \tilde{\mathbf{v}}\, \d S_x,\non
\end{align}
where $\tilde{\mathbf{u}}: \Omega_t^x\to \mathbb{R}^d$, $\tilde{\mathbf{v}}: \Gamma_t^x\to \mathbb{R}^d$ are arbitrary smooth vectorial functions satisfying $\tilde{\mathbf{u}}\cdot \mathbf{n}=\tilde{\mathbf{v}}\cdot \mathbf{n}=0$ on $\Gamma_t^x$.
Thus, we can derive the generalized dissipative forces both in the bulk $\Omega$ and on the boundary $\Gamma$ (again written in the strong form):
\begin{align}
F_{\text{diss}}^{\text{bulk}}= -\phi^2 (\mathbb{M}_\text{b}^{-1}\mathbf{u}),\quad F_{\text{diss}}^{\text{surf}}=-\phi^2 (\mathbb{M}_\text{s}^{-1}\mathbf{v}).
\label{FODI}
\end{align}

From \eqref{FOCO}, \eqref{FODI} and the classical Newton's force balance law $F_{\text{inertial}}+F_{\text{conv}}+F_{\text{diss}}=0$ (recalling that here we have $F_{\text{inertial}}=0$), we deduce the following force balance relations in the bulk $\Omega$ and on the boundary $\Gamma$ that further yield an exact expression of the microscopic velocities $\mathbf{u}$ and $\mathbf{v}$:
\begin{equation}
\left\{
\begin{aligned}
&\phi \nabla_x \mu_\text{b} +\phi^2 (\mathbb{M}_\text{b}^{-1}\mathbf{u})=0, &\text{in}\ \Omega_t^x,\\
&\phi \nabla_\Gamma^x \left(\mu_\text{s}+ \displaystyle{\frac{\partial W_\text{b}}{\partial \nabla_x \phi}}\cdot \mathbf{n}\right)
+\phi^2 (\mathbb{M}_\text{s}^{-1}\mathbf{v})=0, &\text{on}\ \Gamma_t^x.
\end{aligned}
\right.
\label{fba}
\end{equation}

\subsection{The Cahn--Hilliard equation with new dynamic boundary conditions}
In summary, we can deduce from the above relations \eqref{con1}, \eqref{bd1}, \eqref{con2}, \eqref{bumu}, \eqref{bdmu} and \eqref{fba} the following general Cahn--Hilliard system subject to a new type of dynamic boundary condition (an initial condition for the phase function $\phi$ is also implemented):
\begin{equation}
\left\{
\begin{aligned}
&\phi_t=\nabla \cdot(\mathbb{M}_\text{b} \nabla \mu_\text{b}), &\text{in}\ \Omega\times(0,T),\\
&\displaystyle{\mu_\text{b}=-\nabla\cdot \frac{\partial W_\text{b}}{\partial \nabla \phi}+\frac{\partial W_\text{b}}{\partial\phi}},&\text{in}\ \Omega\times(0,T),\\
&\displaystyle{\partial_\mathbf{n}\mu_\text{b} =0},&\text{on}\ \Gamma\times(0,T),\\
&\displaystyle{\phi_t=\nabla_\Gamma \cdot\left[\mathbb{M}_\text{s} \nabla_\Gamma \left(\mu_\text{s}+\frac{\partial W_\text{b}}{\partial \nabla \phi}\cdot \mathbf{n}\right)\right]},&\text{on}\ \Gamma\times(0,T),\\
&\displaystyle{\mu_\text{s}=-\nabla_\Gamma\cdot \frac{\partial W_\text{s}}{\partial \nabla_\Gamma \phi}+\frac{\partial W_\text{s}}{\partial \phi}},&\text{on}\ \Gamma\times(0,T),\\
&\phi|_{t=0}=\phi_0(x),&\text{in}\ \Omega,
\end{aligned}
\right.
\label{full}
\end{equation}
where $T\in (0,+\infty)$. Here, all the derivatives are taken in the Eulerian coordinates and thus the subscript (or superscript) $x$ is simply omitted.

\begin{remark}
(1) It is straightforward to verify that any smooth solution to the Cahn--Hilliard system \eqref{full} satisfies the energy dissipation law \eqref{BEL}.
To see this, multiplying the first equation of \eqref{full} by $\mu_\mathrm{b}$, integrating over $\Omega$ by using the no-flux boundary condition, then multiplying the dynamic boundary condition in \eqref{full} by $\mu_\mathrm{s}+\frac{\partial W_\mathrm{b}}{\partial \nabla \phi}\cdot \mathbf{n}$ and integrating over $\Gamma$, adding the resultants together, we arrive at
\begin{align}
&\frac{\d}{\d t}\left(\int_\Omega W_\mathrm{b}(\phi, \nabla \phi)\, \d x
+\int_{\Gamma} W_\mathrm{s}(\phi, \nabla_\Gamma \phi)\, \d S\right) \nonumber\\
&\quad = -\int_\Omega (\mathbb{M}_\mathrm{b} \nabla \mu_\mathrm{b})\cdot \nabla \mu_\mathrm{b} \, \d x\nonumber\\
&\qquad -\int_{\Gamma}  \left[\mathbb{M}_\mathrm{s} \nabla_\Gamma \left(\mu_\mathrm{s}+\frac{\partial W_\mathrm{b}}{\partial \nabla \phi}\cdot \mathbf{n}\right)\right]\cdot \nabla_\Gamma \left(\mu_\mathrm{s}+\frac{\partial W_\mathrm{b}}{\partial \nabla \phi}\cdot \mathbf{n}\right)\, \d S,
\label{aBEL}
\end{align}
which coincides with \eqref{BEL} in view of the force relation \eqref{fba}.
The energy dissipation property \eqref{aBEL} serves as a starting point for mathematical analysis of the Cahn--Hilliard system \eqref{full}.

(2) By taking time derivative of the total energy directly, we see that the dynamic boundary condition in \eqref{full} can also be viewed as a sufficient condition for the decreasing of energy.
Nevertheless, our model derivation reveals further physical relations behind this consideration on energy dissipation, i.e., the mass conservation and in particular, the force balance.
\end{remark}


\section{Mathematical Analysis}
\setcounter{equation}{0}

\subsection{The initial boundary value problem}

Different physical considerations can be naturally incorporated into the general Cahn--Hilliard system \eqref{full}
by choosing free energy densities $W_\text{b}, W_\text{s}$ as well as mobilities $\mathbb{M}_\text{b}, \mathbb{M}_\text{s}$.
As a first step towards the corresponding mathematical analysis, in the remaining part of this paper, we focus on (possibly) the simplest case with regular potentials and positive constant mobilities. To this end, we assume that
$$
\mathbb{M}_\text{b}=\mathbb{M}_\text{s}=\mathbb{I}_d,
$$
where $\mathbb{I}_d$ is the $d\times d$ identity matrix, while the bulk/surface energy densities take the following specific forms:
\begin{align}
W_\text{b}(\phi, \nabla \phi)=\frac12|\nabla \phi|^2 + F(\phi),
\quad W_\text{s}(\phi, \nabla_\Gamma\phi)=\frac{\kappa}{2}|\nabla_\Gamma \phi|^2+\frac{1}{2}\phi^2+G(\phi),
\label{WbWs}
\end{align}
where $\kappa\geq 0$ is a nonnegative constant, and $F$, $G$ are bulk/surface potential functions that satisfy proper assumptions (see below).
We take the above particular form of $W_\text{s}(\phi, \nabla_\Gamma\phi)$ that contains a quadratic term $\frac{1}{2}\phi^2$ just for convenience in the subsequent analysis.

Inspired by \cite{MZ05}, it will be more convenient to view the trace of the order parameter $\phi$ as an unknown function on the boundary $\Gamma$ (cf. Remark \ref{Rem1}). Thus, after introducing the new variable
$$
\psi:=\phi|_\Gamma,
$$
 the initial boundary value problem that we are going to analyze (see \eqref{newCH}) can be written in the following form:
\begin{equation}
\left\{
\begin{aligned}
&\phi_t=\Delta \mu,\quad \text{with}\ \ \mu=-\Delta \phi +F'(\phi),
&\text{in}\  \Omega\times (0,T),\\
&\partial_\mathbf{n} \mu=0,
&\text{on}\ \Gamma\times (0,T),\\
&\phi|_\Gamma=\psi,
&\text{on}\ \Gamma\times (0,T),\\
&\psi_t=\Delta_\Gamma \mu_\Gamma,\quad\! \text{with}\ \ \mu_\Gamma= -\kappa \Delta_\Gamma \psi +  \psi +\partial_\mathbf{n}\phi + G'(\psi),
&\text{on}\ \Gamma\times (0,T),\\
&\phi|_{t=0}=\phi_0(x), &\text{in}\ \Omega,\\
&\psi|_{t=0}=\psi_0(x):=\phi_0(x)|_{\Gamma},&\text{on}\ \Gamma.
\end{aligned}
\right.
\label{CH}
\end{equation}
Problem \eqref{CH} can be viewed as a Cahn--Hilliard equation for $\phi$ in $\Omega$ subject to a homogeneous Neumann boundary condition for $\mu$ together with a nonhomogeneous Dirichlet boundary condition for $\phi$, coupled with an evolution equation for $\psi$ on the boundary $\Gamma$. The bulk-surface coupling structure is in terms of the Dirichlet boundary condition for the order parameter $\phi$ and its outer normal derivative on the boundary.

Next, we state some basic assumptions on the nonlinearities $F$, $G$.
\begin{itemize}
\item[(\textbf{A1})] $F$, $G\in C^4(\mathbb{R})$.
\end{itemize}
\begin{itemize}
\item[(\textbf{A2})] There exist nonnegative constants $C_F$, $\widetilde{C}_F$, $C_G$, $\widetilde{C}_G\geq  0$ independent of $y\in \mathbb{R}$ such that
  \begin{align*}
  F(y)\geq -C_F,\quad F''(y)\geq -\widetilde{C}_F,\quad G(y)\geq -C_G, \quad G''(y)\geq -\widetilde{C}_G,\quad \forall\, y\in \mathbb{R}.
  \end{align*}
\end{itemize}
\begin{remark}
Assumption (\textbf{A2}) can be regarded as a dissipative condition that guarantees the existence of global weak/strong solutions.
It implies that both $F$ and $G$ are quadratic perturbations of some strictly convex functions $\widetilde{F}$, $\widetilde{G}$, for instance, we can put
\begin{align}
&\widetilde{F}(y)=F(y)+\frac{\widetilde{C}_F+1}{2}y^2-F'(0)y-F(0),\quad \forall\, y\in \mathbb{R},
\label{tF}\\
&\widetilde{G}(y)=G(y)+\frac{\widetilde{C}_G+1}{2}y^2-G'(0)y-G(0),\quad \forall\, y\in \mathbb{R}.
\label{tG}
\end{align}
\end{remark}

When global weak solutions are concerned, we impose the following (subcritical) growth condition:
\begin{itemize}
\item[(\textbf{A3})] there exist positive constants $\widehat{C}_F$, $\widehat{C}_G>0$ independent of $y\in \mathbb{R}$ such that
  \begin{align*}
  &|F''(y)|\leq \widehat{C}_F(1+|y|^{p}),\quad |G''(y)|\leq \widehat{C}_G(1+|y|^{q}),\;\;\;\; \forall\,y\in \mathbb{R},
  \end{align*}
  where the exponents $p$, $q\in [0,+\infty)$ are fixed numbers such that \\
  (i) when $\kappa>0$, $p, q$ are arbitrary if $d=2$; $p=2$, $q$ is arbitrary if $d=3$;\\
  (ii) when $\kappa=0$, $p$ is arbitrary if $d=2$ and $p=2$ if $d=3$; $q=0$ for $d=2,3$.
\end{itemize}

When global strong solutions for the case with surface diffusion (i.e., $\kappa>0$) is concerned, the growth condition (\textbf{A3}) can be replaced by an alternative assumption:
\begin{itemize}
\item[(\textbf{A4})] there exist some positive constants $\rho_1$, $\rho_2$ such that
$$
|\widetilde{F}'(y)|\leq \rho_1|\widetilde{G}'(y)|+\rho_2,\quad \forall\,y\in \mathbb{R}.
$$
\end{itemize}

\begin{remark}
 In this paper, we don't aim to pursue minimal assumptions on the potentials $F$, $G$ in order to avoid additional technical details.
 Nevertheless, the above assumptions cover some typical cases in the literature.

 (1) It is easy to verify that the classical double well potential $F(y)=\frac14(y^2-1)^2$ satisfies (\textbf{A1})--(\textbf{A3}).

 (2) According to (\textbf{A1})--(\textbf{A3}), the surface potential function $G$  is allowed to be a polynomial of even degree with a positive leading coefficient provided that $\kappa>0$. When $\kappa=0$, i.e., the surface diffusion is absent, the assumption on $G$ turns out to be more restrictive. Nevertheless, it covers the typical form of the fluid-solid interfacial free energy in the moving contact line problem (see \cite[Section 4]{Q1}):
 $$
 \frac12\phi^2+G(\phi)= \frac{\gamma}{2} \cos \theta_s \sin \left(\frac{\pi}{2}\phi\right)
 $$
 where $\theta_s$ is the static contact angle and $\gamma$ is the surface tension coefficient on $\Gamma$.

  (3) Assumption (\textbf{A4}) removes the restriction on the growth of potential functions $F$ and $G$. Instead, it requires that the boundary potential $G$ plays a dominating role.
  This assumption can be viewed as a compatibility condition between $F$ and $G$, and it can be further extended to the case with physically relevant singular potentials (cf. \cite{CF15a, CGS14}). A similar assumption in the opposite direction such that the bulk potential $F$ is dominative can be found in \cite[(2.35)--(2.36)]{GMS09}.
\end{remark}

\subsection{Preliminaries}
We assume that $\Omega\subset \mathbb{R}^d$ ($d=2,3$) is a bounded domain with smooth boundary $\Gamma:=\partial \Omega$, $\mathbf{n}$ is the unit outward normal vector on $\Gamma$ and
 $\partial_\mathbf{n}$ denotes the outer normal derivative on $\Gamma$.
Then we introduce the functional framework that will be used in the remaining part of this paper.
For a Banach space $\mathcal{X}$, its norm is denoted by $\|\cdot\|_{\mathcal{X}}$.
The symbol $\langle\cdot, \cdot\rangle_{\mathcal{X}^*,\mathcal{X}}$ stands for a duality pairing between $\mathcal{X}$ and its dual space $\mathcal{X}^*$.
We denote by $L^p(\Omega)$ and $L^p(\Gamma)$ $(p\geq 1)$ the standard Lebesgue spaces with respective norms $\|\cdot\|_{L^p(\Omega)}$
and $\|\cdot\|_{L^p(\Gamma)}$. When $p=2$, the inner products in the Hilbert spaces $L^2(\Omega)$ and $L^2(\Gamma)$ will be denoted by
$(\cdot, \cdot)_{L^2(\Omega)}$ and $(\cdot, \cdot)_{L^2(\Gamma)}$, respectively.
For $s\in \mathbb{R}$, $p\geq 1$, $W^{s,p}(\Omega)$ and $W^{s,p}(\Gamma)$
stand for the Sobolev spaces with corresponding norms $\|\cdot\|_{W^{s,p}(\Omega)}$ and $\|\cdot\|_{W^{s,p}(\Gamma)}$. If $p=2$, we simply denote $W^{s,p}(\Omega)=H^s(\Omega)$ and $W^{s,p}(\Gamma)=H^s(\Gamma)$.
%
We put the product space
$$
\mathcal{H}=L^2(\Omega)\times L^2(\Gamma),
$$
which is a Hilbert space that can be viewed as the completion of $C^0(\overline{\Omega})$ with respect to the norm
 \begin{align}
\|(\phi, \psi)\|_{\mathcal{H}}^2=\int_{\Omega} |\phi|^2\, \d x
+\int_{\Gamma} |\psi|^2 \,\d S,\quad \forall\ (\phi, \psi) \in \mathcal{H}.
\label{V0no}
 \end{align}
 We note that any element $h=(\phi,\psi) \in \mathcal{H}$ will be thought as a pair of functions
belonging, respectively, to $L^2(\Omega)$ and to $L^2(\Gamma)$.
If no additional regularity is imposed, the second component of $h$ (i.e., $\psi$) is not necessary to be
the trace of the first one (i.e., $\phi$).
We recall that the Dirichlet trace map $\gamma : \{\phi|_{\overline{\Omega}}:\ \phi \in C^\infty(\mathbb{R}^d)\}\to C^\infty(\Gamma)$,
defined by $\gamma\phi=\phi|_{\Gamma}$, extends to a linear continuous operator $\gamma : H^s(\Omega)\to H^{s-\frac12}(\Gamma)$, for all
$s >\frac12$ (cf. \cite[Theorem 3.37]{ML00}). In the following text, we shall always use the notion $\phi|_\Gamma$ to indicate the trace operator defined in a suitable sense. Thus, when we consider a function $\phi\in H^s(\Omega)$ (with $s>\frac12$), the symbol $\phi$ will be intended, as a pair $(\phi, \psi)$ formed by
the function $\phi$ in $\Omega$ and its trace $\psi:=\phi|_\Gamma$ on $\Gamma$. In this context, we introduce the notions
$$
V^s=\big\{(\phi, \psi)\in H^s(\Omega)\times H^{s-\frac12}(\Gamma):\ \psi=\phi|_\Gamma \big\}, \quad \forall\, s>\frac12,
$$
with the equivalent norm given by $\|\phi\|_{H^s(\Omega)}$.
Besides, we denote
$$
\mathcal{V}^s=\big\{(\phi, \psi)\in H^s(\Omega)\times H^{s}(\Gamma):\ \psi=\phi|_\Gamma \big\}, \quad \forall\, s>\frac12,
$$
with the induced graph norm given by
 $\|(\phi, \psi)\|^2_{\mathcal{V}^s}=\|\phi\|^2_{H^s(\Omega)}+\|\psi\|^2_{H^s(\Gamma)}$. We note that $\mathcal{V}^s$
can be identified with a closed subspace of the product space $H^s(\Omega)\times H^s(\Gamma)$ and for any $s_1>s_2>\frac12$, the dense and compact embeddings $\mathcal{V}^{s_1}\hookrightarrow \mathcal{V}^{s_2}$ hold.
For later convenience, we set, for a parameter $\kappa\geq 0$
$$
\mathds{V}_\kappa^s:=\mathcal{V}^s\ \ \text{if}\ \kappa>0,
\qquad \mathds{V}^s_\kappa:=V^s\ \ \text{if}\ \kappa=0.
$$
We see that for any fixed $\kappa\geq 0$, $\mathds{V}_\kappa^{1}$ is a Hilbert space, which can be viewed as the
completion of $C^1(\overline{\Omega})$ with respect to the following equivalent norm
\begin{align}
 \|(\phi,\psi)\|_{\mathds{V}_\kappa^1}^2
 =\int_{\Omega} |\nabla \phi|^2\, \d x
+\int_{\Gamma}(\kappa|\nabla_{\Gamma} \psi|^2 + |\psi|^2)\, \d S,
\quad \forall\ (\phi, \psi) \in \mathds{V}_\kappa^1.
\label{V1no}
\end{align}

We denote by $|\Omega|$ the Lebesgue measure of the domain $\Omega$ and by $|\Gamma|$ the $d-1$-dimensional measure of its boundary $\Gamma$. For every $g\in (H^1(\Omega))^*$ (resp. $g\in (H^1(\Gamma))^*$), we denote by $\langle g\rangle_\Omega$ (resp. $\langle g\rangle_\Gamma$) the generalized average of $g$ over $\Omega$ (resp. $\Gamma$) such that
\begin{align*}
&\langle{g}\rangle_\Omega=\frac{1}{|\Omega|}\langle g,1\rangle_{(H^1(\Omega))^*,H^1(\Omega)},
\quad \langle{g}\rangle_\Gamma=\frac{1}{|\Gamma|}\langle g,1\rangle_{(H^1(\Gamma))^*,H^1(\Gamma)}.
\end{align*}
If $g\in L^2(\Omega)$ (resp. $g\in L^2(\Gamma)$), the above mean values simply reduce to
\begin{align*}
&\langle{g}\rangle_\Omega=\frac{1}{|\Omega|}\int_\Omega g\, \d x,\quad \langle{g}\rangle_\Gamma=\frac{1}{|\Gamma|}\int_\Gamma g\, \d S.
\end{align*}
Then we set
\begin{align*}
 \dot{\mathcal{H}}& =\big\{(\phi, \psi)\in \mathcal{H}: \ \langle \phi \rangle_\Omega=\langle \psi \rangle_\Gamma=0\big\}, \\
 \dot{V}^s& =\big\{(\phi, \psi)\in V^s: \ \langle \phi \rangle_\Omega=\langle \psi \rangle_\Gamma=0\big\},\quad \forall\, s>\frac12,\\
 \dot{\mathcal{V}}^s& =\big\{(\phi, \psi)\in \mathcal{V}^s: \ \langle \phi \rangle_\Omega=\langle \psi \rangle_\Gamma=0\big\},\quad \forall\, s>\frac12.
\end{align*}
The minus Laplace operator in $\Omega$ subject to the homogeneous Neumann boundary condition and its domain are denoted by
\begin{align*}
&A_\Omega=-\Delta :D(A_\Omega)\subset L^2(\Omega) \to L^2(\Omega), \\
&\qquad \text{with}\ \  D(A_\Omega) = \{\phi \in H^2(\Omega):\partial_\mathbf{n} \phi=0 \ \text{on}\ \Gamma\}.
\end{align*}
We denote by $A_\Omega^0$ the restriction of $A_\Omega$ to the linear subspace of $L^2(\Omega)$ with zero mean values $\dot{L^2}(\Omega):=\{\phi\in L^2(\Omega):\ \langle \phi \rangle_\Omega=0\}$. Note that $A_\Omega^0$ is a positive linear operator and $(A_\Omega^0)^{-1}$ is compact on $\dot{L^2}(\Omega)$.
Set
$$
(H^1(\Omega))^*_0=\{\phi\in (H^1(\Omega))^*:\ \langle \phi\rangle_\Omega=0\}.
$$
 For any function $\phi\in (H^1(\Omega))^*_0$, the Neumann problem for $A_\Omega$ in $\Omega$ with source $\phi$ admits a unique weak solution $u=(A_\Omega^0)^{-1} \phi\in H^1(\Omega)\cap \dot{L^2}(\Omega)$.
As a consequence, for any $u, w\in (H^1(\Omega))^*_0$, we have
\begin{align*}
\langle u, (A_\Omega^0)^{-1} w\rangle_{(H^1(\Omega))^*,H^1(\Omega)}
& = \langle w, (A_\Omega^0)^{-1} u\rangle_{(H^1(\Omega))^*, H^1(\Omega)}\\
& =\int_\Omega \nabla  (A_\Omega^0)^{-1} u\cdot \nabla  (A_\Omega^0)^{-1} w \d x.
\end{align*}
Concerning the minus Laplace--Beltrami operator $A_\Gamma:= -\Delta_\Gamma$ defined on $\Gamma$, it is well known that $A_\Gamma$ is a
nonnegative self-adjoint operator in $L^2(\Gamma)$. Thus, the Sobolev spaces $H^s(\Gamma)$, $s\in \mathbb{R}$, can
 be identified as $H^s(\Gamma)=D((I+A_\Gamma)^\frac{s}{2})$. In particular, $\|\psi\|_{L^2(\Gamma)}+\|A_\Gamma \psi\|_{L^2(\Gamma)}$ is an
equivalent norm on $H^2(\Gamma)$ and $\|\psi\|_{L^2(\Gamma)}+\|\nabla_\Gamma \psi\|_{L^2(\Gamma)}$ is an
equivalent norm on $H^1(\Gamma)$ (cf. \cite{LM68}). The restriction of $A_\Gamma$ to the linear subspace  $\dot{L^2}(\Gamma):=\{\phi\in L^2(\Gamma):\, \langle \phi \rangle_\Gamma=0\}$ is denoted by $A_\Gamma^0$, which is a positive linear operator.
Besides, its inverse $(A_\Gamma^0)^{-1}$ is a compact operator on $\dot{L^2}(\Gamma)$. Like before, we set
$$
(H^1(\Gamma))^*_0=\{\psi\in (H^1(\Gamma))^*:\ \langle \psi\rangle_\Gamma=0\}.
$$
Then for any function $\psi\in (H^1(\Gamma))^*_0$, the Poisson equation for $A_\Gamma$ on $\Gamma$ (i.e., a compact $d-1$ manifold without boundary) with source $\psi$ admits a unique weak solution $v=(A_\Gamma^0)^{-1} \psi\in H^1(\Gamma)\cap \dot{L^2}(\Gamma)$.
As a consequence, for any $v, w\in (H^1(\Gamma))^*_0$, we have
\begin{align*}
\langle v, (A_\Gamma^0)^{-1} w\rangle_{(H^1(\Gamma))^*,H^1(\Gamma)}
& = \langle w, (A_\Gamma^0)^{-1} v\rangle_{(H^1(\Gamma))^*, H^1(\Gamma)}\\
& = \int_\Gamma \nabla_\Gamma  (A_\Gamma^0)^{-1} v\cdot \nabla_\Gamma  (A_\Gamma^0)^{-1} w\, \d S.
\end{align*}
We can endow the spaces $(H^1(\Omega))^*$, $(H^1(\Omega))^*_0$, $(H^1(\Gamma))^*$ and $(H^1(\Gamma))^*_0$ with the following equivalent norms
\begin{align*}
&\|\phi\|^2_{(H^1(\Omega))^*}=\|\nabla (A_\Omega^0)^{-1}(\phi-\langle \phi\rangle_\Omega)\|_{L^2(\Omega)}^2+ |\langle \phi\rangle_\Omega|^2,\qquad\ \, \forall\, \phi\in (H^1(\Omega))^*,\\
&\|\phi\|^2_{(H^1(\Omega))^*_0}=\|\nabla (A_\Omega^0)^{-1}\phi\|_{L^2(\Omega)}^2,\qquad\qquad\qquad\qquad\qquad \ \  \forall\,\phi\in (H^1(\Omega))^*_0,\\
&\|\psi\|^2_{(H^1(\Gamma))^*}=\|\nabla_\Gamma(A_\Gamma^0)^{-1}(\psi-\langle \psi\rangle_\Gamma)\|_{L^2(\Gamma)}^2+ |\langle \psi\rangle_\Gamma|^2,\qquad \forall\,\psi\in (H^1(\Gamma))^*,\\
&\|\psi\|^2_{(H^1(\Gamma))^*_0}=\|\nabla_\Gamma(A_\Gamma^0)^{-1}\psi\|_{L^2(\Gamma)}^2,
\qquad\qquad\qquad\qquad\qquad \  \forall\, \psi\in (H^1(\Gamma))^*_0.
\end{align*}
Moreover, we have
\begin{align*}
&\frac12\frac{\d}{\d t}\|\phi\|^2_{(H^1(\Omega))^*_0}=\langle \phi_t, (A_\Omega^0)^{-1} \phi\rangle_{(H^1(\Omega))^*,H^1(\Omega)},\quad \forall\, \phi\in H^1(0,T;(H^1(\Omega))^*_0),\\
&\frac12\frac{\d}{\d t}\|\psi\|^2_{(H^1(\Gamma))^*_0}=\langle \psi_t, (A_\Gamma^0)^{-1} \psi\rangle_{(H^1(\Gamma))^*,H^1(\Gamma)},\ \ \
\forall\, \psi\in H^1(0,T;(H^1(\Gamma))^*_0).
\end{align*}

Finally, for reader's convenience, we collect some useful results in functional analysis that will be frequently used later.
\begin{itemize}
\item \emph{The Gagliardo--Nirenberg Inequality} (see e.g., \cite{Ni}):
\be
\|D^j f\|_{L^p(\Omega)}\leq C
\|f\|_{L^q(\Omega)}^{1-a} \| f\|_{W^{m,r}(\Omega)}^a ,
\quad \forall \, f\in W^{m,r}(\Omega)\cap L^q(\Omega),\nonumber
\ee
where $j, m$ are arbitrary integers satisfying
$0\leq j< m$ and $\frac{j}{m}\leq a\leq 1$, and
$1\leq q, r\leq +\infty$  such that
$$
\frac{1}{p}-\frac{j}{d}=a\left(\frac{1}{r}-\frac{m}{d}\right)+(1-a)\frac{1}{q}.
$$
If $1<r<+ \infty$ and $m-j-\frac{n}{r}$ is a nonnegative integer, then the
above inequality holds only for $\frac{j}{m}\leq a<1$.

\item \emph{Trace Theorem} (cf. \cite[Part I, Chapter 2, Theorems 2.24, 2.27]{BG}).

(1) Assume that $1\leq p\leq\infty$, $s>\frac{1}{p}$ and $s-\frac{1}{p}$ is not an integer, then there exists a unique operator $\gamma_0\in \mathcal{L}(W^{s,p}(\Omega); W^{s-\frac{1}{p}, p}(\Gamma))$ such that $\gamma_0 v=v|_\Gamma$, $\forall\, v\in C^\infty(\overline{\Omega})\cap W^{s,p}(\Omega)$.

(2) Let $A=-\Delta$. Consider the space $W^{s,p}_{A,j}(\Omega)=\{v\in W^{s,p}(\Omega),\, Av\in W^{-2+j+\frac{1}{p}, p}(\Omega)\}$, $j=0,1$, $s\in \mathbb{R}$, $1<p<\infty$, and either $p=2$ or $s-\frac{1}{p}$ is not an integer. Then there exists unique operators $\gamma_0\in \mathcal{L}(W^{s,p}_{A,0}(\Omega); W^{s-\frac{1}{p}, p}(\Gamma))$ and $\gamma_A\in \mathcal{L}(W^{s,p}_{A,1}(\Omega); W^{s-1-\frac{1}{p}, p}(\Gamma))$ such that $\gamma_0 v=v|_\Gamma$ and $\gamma_A v=\partial_\mathbf{n} v|_\Gamma$, $\forall\, v\in C^\infty(\overline{\Omega})\cap W^{s,p}_{A,j}(\Omega)$.

\item \emph{Inverse Trace Theorem} (cf. \cite[Theorem 4.2.3]{HWb}).

Let $s>\frac12$. There exists a linear bounded right inverse $\mathcal{R}$ to the Dirichlet trace operator $\gamma_0$ with $\mathcal{R}: H^{s-\frac12}(\Gamma)\to H^s(\Omega)$, $\gamma_0(\mathcal{R}v)=v$ for all $v\in H^{s-\frac12}(\Gamma)$. Moreover,
\begin{align}
\|\mathcal{R}v\|_{ H^s(\Omega)}\leq c_\mathcal{R}\|v\|_{H^{s-\frac12}(\Gamma)},\label{cZ}
\end{align}
with $c_\mathcal{R}>0$ being independent of $v$.

\item \emph{Generalized First Green's Formula} (cf. \cite[Lemma 5.1.1]{HWb}).

Let $A=-\Delta$ and $H^1_{A}(\Omega)=\{u\in H^1(\Omega),\, Au\in (H^1(\Omega))^*\}$ equipped with the graph norm $\|u\|_{H^1_{A}(\Omega)}:=\|u\|_{H^1(\Omega)}+\|Au\|_{(H^1(\Omega))^*}$.
For any fixed $u\in  H^1_{A}(\Omega)$, the mapping
$$
v\mapsto \langle\tau u, v\rangle_{H^{-\frac12}(\Gamma), H^\frac12(\Gamma)}:=\int_\Omega \nabla u\cdot \nabla (\mathcal{R}v)\, \d x-\langle Au, \mathcal{R}v\rangle_{(H^1(\Omega))^*, H^1(\Omega)}
$$
is a continuous linear functional $\tau u$ on arbitrary $v\in H^\frac12(\Gamma)$ that coincides for $u \in H^2(\Omega)$ with $\partial_\mathbf{n}u$, i.e., $\tau u=\partial_\mathbf{n} u|_\Gamma$. The mapping $\tau:  H^1_{A}(\Omega)\to H^{-\frac12}(\Gamma)$ with $u\to \tau u$ is continuous. Here, $\mathcal{R}$ is a right inverse to the trace operator $\gamma_0$ defined above. However, the resulting operator $\tau u=\partial_\mathbf{n} u$ does not depend on the special choice of $\mathcal{R}$. Besides,
\begin{align}
\|\partial_\mathbf{n} u\|_{H^{-\frac12}(\Gamma)}\leq c_\mathcal{R}\Big(\|\nabla u\|_{L^2(\Omega)}+\|Au\|_{(H^1(\Omega))^*}\Big),\label{cZ1}
\end{align}
where $c_\mathcal{R}$ is given in \eqref{cZ} being independent of $u$. Finally, the following generalized first Green's formula holds
$$
\langle Au, w\rangle_{(H^1(\Omega))^*, H^1(\Omega)}=\int_\Omega \nabla u\cdot \nabla w\, \d x-\langle\tau u, \gamma_0 w\rangle_{H^{-\frac12}(\Gamma), H^\frac12(\Gamma)},
$$
for every $u\in H^1_{A}(\Omega)$ and $w\in H^1(\Omega)$.
\end{itemize}

\subsection{Main results}
As in Section 2, the total free energy for problem \eqref{CH} is given by
\begin{align}
E(\phi, \psi)&=\frac12\int_{\Omega} |\nabla \phi|^2\, \d x
+ \frac12 \int_{\Gamma}\big(\kappa|\nabla_{\Gamma} \psi|^2 + |\psi|^2\big)\, \d S\nonumber\\
&\quad  +\int_\Omega F(\phi)\,\d x
+\int_\Gamma G(\psi)\,\d S, \quad\text{for some}\ \kappa\geq 0. \label{E}
\end{align}
Next, we introduce the notion of (global) weak/strong solutions to problem \eqref{CH}:
\begin{definition}
\label{defweak}
Let $\kappa\geq 0$, $T\in (0,+\infty)$. For any initial data $(\phi_0, \psi_0)\in \mathcal{V}^1$, a pair $(\phi,\psi)$
is called a weak solution to problem \eqref{CH} on $[0,T]$, if it satisfies
\begin{align*}
&(\phi, \psi)\in L^\infty(0,T; \mathds{V}^1_\kappa)\cap L^2(0,T; \mathds{V}^r_\kappa),\\
&(\mu, \mu_\Gamma)\in L^2(0,T; H^1(\Omega)\times H^1(\Gamma)),\\
&(\phi_t, \psi_t)\in L^2(0, T; (H^1(\Omega))^*\times (H^1(\Gamma))^*),
\end{align*}
with $r=3$ if $\kappa>0$ and $r=\frac52$ if $\kappa=0$. The following weak formulations are satisfied
\begin{align}
&\langle \phi_t(t), \zeta\rangle_{(H^1(\Omega))^*, H^1(\Omega)}
+\int_\Omega \nabla \mu (t)\cdot \nabla \zeta\, \d x=0,
\label{we1} \\
&\langle \psi_t(t), \eta\rangle_{(H^1(\Gamma))^*, H^1(\Gamma)}
+\int_\Gamma \nabla_\Gamma \mu_\Gamma(t) \cdot \nabla_\Gamma\eta\, \d S=0,
\label{we1a}
\end{align}
for every $\zeta\in H^1(\Omega)$ and $\eta\in H^1(\Gamma)$ and almost every $t\in (0,T)$, with
\begin{align}
&\mu=-\Delta \phi +F'(\phi),\qquad\qquad \qquad\quad\ \,  \text{a.e. in } \Omega\times (0,T),\label{we2a}\\
&\mu_\Gamma= -\kappa\Delta_\Gamma \psi +  \psi+\partial_\mathbf{n}\phi +G'(\psi),
\quad \text{a.e. on } \Gamma\times (0,T).
\label{we2}
\end{align}
Besides, the initial conditions are fulfilled
\begin{align}
\phi|_{t=0}=\phi_0(x)\quad \text{in}\ \Omega, \quad \text{and}\quad \psi|_{t=0}=\psi_0(x)\quad \text{on}\ \Gamma,
\label{ini}
\end{align}
and the bulk/surface mass conservation properties hold
\begin{align}
\langle\phi(t)\rangle_\Omega=\langle\phi_0\rangle_\Omega,\quad \langle\psi(t)\rangle_\Gamma=\langle\psi_0\rangle_\Gamma,\quad \forall\,t\in [0,T].
\label{massconwe}
\end{align}
Moreover, $(\phi,\psi)$ satisfies the energy inequality
\begin{align}
&E(\phi(t), \psi(t))
+\int_0^{t}\Big(\|\nabla\mu(\tau)\|_{L^2(\Omega)}^2+\|\nabla_\Gamma \mu_\Gamma(\tau) \|_{L^2(\Gamma)}^2\Big)\,\d \tau
\leq E(\phi_0, \psi_0),\quad \forall\, t\in [0,T],
\label{aenergyin}
\end{align}
where $E$ is defined in \eqref{E}.
\end{definition}
\begin{definition}
\label{defstrong}
Let $\kappa\geq 0$, $T\in (0,+\infty)$. For any initial datum $(\phi_0, \psi_0)\in \mathcal{V}^3$, a pair $(\phi,\psi)$
is called a strong solution to problem \eqref{CH} on $[0,T]$, if it satisfies the additional regularity
\begin{align*}
&(\phi, \psi)\in L^\infty(0,T; \mathds{V}^{r_1}_\kappa)\cap L^2(0,T; \mathds{V}^{r_2}_\kappa),\\
&\mu\in L^\infty(0,T; H^1(\Omega))\cap L^2(0,T; H^3(\Omega)),\\
& \mu_\Gamma \in L^\infty(0,T; H^1(\Gamma))\cap L^2(0, T; H^{r_3}(\Gamma)),\\
&(\phi_t, \psi_t)\in L^\infty(0,T; (H^1(\Omega))^*\times(H^1(\Gamma))^*)\cap L^2(0, T;\mathds{V}^1_\kappa),
\end{align*}
with $r_1=3$, $r_2=5$, $r_3=3$ if $\kappa>0$ and $r_1=\frac52$, $r_2=\frac72$, $r_3=2$ if $\kappa=0$.
The equations and boundary conditions in problem \eqref{CH} are satisfied a.e. in $\Omega\times (0,T)$ and on $\Gamma\times(0,T)$, respectively.
The initial conditions and the mass conservation \eqref{massconwe} hold as well. Besides, the following energy identity is satisfied
\begin{align}
&\frac{\d}{\d t} E(\phi(t), \psi(t))+\|\nabla\mu\|_{L^2(\Omega)}^2+\|\nabla_\Gamma\mu_\Gamma \|_{L^2(\Gamma)}^2=0, \quad \text{for a.e. } t\in (0,T).\label{BELa}
\end{align}
\end{definition}
\smallskip

We are in a position to state the main results of this paper. \smallskip

Our first result regards the global well-posedness of problem \eqref{CH}. We distinguish two
cases with surface diffusion ($\kappa>0$) or without surface diffusion ($\kappa=0$).

\begin{theorem}[Global weak/strong solutions for $\kappa>0$]
\label{main1}
Let $\kappa>0$, $T\in (0,+\infty)$, $\Omega\subset \mathbb{R}^d$ ($d=2,3$) be a bounded domain with smooth boundary $\Gamma$.

 (1) If assumptions (\textbf{A1})--(\textbf{A3}) are satisfied, then for any initial datum $(\phi_0, \psi_0)\in \mathcal{V}^1$, problem \eqref{CH} admits a unique global weak solution $(\phi, \psi)$ on $[0,T]$.

 (2) If assumptions (\textbf{A1}), (\textbf{A2}) are satisfied, besides, either  (\textbf{A3}) or (\textbf{A4}) is fulfilled,   then for any initial datum $(\phi_0, \psi_0)\in \mathcal{V}^3$,  problem \eqref{CH} admits a unique global strong solution $(\phi, \psi)$ on $[0,T]$.
\end{theorem}

\begin{theorem}[Global weak/strong solutions for $\kappa=0$]\label{main1a}
 Let $\kappa=0$, $T\in (0,+\infty)$ and assumptions (\textbf{A1})--(\textbf{A3}) be satisfied. Suppose that $\Omega\subset \mathbb{R}^d$ ($d=2,3$) is a bounded domain with smooth boundary $\Gamma$ satisfying the condition
 \begin{align}
 c_\mathcal{R}|\Gamma|^\frac12|\Omega|^{-1}<1,\label{geo}
 \end{align} where $c_\mathcal{R}>0$ is the constant given in the inverse trace theorem (see \eqref{cZ}).
 For any initial datum $(\phi_0, \psi_0)\in \mathcal{V}^1$ (resp. $(\phi_0, \psi_0)\in \mathcal{V}^3$), problem \eqref{CH} admits a unique global weak (resp. strong) solution $(\phi, \psi)$ on $[0,T]$.
\end{theorem}

\begin{remark}
(1) When $\kappa=0$, the results are less satisfactory since the requirement on the function space for initial datum is slightly stronger than the solution itself and an additional geometric assumption \eqref{geo} on the domain is needed. This is mainly due to some technical difficulties when we try to obtain uniform estimates independent of $\kappa$ (cf. Section 5).

(2) In our current system, we were addressing the situations when the physics on the boundary is allowed to be somewhat independent to those in the bulk and the resulting dynamics happen in different spatial scales. Thus the usual isoperimetric relations or scaling (from the bulk, of the whole domain) may not hold anymore. Nevertheless, by a simple scaling argument, we see that typical domains like a ball (in $3D$) or a circle (in $2D$) with a sufficiently large radius $R$ are admissible to the condition \eqref{geo}.
\end{remark}
\begin{remark}
 In a very recent preprint \cite{GK18}, the authors proved the existence and uniqueness of global weak solutions to problem \eqref{CH} without using the assumption \eqref{geo} when $\kappa=0$.
 We note that the notion of weak solution introduced in \cite[Definition 2]{GK18} is weaker than that in our Definition \ref{defweak}.
 Instead of the pointwise expressions \eqref{we2a}--\eqref{we2} for the chemical potentials $(\mu, \mu_\Gamma)$, due to the lower spatial regularity on $(\phi, \psi)$ obtained in \cite{GK18}, there the pair $(\mu, \mu_\Gamma)$ only satisfies a suitable weak formulation, which warrants the uniqueness of $(\mu, \mu_\Gamma)$ up to some constants.
\end{remark}
Our second result concerns the long-time behavior of problem \eqref{CH}.
The following additional assumption on the potentials is required:
\begin{itemize}
\item[(\textbf{AN})] $F$ and $G$ are real analytic functions on $\mathbb{R}$.
\end{itemize}

\begin{theorem}[Uniqueness of asymptotic limit as $t\to+\infty$]
\label{main2}
Suppose that the assumptions of Theorem \ref{main1} or Theorem \ref{main1a} are satisfied. In addition, we assume (\textbf{AN}).
Then for any initial datum $(\phi_0, \psi_0)\in \mathcal{V}^1$ (or $\in \mathcal{V}^3$), the corresponding global weak (or strong) solution $(\phi, \psi)$ of problem \eqref{CH} converges to a single equilibrium $(\phi_\infty, \psi_\infty)$ as time goes to infinity such that
\begin{align}
\lim_{t\to+\infty}\|(\phi(t), \psi(t))-(\phi_\infty, \psi_\infty)\|_{\mathds{V}^{r-\epsilon}_\kappa}=0,
\label{conv1}
\end{align}
for any $\epsilon\in (0,1)$, where $r=3$ if $\kappa>0$, $r=\frac52$ if $\kappa=0$, $(\phi_\infty, \psi_\infty)\in \mathds{V}_\kappa^r$ satisfies the following nonlinear nonlocal elliptic problem
\begin{equation}
\left\{
\begin{aligned}
&-\Delta \phi_\infty+F'(\phi_\infty)=\lambda_1,
&\text{in}\ \Omega, \\
&-\kappa\Delta_\Gamma \psi_\infty +  \psi_\infty +\partial_\mathbf{n}\phi_\infty + G'(\psi_\infty)=\lambda_2,
&\text{on}\ \Gamma,\\
&\quad \text{with}\ \ \langle \phi_\infty\rangle_\Omega = \langle \phi_0\rangle_\Omega,\quad \langle \psi_\infty\rangle_\Gamma = \langle \psi_0\rangle_\Gamma,
\end{aligned}
\right.
\label{sta}
\end{equation}
with the constants $\lambda_1, \lambda_2$ given by
\begin{equation}
\left\{
\begin{aligned}
&\lambda_1=-|\Omega|^{-1}|\Gamma|\langle \partial_\mathbf{n}\phi_\infty\rangle_\Gamma+ \langle F'(\phi_\infty)\rangle_\Omega,\\
&\lambda_2=\langle \partial_\mathbf{n}\phi_\infty\rangle_\Gamma+\langle\psi_\infty\rangle_\Gamma+\langle G'(\psi_\infty)\rangle_\Gamma.
\end{aligned}
\right.
\label{staL}
\end{equation}
Moreover, we have the following estimate on convergence rate
\begin{align}
\|(\phi(t), \psi(t))-(\phi_\infty, \psi_\infty)\|_{\mathds{V}_\kappa^1}\leq C(1+t)^{-\frac{\theta}{1-2\theta}},\quad \forall\, t\geq 0,\label{convrate}
\end{align}
where the constant $\theta\in (0, \frac12)$ depends on $(\phi_\infty, \psi_\infty)$ and $C$ is a positive constant depending on $\|(\phi_0, \psi_0)\|_{\mathcal{V}^1}$, $\|(\phi_\infty, \psi_\infty)\|_{\mathds{V}_\kappa^1}$, $\Omega$, $\Gamma$ and $\kappa$.
\end{theorem}

\begin{remark}
Combining the smoothing effect of solutions to problem \eqref{CH} (see e.g., Lemma \ref{comp}) with the energy method in \cite{WH07}, the same estimate on convergence rate as \eqref{convrate} can be obtained in higher-order spaces $\mathds{V}_\kappa^m$, for all $m\in \mathbb{N}$, namely,
\begin{align}
\|(\phi(t), \psi(t))-(\phi_\infty, \psi_\infty)\|_{\mathds{V}_\kappa^m}\leq C(1+t)^{-\frac{\theta}{1-2\theta}},\quad \forall\, t\geq 1,\nonumber
\end{align}
where
$C$ is a positive constant depending on $\|(\phi_0, \psi_0)\|_{\mathcal{V}^1}$, $\|(\phi_\infty, \psi_\infty)\|_{\mathds{V}_\kappa^m}$, $\Omega$, $\Gamma$ and $\kappa$.
\end{remark}

Finally, for any given $\kappa>0$, we are able to characterize the stability for local energy minimizers $(\phi^*, \psi^*)$ of $E(\phi, \psi)$ over the set
\begin{align}
\mathcal{K}_1=\{(\phi, \psi)\in \mathcal{V}^1:\ \langle\phi\rangle_\Omega=\langle \hat{\phi}\rangle_\Omega,
\quad \langle\psi\rangle_\Gamma=\langle \hat{\psi}\rangle_\Gamma\},\label{K1}
\end{align}
where the pair $(\hat{\phi}, \hat{\psi})\in \mathcal{V}^1$ is arbitrary but fixed. Namely, $(\phi^*, \psi^*)\in \mathcal{K}_1$ and
$$E(\phi^*, \psi^*)=\inf\{E(\phi, \psi):\ (\phi, \psi)\in \mathcal{K}_1\cap \mathbf{B}_{\mathcal{V}^1}((\phi^*, \psi^*); \sigma)\},\quad \text{for some}\ \sigma>0.$$
Here, $\mathbf{B}_{\mathcal{V}^1}((\phi, \psi); \sigma)$ denotes the open ball in $\mathcal{V}^1$ with radius $\sigma$ centered at $(\phi, \psi)$.

\begin{theorem}[Stability criterion]
\label{main3}
Suppose that $\kappa>0$, $\Omega\subset \mathbb{R}^d$ ($d=2,3$) is a bounded domain with smooth boundary $\Gamma$, moreover, the assumptions (\textbf{A2}), (\textbf{A3}) and (\textbf{AN}) are satisfied.

(1) Let $(\phi^*, \psi^*)\in \mathcal{K}_1$ be a local energy minimizer of $E(\phi, \psi)$ over the set $\mathcal{K}_1$. Then $(\phi^*, \psi^*)$ is Lyapunov stable. Namely, for any $\epsilon>0$, there exists $\delta\in (0,\sigma)$ such that if the initial datum $(\phi_0, \psi_0)\in \mathcal{K}_1$ satisfies $\|(\phi_0, \psi_0)-(\phi^*, \psi^*)\|_{\mathcal{V}^1}<\delta$, then the global weak solution $(\phi, \psi)$ to problem \eqref{CH} (cf. Theorem \ref{main1}) satisfies $$\|(\phi(t), \psi(t))-(\phi^*, \psi^*)\|_{\mathcal{V}^1}<\epsilon,\quad \forall\,t\geq 0.$$

(2) Let $(\phi^*, \psi^*)\in \mathcal{K}_1$ be a stationary point that is a weak solution of the stationary problem \eqref{sta}--\eqref{staL} (where the mass constraints are changed corresponding to the definition of the set $\mathcal{K}_1$). If $(\phi^*, \psi^*)$ does not attain any local minimum of $E(\phi, \psi)$ over $\mathcal{K}_1$, then $(\phi^*, \psi^*)$ is not Lyapunov stable.
\end{theorem}

\begin{remark}
Theorem \ref{main2} implies the long-time stabilization of problem \eqref{CH} for arbitrary large initial datum, while Theorem \ref{main3} provides a stability criterion for steady states that are allowed to be non-isolated. Furthermore, we easily infer from Theorems \ref{main2}, \ref{main3} that any isolated local energy minimizer of $E(\phi, \psi)$ is indeed (locally) asymptotic stable.
It remains an open question whether a stability result similar to Theorem \ref{main3} holds for the case $\kappa=0$.
\end{remark}

\section{A Regularized Problem: the Cahn--Hilliard Equation with Viscous Terms}
\setcounter{equation}{0}

In this section, we study an approximating problem for the original Cahn--Hilliard equation \eqref{CH} with surface diffusion (i.e., $\kappa>0$) by adding artificial viscous terms in both equations for $\phi$ and $\psi$.
More precisely, for any given $\alpha\in (0,1]$ and $\kappa>0$, we consider the following regularized problem:
\begin{equation}
\left\{
\begin{aligned}
&\phi_t^\alpha=\Delta \mu^\alpha,
\quad \text{with}\ \ \mu^\alpha=-\Delta \phi^\alpha +\alpha \phi_t^\alpha +F'(\phi^\alpha),
&\text{in}\  \Omega\times (0,T),\\
&\partial_\mathbf{n} \mu^\alpha=0,
&\text{on}\ \Gamma\times (0,T),\\
&\phi^\alpha|_\Gamma=\psi^\alpha,
&\text{on}\ \Gamma\times (0,T),\\
&\psi_t^\alpha=\Delta_\Gamma\mu_\Gamma^\alpha,
&\text{on}\ \Gamma\times (0,T),\\
&\quad \text{with}\ \ \mu_\Gamma^\alpha= -\kappa\Delta_\Gamma \psi^\alpha +  \psi^\alpha+ \alpha\psi^\alpha_t+\partial_\mathbf{n}\phi^\alpha+ G'(\psi^\alpha),
&\text{on}\ \Gamma\times (0,T),\\
&\phi^\alpha|_{t=0}=\phi_0(x),
&\text{in}\ \Omega,\\
&\psi^\alpha|_{t=0}=\psi_0(x):=\phi_0(x)|_{\Gamma},
&\text{on}\ \Gamma.
\end{aligned}
\right.
\label{ACH}
\end{equation}

\subsection{Results on related linear systems}
Below we report technical tools related to some specific linear elliptic and parabolic problems, which will be necessary in the subsequent analysis.

First, we recall the following regularity result for a linear elliptic boundary value problem (see, for instance, \cite[Corollary A.1]{MZ05}).
\begin{lemma}\label{esLepH2}
Let $\kappa>0$, $\Omega\subset \mathbb{R}^d$ ($d=2,3$) be a bounded domain with smooth boundary $\Gamma$. Consider the following linear elliptic problem
\begin{equation}
\left\{
\begin{aligned}
&-\Delta \phi=h_1,&\text{in}\ \Omega,\\
&\phi|_\Gamma=\psi,&\text{on}\ \Gamma,\\
&-\kappa\Delta_\Gamma \psi+\psi+\partial_\mathbf{n}\phi=h_2, &\text{on}\ \Gamma,\\
\end{aligned}
\right.
\label{Lep}
\end{equation}
where $(h_1, h_2)\in H^s(\Omega)\times H^s(\Gamma)$ for any $s\geq 0$ and $s+\frac12\notin \mathbb{N}$.
Then every solution $(\phi, \psi)$ to problem \eqref{Lep} satisfies the following estimate
\begin{align}
\|\phi\|_{H^{s+2}(\Omega)}+\|\psi\|_{H^{s+2}(\Gamma)}\leq C(\|h_1\|_{H^s(\Omega)}+\|h_2\|_{H^s(\Gamma)}),
\label{esLeph2}
\end{align}
for some constant $C>0$ that may depend on $\kappa$, $s$, $\Omega$ and $\Gamma$, but is independent of the solution $(\phi, \psi)$.
\end{lemma}
Next, we state a well-posedness result for a linear fourth-order parabolic equation subject to a fourth-order dynamic boundary condition, which is crucial to solve the nonlinear approximating problem \eqref{ACH}.
\begin{lemma}\label{exLpa}
Let $\alpha\in (0,1]$, $\kappa>0$ and $T\in (0,+\infty)$. For any  initial data $(\phi_0,\psi_0)\in \mathcal{V}^2$ and external source terms $(h_1,h_2)\in L^2(0, T; H^1(\Omega)\times H^1(\Gamma))\cap H^1(0,T; \mathcal{H})$, we consider the following linear system:
\begin{equation}
\left\{
\begin{aligned}
&\phi_t=\Delta \widetilde{\mu},\quad
\text{with}\ \ \widetilde{\mu}=-\Delta\phi+\alpha\phi_t+h_1,
&\text{in}\ \Omega\times(0,T),\\
&\partial_\mathbf{n}\widetilde{\mu}=0,
&\text{on}\ \Gamma\times(0,T),\\
&\phi|_\Gamma=\psi,
&\text{on}\ \Gamma\times(0,T),\\
&\psi_t=\Delta_\Gamma \widetilde{\mu}_\Gamma,
&\text{on}\ \Gamma\times(0,T),\\
&\quad \text{with}\ \ \widetilde{\mu}_\Gamma=-\kappa \Delta_\Gamma\psi+\psi+\partial_\mathbf{n}\phi+\alpha \psi_t+h_2,
&\text{on}\ \Gamma\times(0,T),\\
&\phi|_{t=0}=\phi_0(x),
&\text{in}\ \Omega,\\
&\psi|_{t=0}=\psi_0(x):=\phi_0(x)|_{\Gamma},
&\text{on}\ \Gamma.
\end{aligned}
\label{Lpa}
\right.
\end{equation}
 Then problem \eqref{Lpa} admits a unique global strong solution $(\phi, \psi)$ on $[0,T]$ such that
\begin{align*}
&(\phi, \psi)\in C([0,T]; \mathcal{V}^2)\cap L^2(0,T; \mathcal{V}^3),\\
&(\phi_t, \psi_t)\in L^\infty(0,T; \mathcal{H})\cap L^2(0,T; \mathcal{V}^1).
\end{align*}
Moreover, $(\phi, \psi)$ satisfies the mass conservation properties
\begin{align}
\langle\phi(t)\rangle_\Omega =\langle\phi_0\rangle_\Omega,\quad \langle\psi(t)\rangle_\Gamma =\langle\psi_0\rangle_\Gamma,\quad \forall\, t\in [0,T],\nonumber
\end{align}
and there exists a constant $C > 0$ that may depend on $\Omega$, $\Gamma$ and $\kappa$, but is independent of $\alpha$, $T$, $(\phi, \psi)$ as well as the initial datum $(\phi_0, \psi_0)$, such that the following estimates hold:
\begin{align}
&\|(\phi, \psi)\|_{L^\infty(0,T; \mathcal{V}^1)}^2+\|(\phi_t, \psi_t)\|_{L^2(0,T; (H^1(\Omega))^*\times (H^1(\Gamma))^*)}^2\nonumber\\
&\qquad +\alpha \|(\phi_t, \psi_t)\|_{L^2(0,T; \mathcal{H})}^2+ \|(\phi, \psi)\|_{L^2(0,T; \mathcal{V}^2)}^2\nonumber\\
&\qquad +\|\widetilde{\mu}\|_{L^2(0,T; H^1(\Omega))}^2+ \|\widetilde{\mu}_\Gamma\|_{L^2(0,T; H^1(\Gamma))}^2\nonumber\\
&\quad \leq C(1+T)\|(\phi_0, \psi_0)\|_{\mathcal{V}^1}^2+C(1+\alpha^{-1})\|(h_1, h_2)\|_{L^2(0, T; \mathcal{H})}^2,
\label{esLIN}
\end{align}
and
\begin{align}
&\|(\phi, \psi)\|_{L^\infty(0,T; \mathcal{V}^2)}^2+\|(\phi_t, \psi_t)\|_{L^\infty(0,T; (H^1(\Omega))^*\times(H^1(\Gamma))^*)}^2\nonumber\\
&\qquad +\alpha \|(\phi_t, \psi_t)\|_{L^\infty(0,T; \mathcal{H})}^2
 + \|(\phi_t, \psi_t)\|_{L^2(0,T; \mathcal{V}^1)}^2
 + \|(\phi, \psi)\|_{L^2(0,T; \mathcal{V}^3)}^2 \nonumber\\
&\qquad +\alpha \|\widetilde{\mu}\|_{L^2(0,T; H^2(\Omega))}^2
+ \alpha\|\widetilde{\mu}_\Gamma\|_{L^2(0,T; H^2(\Gamma))}^2\nonumber\\
&\quad \leq Ce^{CT}(1+\alpha^{-1})\|(\phi_0, \psi_0)\|_{\mathcal{V}^2}^2 +Ce^{CT}(1+\alpha^{-1})\|(h_1, h_2)\|_{L^2(0, T; H^1(\Omega)\times H^1(\Gamma))}^2\nonumber\\
&\qquad +Ce^{CT}(1+\alpha^{-1})\|(h_1, h_2)\|_{H^1(0, T; \mathcal{H})}^2.
\label{esLINh}
\end{align}
\end{lemma}
\begin{remark}
The proof of Lemma \ref{exLpa} is rather involved due to the fourth-order dynamic boundary condition for $\psi$. A detailed proof will be postponed to the Appendix.
\end{remark}

\subsection{Local well-posedness of the viscous Cahn--Hilliard equation}

We prove the main result of this section, namely, local well-posedness of the approximating problem \eqref{ACH}.

\begin{proposition}
\label{exeACH}
Let $\alpha\in (0,1]$ and $\kappa>0$ be given. Suppose that $\Omega\subset \mathbb{R}^d$ ($d=2,3$) is a bounded domain with smooth boundary $\Gamma$ and the assumption (\textbf{A1}) is satisfied. For any initial datum $(\phi_0, \psi_0)\in \mathcal{V}^2$, there exists a time $T_\alpha>0$ such that the viscous Cahn--Hilliard problem \eqref{ACH} admits a unique
strong solution $(\phi^\alpha, \psi^\alpha)$ on $[0,T_\alpha]$ satisfying
\begin{align*}
&(\phi^\alpha, \psi^\alpha)\in C([0,T_\alpha]; \mathcal{V}^2)\cap L^2(0,T_\alpha; \mathcal{V}^3)\nonumber\\
&(\phi^\alpha_t, \psi^\alpha_t)\in L^\infty(0,T_\alpha; \mathcal{H})\cap L^2(0,T_\alpha; \mathcal{V}^1),\\
&\mu^\alpha\in L^2(0,T_\alpha; H^2(\Omega)),
\quad \mu^\alpha_\Gamma \in L^2(0, T_\alpha; H^2(\Gamma)),
\end{align*}
with $(\phi^\alpha, \psi^\alpha)|_{t=0}=(\phi_0, \psi_0)$.
The equations and boundary conditions in the system \eqref{ACH} are satisfied almost everywhere in $\Omega\times(0,T_\alpha)$ and on $\Gamma\times (0,T_\alpha)$, respectively.
\end{proposition}
\begin{proof}
The proof of Proposition \ref {exeACH} relies on Lemma \ref{exLpa} for the linear parabolic system \eqref{Lpa} together with the contraction mapping principle.

For any given $T\in (0,+\infty)$ and $(\phi_0, \psi_0)\in \mathcal{V}^2$, we introduce the space
\begin{align}
\mathcal{Y}_T&:=\{(\phi^\alpha, \psi^\alpha): (\phi^\alpha, \psi^\alpha)\in C([0,T]; \mathcal{V}^2)\cap L^2(0,T; \mathcal{V}^3),\nonumber\\
&\qquad (\phi^\alpha_t, \psi^\alpha_t)\in L^\infty(0,T; \mathcal{H})\cap L^2(0,T; \mathcal{V}^1), \ (\phi^\alpha, \psi^\alpha)|_{t=0}=(\phi_0,\psi_0)\},
\label{ZT}
\end{align}
with the following equivalent norm
\begin{align}
\|(\phi, \psi)\|_{\mathcal{Y}_T}^2&=\|(\phi, \psi)\|_{L^\infty(0,T; \mathcal{V}^2)}^2+\|(\phi, \psi)\|_{L^2(0,T; \mathcal{V}^3)}^2\nonumber\\
& \quad +\|(\phi_t, \psi_t)\|_{L^\infty(0,T; \mathcal{H})}^2+\|(\phi_t, \psi_t)\|_{L^2(0,T; \mathcal{V}^1)}^2.
\label{normZ}
\end{align}
 Then for every given pair $(\widehat{\phi}^\alpha, \widehat{\psi}^\alpha)\in \mathcal{Y}_T$, we look for the unique strong solution $(\phi^\alpha, \psi^\alpha)\in \mathcal{Y}_T$ to the following auxiliary linear problem
\begin{equation}
\left\{
\begin{aligned}
&\phi^\alpha_t-\Delta \widetilde{\mu}^\alpha=0,
\quad \text{with}\ \ \widetilde{\mu}^\alpha=-\Delta\phi^\alpha+\alpha\phi^\alpha_t+h^\alpha_1,
&\text{in}\ \Omega\times(0,T),\\
&\partial_\mathbf{n}\widetilde{\mu}^\alpha=0,
&\text{on}\ \Gamma\times(0,T),\\
&\phi^\alpha|_\Gamma=\psi^\alpha,
&\text{on}\ \Gamma\times(0,T),\\
&\psi^\alpha_t-\Delta_\Gamma \widetilde{\mu}^\alpha_\Gamma=0,
&\text{on}\ \Gamma\times(0,T),\\
&\quad \text{with}\ \  \widetilde{\mu}^\alpha_\Gamma=-\kappa\Delta_\Gamma\psi^\alpha+\psi^\alpha+\alpha \psi^\alpha_t+\partial_\mathbf{n}\phi^\alpha+h^\alpha_2,
&\text{on}\ \Gamma\times(0,T),\\
&\phi^\alpha|_{t=0}=\phi_0(x),
&\text{in}\ \Omega,\\
&\psi^\alpha|_{t=0}=\psi_0(x):=\phi_0(x)|_{\Gamma},
&\text{on}\ \Gamma,
\end{aligned}
\label{Lpaalp}
\right.
\end{equation}
where
$$
h^\alpha_1=F'(\widehat{\phi}^\alpha)\quad \text{and}\quad h^\alpha_2=G'(\widehat{\psi}^\alpha).
$$
 By assumption (\textbf{A1}), the Sobolev embedding theorem and the fact $(\widehat{\phi}^\alpha, \widehat{\psi}^\alpha)\in \mathcal{Y}_T$, we can easily verify that the external source terms $(h^\alpha_1,h^\alpha_2)$ satisfy
 $$
 (h^\alpha_1,h^\alpha_2)\in L^2(0,T; H^1(\Omega)\times H^1(\Gamma))\cap H^1(0,T; \mathcal{H}).
 $$
 Therefore, it follows from Lemma \ref{exLpa} that the linear problem \eqref{Lpaalp} admits a unique strong solution $(\phi^\alpha, \psi^\alpha)\in \mathcal{Y}_T$. As a consequence, the mapping
$$
\mathfrak{S}: \mathcal{Y}_T\to\mathcal{Y}_T \quad \text{defined by}\quad \mathfrak{S}(\widehat{\phi}^\alpha, \widehat{\psi}^\alpha)=(\phi^\alpha, \psi^\alpha)
$$
is well-defined.

Next, we introduce a closed bounded subset of $\mathcal{Y}_T$ denoted by
$$
\mathcal{M}_T:=\{(\phi, \psi)\in \mathcal{Y}_T:\ \|(\phi, \psi)\|_{\mathcal{Y}_T}\leq M\},
$$
where $M\geq 1$ is a sufficiently large constant to be chosen later (see \eqref{M1}). Our aim is to prove that for some $T>0$ being sufficiently small, $\mathfrak{S}$ maps the set $\mathcal{M}_T$ into itself and it is indeed a contraction mapping on $\mathcal{M}_T$.

To this end, for any pair $(\widehat{\phi}^\alpha, \widehat{\psi}^\alpha)\in \mathcal{M}_T$, we already know that $(\phi^\alpha, \psi^\alpha)=\mathfrak{S}(\widehat{\phi}^\alpha, \widehat{\psi}^\alpha)\in \mathcal{Y}_T$.
Applying the estimate \eqref{esLINh}, we get
\begin{align}
&\|(\phi^\alpha, \psi^\alpha)\|_{L^\infty(0,T; \mathcal{V}^2)}^2
+\|(\phi_t^\alpha, \psi_t^\alpha)\|_{L^\infty(0,T; (H^1(\Omega))^*\times(H^1(\Gamma))^*)}^2\nonumber\\
&\qquad +\alpha \|(\phi_t^\alpha, \psi_t^\alpha)\|_{L^\infty(0,T; \mathcal{H})}^2
+ \|(\phi_t^\alpha, \psi_t^\alpha)\|_{L^2(0,T; \mathcal{V}^1)}^2
+ \|(\phi^\alpha, \psi^\alpha)\|_{L^2(0,T; \mathcal{V}^3)}^2 \nonumber\\
&\qquad +\alpha \|\widetilde{\mu}^\alpha\|_{L^2(0,T; H^2(\Omega))}^2
+ \alpha\|\widetilde{\mu}^\alpha_\Gamma\|_{L^2(0,T; H^2(\Gamma))}^2\nonumber\\
&\quad \leq C_1e^{C_1T}(1+\alpha^{-1})\|(\phi_0, \psi_0)\|_{\mathcal{V}^2}^2 \nonumber\\
&\qquad +C_1e^{C_1T}(1+\alpha^{-1})\|(h_1^\alpha, h_2^\alpha)\|_{L^2(0, T; H^1(\Omega)\times H^1(\Gamma))}^2\nonumber\\
&\qquad +C_1e^{C_1T}(1+\alpha^{-1})\|(h_1^\alpha, h_2^\alpha)\|_{H^1(0, T; \mathcal{H})}^2,
\label{esLINZh}
\end{align}
where the constant $C_1\geq 1$ is independent of $\alpha$ and $T$.
Besides, using assumption (\textbf{A1}) and the Sobolev embedding theorem, we obtain the following estimates
\begin{align}
&\|(h^\alpha_1,h^\alpha_2)\|_{L^2(0,T; H^1(\Omega)\times H^1(\Gamma))}^2\nonumber\\
&\quad \leq T\left(\|F'(\widehat{\phi}^\alpha)\|_{L^\infty(0,T;H^1(\Omega))}^2
+\|G'(\widehat{\psi}^\alpha)\|_{L^\infty(0,T;H^1(\Gamma))}^2\right)\nonumber\\
&\quad
\leq T|\Omega|\|F'(\widehat{\phi}^\alpha)\|_{L^\infty(0,T; L^\infty(\Omega))}^2
+T|\Gamma| \|G'(\widehat{\phi}^\alpha)\|_{L^\infty(0,T; L^\infty(\Gamma))}^2\nonumber\\
&\qquad +T\|F''(\widehat{\phi}^\alpha)\|_{L^\infty(0,T; L^\infty(\Omega))}^2\|\nabla \widehat{\phi}^\alpha\|_{L^\infty(0,T; L^2(\Omega))}^2 \nonumber\\
&\qquad +T \|G''(\widehat{\psi}^\alpha)\|_{L^\infty(0,T; L^\infty(\Gamma))}^2\|\nabla_\Gamma \widehat{\psi}^\alpha\|_{L^\infty(0,T; L^2(\Gamma))}^2\nonumber\\
&\quad \leq C_2T \max_{|y|\leq C_2'M}\Big(|F'(y)|^2+|G'(y)|^2+|F''(y)|^2M^2 +|G''(y)|^2 M^2\Big)\nonumber\\
&\quad \leq C_2T M^2 \max_{|y|\leq C_2'M}\Big(|F'(y)|^2+|G'(y)|^2+|F''(y)|^2 +|G''(y)|^2\Big)
\label{eshh}
\end{align}
and
\begin{align}
&\|(\partial_th^\alpha_1, \partial_t h^\alpha_2)\|_{L^2(0,T; \mathcal{H})}^2\nonumber\\
&\quad \leq T\|F''(\widehat{\phi}^\alpha)\|_{L^\infty(0,T; L^\infty(\Omega))}^2\|\widehat{\phi}^\alpha_t\|_{L^\infty(0,T; L^2(\Omega))}^2 \nonumber\\
&\qquad + T\|G''(\widehat{\psi}^\alpha)\|_{L^\infty(0,T; L^\infty(\Gamma))}^2
\|\widehat{\psi}^\alpha_t\|_{L^\infty(0,T; L^2(\Gamma))}^2\nonumber\\
& \quad  \leq C_2 T M^2\max_{|y|\leq C_2'M}\Big(|F''(y)|^2 +|G''(y)|^2\Big),
\label{eshha}
\end{align}
where the positive constants $C_2, C_2'\geq 1$ are independent of $\alpha$, $M$ and $T$.
As a consequence, it follows from \eqref{esLINZh}--\eqref{eshha} and the assumption $\alpha\in (0,1]$ that
\begin{align}
&\|(\phi^\alpha(t), \psi^\alpha(t))\|_{\mathcal{Y}_T}^2
 \leq C_3\alpha^{-2}e^{C_1T}\Big(\|(\phi_0, \psi_0)\|_{\mathcal{V}^2}^2 + T M^2 P_1(M)\Big),\label{esZt1}
\end{align}
where $C_3\geq 1$ is independent of $\alpha$, $M$ and $T$, $P_1$ is a monotone increasing function given by
$$
P_1(M)=\max_{|y|\leq C_2'M}\Big(|F'(y)|^2+|G'(y)|^2+|F''(y)|^2 +|G''(y)|^2\Big).
$$
In view of \eqref{esZt1}, we can choose $M$ to be a sufficiently large constant such that
\begin{align}
M^2\geq  4C_3\alpha^{-2}e^{C_1}\|(\phi_0, \psi_0)\|_{\mathcal{V}^2}^2.\label{M1}
\end{align}
Then for any $T$ satisfying
\begin{align}
0<T\leq \min\left\{1,\ \frac{\alpha^2}{4C_3e^{C_1}P_1(M)}\right\},\label{T1}
\end{align}
it easily follows from \eqref{esZt1}--\eqref{T1} that $\|(\phi^\alpha(t), \psi^\alpha(t))\|_{\mathcal{Y}_T}^2\leq \frac12 M^2$, which further implies $(\phi^\alpha, \psi^\alpha)=\mathfrak{S}(\widehat{\phi}^\alpha, \widehat{\psi}^\alpha)\in \mathcal{M}_T$. Namely,
 $$
 \mathfrak{S}:\ \mathcal{M}_T \to \mathcal{M}_T.
 $$

Next, we show that for certain sufficiently small $T>0$, the mapping $\mathfrak{S}:\ \mathcal{M}_T \to \mathcal{M}_T$ is indeed a contraction with respective to the metric induced by the norm of $\mathcal{Y}_T$ (cf. \eqref{normZ}).
Let $(\widehat{\phi}^\alpha_i, \widehat{\psi}^\alpha_i)\in \mathcal{M}_T$ and $(\phi^\alpha_i, \psi^\alpha_i)=\mathfrak{S}(\widehat{\phi}^\alpha_i, \widehat{\psi}^\alpha_i)$, $i=1,2$.
Their differences are denoted by
\begin{align*}
&\widehat{\phi}^\alpha=\widehat{\phi}^\alpha_1-\widehat{\phi}^\alpha_2,
\quad  \widehat{\psi}^\alpha=\widehat{\psi}^\alpha_1-\widehat{\psi}^\alpha_2,\\
&\overline{\phi}^\alpha= \phi^\alpha_1- \phi^\alpha_2,
\quad  \overline{\psi}^\alpha=\psi^\alpha_1-\psi^\alpha_2.
\end{align*}
Then $(\overline{\phi}^\alpha, \overline{\psi}^\alpha)$ is a strong solution of the following system
\begin{equation}
\left\{
\begin{aligned}
&\overline{\phi}^\alpha_t-\Delta \overline{\mu}^\alpha=0,
\quad \text{with}\ \ \overline{\mu}^\alpha=-\Delta\overline{\phi}^\alpha
+\alpha\overline{\phi}^\alpha_t+\overline{h}^\alpha_1,
&\text{in}\ \Omega\times (0,T),\\
&\partial_\mathbf{n}\overline{\mu}^\alpha=0,
&\text{on}\ \Gamma\times (0,T),\\
&\overline{\phi}^\alpha|_\Gamma=\overline{\psi}^\alpha,
&\text{on}\ \Gamma\times (0,T),\\
&\overline{\psi}^\alpha_t-\Delta_\Gamma \overline{\mu}^\alpha_\Gamma=0,
&\text{on}\ \Gamma\times (0,T),\\
&\quad \text{with}\ \ \overline{\mu}^\alpha_\Gamma=-\kappa \Delta_\Gamma\overline{\psi}^\alpha+\overline{\psi}^\alpha+\alpha \overline{\psi}^\alpha_t+\partial_\mathbf{n}\overline{\phi}^\alpha+\overline{h}^\alpha_2,
&\text{on}\ \Gamma\times (0,T),\\
&\overline{\phi}^\alpha|_{t=0}=0,
&\text{in}\ \Omega,\\
&\overline{\psi}^\alpha|_{t=0}=0,
&\text{on}\ \Gamma,
\end{aligned}
\label{Lpaalpd}
\right.
\end{equation}
where
\begin{align*}
\overline{h}^\alpha_1=F'(\widehat{\phi}^\alpha_1)-F'(\widehat{\phi}^\alpha_2),
\quad\overline{h}^\alpha_2=G'(\widehat{\psi}^\alpha_1)-G'(\widehat{\psi}^\alpha_2).
\end{align*}
We infer from the linear estimate \eqref{esLINh} that
\begin{align}
&\|(\overline{\phi}^\alpha, \overline{\psi}^\alpha)\|_{L^\infty(0,T; \mathcal{V}^2)}^2
+ \|(\overline{\phi}^\alpha, \overline{\psi}^\alpha)\|_{L^2(0,T; \mathcal{V}^3)}^2\nonumber\\
&\qquad  + \|(\overline{\phi}^\alpha_t, \overline{\psi}^\alpha_t)\|_{L^\infty(0,T; \mathcal{H})}^2
+ \|(\overline{\phi}^\alpha_t, \overline{\psi}^\alpha_t)\|_{L^2(0,T; \mathcal{V}^1)}^2
 \nonumber\\
&\quad \leq C_4e^{C_4T}\alpha^{-2}\Big(\|(\overline{h}^\alpha_1, \overline{h}^\alpha_2)\|_{L^2(0, T; H^1(\Omega)\times H^1(\Gamma))}^2
+\|(\overline{h}^\alpha_1, \overline{h}^\alpha_2)\|_{H^1(0, T; \mathcal{H})}^2\Big).
\label{esLDi}
\end{align}
where $C_4\geq 1$ is a positive constant independent of $M$, $T$ and $\alpha$. By assumption (\textbf{A1}), the fact $(\widehat{\phi}^\alpha_i, \widehat{\psi}^\alpha_i)\in \mathcal{M}_T$ ($i=1,2$) and the Sobolev embedding theorem, we see that
\begin{align}
&\|F'(\widehat{\phi}^\alpha_1)-F'(\widehat{\phi}^\alpha_2)\|_{L^2(0,T; L^2(\Omega))}^2\nonumber\\
&\quad =\left\|\int_0^1 F''(\widehat{\phi}^\alpha_2+ s(\widehat{\phi}^\alpha_1-\widehat{\phi}^\alpha_2))(\widehat{\phi}^\alpha_1-\widehat{\phi}^\alpha_2) \d s \right\|_{L^2(0,T;L^2(\Omega))}^2\nonumber\\
&\quad \leq T\max_{|y|\leq C_2'M}|F''(y)|^2\|\widehat{\phi}^\alpha\|_{L^\infty(0,T;L^2(\Omega))}^2,\nonumber
\end{align}
\begin{align}
&\|F''(\widehat{\phi}^\alpha_1)\nabla \widehat{\phi}^\alpha_1-F''(\widehat{\phi}^\alpha_2)\nabla \widehat{\phi}^\alpha_2\|_{L^2(0,T;L^2(\Omega))}^2\nonumber\\
&\quad \leq 2\|(F''(\widehat{\phi}^\alpha_1)-F''(\widehat{\phi}^\alpha_2))\nabla \widehat{\phi}^\alpha_1\|_{L^2(0,T;L^2(\Omega))}^2
   +2\|F''(\widehat{\phi}^\alpha_2)\nabla \widehat{\phi}^\alpha\|_{L^2(0,T;L^2(\Omega))}^2\nonumber\\
&\quad \leq CT\max_{|y|\leq C_2'M}|F'''(y)|^2\|\nabla \widehat{\phi}^\alpha_1\|_{L^\infty(0,T;L^2(\Omega))}^2\|\widehat{\phi}^\alpha\|_{L^\infty(0,T;L^\infty(\Omega))}^2\nonumber\\
&\qquad + CT\max_{|y|\leq C_2'M}|F''(y)|^2\|\nabla \widehat{\phi}^\alpha\|_{L^\infty(0,T;L^2(\Omega))}^2\nonumber\\
&\quad\leq CT\Big(\max_{|y|\leq C_2'M}|F''(y)|^2+\max_{|y|\leq C_2'M}|F'''(y)|^2M^2\Big)\|\widehat{\phi}^\alpha\|_{L^\infty(0,T;H^2(\Omega))}^2,\nonumber
\end{align}
\begin{align}
&\|F''(\widehat{\phi}^\alpha_1)\partial_t\widehat{\phi}^\alpha_1
-F''(\widehat{\phi}^\alpha_2)\partial_t\widehat{\phi}^\alpha_2\|_{L^2(0,T;L^2(\Omega))}^2\nonumber\\
&\quad \leq 2\|(F''(\widehat{\phi}^\alpha_1)-F''(\widehat{\phi}^\alpha_2))\partial_t\widehat{\phi}^\alpha_1\|_{L^2(0,T;L^2(\Omega))}^2
   +2\|F''(\widehat{\phi}^\alpha_2)\widehat{\phi}^\alpha_t\|_{L^2(0,T;L^2(\Omega))}^2\nonumber\\
&\quad \leq CT\max_{|y|\leq C_2'M}|F'''(y)|^2\|\partial_t\widehat{\phi}^\alpha_1\|_{L^\infty(0,T;L^2(\Omega))}^2\|\widehat{\phi}^\alpha\|_{L^\infty(0,T;L^\infty(\Omega))}^2\nonumber\\
&\qquad + CT\max_{|y|\leq C_2'M}|F''(y)|^2\|\widehat{\phi}^\alpha_t\|_{L^\infty(0,T;L^2(\Omega))}^2\nonumber\\
&\quad\leq CT\max_{|y|\leq C_2'M}|F'''(y)|^2M^2\|\widehat{\phi}^\alpha\|_{L^\infty(0,T;H^2(\Omega))}^2\nonumber\\
&\qquad +CT\max_{|y|\leq C_2'M}|F''(y)|^2\|\widehat{\phi}^\alpha_t\|_{L^\infty(0,T;L^2(\Omega))}^2,\nonumber
\end{align}
and in a similar manner,
\begin{align}
&\|G'(\widehat{\psi}^\alpha_1)-G'(\widehat{\psi}^\alpha_2)\|_{L^2(0,T; L^2(\Gamma))}^2
\leq T\max_{|y|\leq C_2'M}|G''(y)|^2\|\widehat{\psi}^\alpha\|_{L^\infty(0,T;L^2(\Gamma))}^2,\nonumber
\end{align}
\begin{align}
&\|G''(\widehat{\psi}^\alpha_1)\nabla \widehat{\psi}^\alpha_1-G''(\widehat{\psi}^\alpha_2)\nabla \widehat{\psi}^\alpha_2\|_{L^2(0,T;L^2(\Gamma))}^2\nonumber\\
&\quad\leq CT\Big(\max_{|y|\leq C_2'M}|G''(y)|^2+\max_{|y|\leq C_2'M}|G'''(y)|^2M^2\Big)\|\widehat{\psi}^\alpha\|_{L^\infty(0,T;H^2(\Gamma))}^2.\nonumber
\end{align}
\begin{align}
&\|G''(\widehat{\psi}^\alpha_1)\partial_t\widehat{\psi}^\alpha_1
-G''(\widehat{\psi}^\alpha_2)\partial_t\widehat{\psi}^\alpha_2\|_{L^2(0,T;L^2(\Gamma))}^2\nonumber\\
&\quad\leq CT\max_{|y|\leq C_2'M}|G'''(y)|^2M^2\|\widehat{\psi}^\alpha\|_{L^\infty(0,T;H^2(\Gamma))}^2\nonumber\\
&\qquad + CT\max_{|y|\leq C_2'M}|G''(y)|^2\|\widehat{\psi}^\alpha_t\|_{L^\infty(0,T;L^2(\Gamma))}^2.\nonumber
\end{align}
Combining the above estimates with \eqref{esLDi}, we infer that
\begin{align}
&\|(\overline{\phi}^\alpha, \overline{\psi}^\alpha)\|_{L^\infty(0,T; \mathcal{V}^2)}^2+ \|(\overline{\phi}^\alpha, \overline{\psi}^\alpha)\|_{L^2(0,T; \mathcal{V}^3)}^2
 \nonumber\\
&\qquad + \|(\overline{\phi}^\alpha_t, \overline{\psi}^\alpha_t)\|_{L^\infty(0,T; \mathcal{H})}^2
 + \|(\overline{\phi}^\alpha_t, \overline{\psi}^\alpha_t)\|_{L^2(0,T; \mathcal{V}^1)}^2 \nonumber\\
&\quad \leq C_5e^{C_4T}\alpha^{-2}T P_2(M)\Big(\|(\widehat{\phi}^\alpha, \widehat{\psi}^\alpha)\|_{L^\infty(0,T; \mathcal{V}^2)}^2+\|(\widehat{\phi}^\alpha_t, \widehat{\psi}^\alpha_t)\|_{L^\infty(0,T;\mathcal{H})}^2\Big),\nonumber
\end{align}
where $C_5\geq 1$ is independent of $M$, $T$ as well as $\alpha$,
$P_2$ is a monotone increasing function given by
$$
P_2(M)=\max_{|y|\leq C_2'M}\big(|F''(y)|^2+|F'''(y)|^2M^2+|G''(y)|^2+|G'''(y)|^2M^2\big).
$$
Thus, for any $T$ satisfying
\begin{align}
0<T\leq  \min\left\{1,\ \frac{\alpha^{2}}{4C_5e^{C_4}P_2(M)}\right\},\label{T2}
\end{align}
we see that
$$
\|(\overline{\phi}^\alpha(t), \overline{\psi}^\alpha(t))\|_{\mathcal{Y}_T}\leq \frac12 \|(\widehat{\phi}^\alpha(t), \widehat{\psi}^\alpha(t))\|_{\mathcal{Y}_T}.
$$

In summary, we can conclude from \eqref{T1} and \eqref{T2} that for any given $\alpha\in (0,1]$, there exists a time $T_\alpha$ satisfying
\begin{align}
0<T_\alpha \leq  \min\left\{1,\ \frac{\alpha^2}{4C_3e^{C_1}P_1(M)},\ \frac{\alpha^{2}}{4C_5e^{C_4}P_2(M)}\right\},\nonumber
\end{align}
such that the mapping $\mathfrak{S}$ is a contraction from the set $\mathcal{M}_{T_\alpha}$ into itself. Thanks to the classical contraction mapping principle, we deduce that the regularized problem \eqref{ACH} admits a unique local strong solution $(\phi^\alpha, \psi^\alpha)\in \mathcal{Y}_{T_\alpha}$ for every $\alpha\in (0,1]$.
Finally, it is easy to verify that $\mu^\alpha\in L^2(0,T_\alpha; H^2(\Omega))$, $\mu^\alpha_\Gamma\in L^2(0, T_\alpha; H^2(\Gamma))$ by using the estimate \eqref{esLINZh} with the choices $h^\alpha_1=F'(\phi^\alpha)$ and  $h^\alpha_2=G'(\psi^\alpha)$. Besides, the continuity property $(\phi^\alpha, \psi^\alpha)\in C([0,T_\alpha]; \mathcal{V}^2)$ simply follows from the interpolation theorem (see, e.g., \cite{SI}).

The proof of Proposition \ref{exeACH} is complete.
\end{proof}

\section{Global Well-posedness}
\setcounter{equation}{0}

This section is devoted to the proof of Theorems \ref{main1} and \ref{main1a}.

\subsection{A priori estimates}
The crucial step is to derive some uniform global-in-time \emph{a priori} estimates for the regularized problem \eqref{ACH} with respective to the parameters $\alpha$ and $\kappa$.
Within this subsection,  we always assume that
$$
\kappa\in (0,\bar{\kappa}]\quad\text{and}\quad  \alpha\in (0,1],
$$
where $\bar{\kappa}>0$ is an arbitrary but fixed positive number. \medskip

\textbf{First estimate: basic energy estimate}. Testing the first and fourth equations in the regularized system \eqref{ACH}
by $\mu^\alpha$ and $\mu_\Gamma^\alpha$ respectively, adding the resultants together, we have for $t\geq 0$,
\begin{align}
&\frac{\mathrm{d}}{\mathrm{d}t}E(\phi^\alpha(t), \psi^\alpha(t))
+\|\nabla\mu^\alpha\|_{L^2(\Omega)}^2
+\|\nabla_\Gamma\mu_\Gamma^\alpha \|_{L^2(\Gamma)}^2+\alpha\|(\phi^\alpha_t, \psi^\alpha_t)\|_{\mathcal{H}}^2=0,
\label{BELb}
\end{align}
where
\begin{align}
E(\phi^\alpha(t), \psi^\alpha(t))
&=\frac12\int_{\Omega} |\nabla \phi^\alpha(t)|^2\, \d x
+ \frac12 \int_{\Gamma}(\kappa|\nabla_{\Gamma} \psi^\alpha(t)|^2 + |\psi^\alpha(t)|^2)\, \d S\nonumber\\
&\quad +\int_\Omega F(\phi^\alpha(t))\, \d x
+\int_\Gamma G(\psi^\alpha(t))\,\d S.
\label{Ealp}
\end{align}
Integrating \eqref{BELb} with respect to time, from \eqref{Ealp} and assumption (\textbf{A2}) we deduce the following uniform estimate
\begin{align}
&\frac12\int_{\Omega} |\nabla \phi^\alpha(t)|^2\, \d x
+ \frac12 \int_{\Gamma}(\kappa|\nabla_{\Gamma} \psi^\alpha(t)|^2 + |\psi^\alpha(t)|^2)\, \d S\nonumber\\
&\qquad + \int_0^{t}\Big(\|\nabla\mu^\alpha(\tau)\|_{L^2(\Omega)}^2
+\|\nabla_\Gamma\mu_\Gamma^\alpha(\tau) \|_{L^2(\Gamma)}^2\Big)\,\d \tau
+\alpha\int_0^{t}\|(\phi^\alpha_t(\tau), \psi^\alpha_t(\tau))\|_{\mathcal{H}}^2 \d \tau\nonumber\\
&\quad \leq E(\phi_0, \psi_0)+C_F|\Omega|+C_G|\Gamma|,\quad \forall\, t\geq 0.\label{esunif1}
\end{align}
On the other hand, using the mass conservation property $\langle\phi^\alpha_t\rangle_\Omega=\langle\psi^\alpha_t\rangle_\Gamma=0$, we infer from the equations for $\phi^\alpha$ and $\psi^\alpha$ that
\begin{align}
\|\phi^\alpha_t\|_{(H^1(\Omega))^*}=\|\nabla \mu^\alpha\|_{L^2(\Omega)},\quad  \|\psi^\alpha_t\|_{(H^1(\Gamma))^*}
=\|\nabla_\Gamma\mu_\Gamma^\alpha \|_{L^2(\Gamma)}.\label{phit}
\end{align}
Hence, \eqref{esunif1} also yields
\begin{align}
\int_0^{+\infty}\Big(\|\phi^\alpha_t(t)\|_{(H^1(\Omega))^*}^2
+\|\psi^\alpha_t(t)\|_{(H^1(\Gamma))^*}^2\Big)\, \d t
\leq E(\phi_0, \psi_0)+C_F|\Omega|+C_G|\Gamma|. \label{phiuni}
\end{align}

\begin{remark}\label{remie}
When $(\phi_0, \psi_0)\in \mathcal{V}^1$, it follows from the growth assumption (\textbf{A3}) and the Sobolev embedding theorem that the initial energy  satisfies $E(\phi_0, \psi_0)\leq C(\|(\phi_0, \psi_0)\|_{\mathcal{V}^1})$.  On the other hand, when $(\phi_0, \psi_0)\in \mathcal{V}^2$, the Sobolev embedding theorem and the assumption (\textbf{A1}) easily imply $E(\phi_0, \psi_0)\leq C(\|(\phi_0, \psi_0)\|_{\mathcal{V}^2})$. For both cases, the estimate on $E(\phi_0, \psi_0)$ is independent of $\alpha\in (0,1]$. Besides, these bounds may depend on the upper bound $\bar{\kappa}$ but are independent of $\kappa$.
\end{remark}

\textbf{Second estimate: on the time derivatives $(\phi_t^\alpha, \psi_t^\alpha)$}.
Differentiating the system \eqref{ACH} with respect to time, we obtain
\begin{equation}
\left\{
\begin{aligned}
&\phi_{tt}^\alpha=\Delta \mu^\alpha_t, \quad \text{with}\ \ \mu^\alpha_t=-\Delta \phi^\alpha_t +\alpha \phi_{tt}^\alpha +F''(\phi^\alpha)\phi^\alpha_t,
&\text{in}\ \Omega\times(0,T),\\
&\partial_\mathbf{n} \mu_t^\alpha=0,
&\text{on}\ \Gamma\times(0,T),\\
&\phi^\alpha_t|_\Gamma=\psi^\alpha_t,
&\text{on}\ \Gamma\times(0,T),\\
&\psi_{tt}^\alpha=\Delta_\Gamma (\mu_\Gamma^\alpha)_t,
&\text{on}\ \Gamma\times(0,T),\\
&\quad \text{with}\ \ (\mu_\Gamma^\alpha)_t
= -\kappa\Delta_\Gamma \psi^\alpha_t +  \psi^\alpha_t+ \alpha\psi^\alpha_{tt}
+\partial_\mathbf{n}\phi_t^\alpha + G''(\psi^\alpha)\psi^\alpha_t,
&\text{on}\ \Gamma\times(0,T),\\
&\phi^\alpha_t|_{t=0}=\Delta \mu^\alpha(0),
&\text{in}\ \Omega,\\
&\psi^\alpha_t|_{t=0}=\Delta_\Gamma \mu_\Gamma^\alpha(0),
&\text{on}\ \Gamma.
\end{aligned}
\label{ACHt}
\right.
\end{equation}
By an analogous observation like in \cite{MZ05}, we see that for the regularized problem \eqref{ACH}, if $\phi^\alpha (t)$ is known for some $t$, then the value of the chemical potential $\mu^\alpha(t)$ can be uniquely determined by solving the linear elliptic problem (with $\alpha>0$):
\begin{equation}
\left\{
\begin{aligned}
&\mu^\alpha(t)-\alpha\Delta \mu^\alpha(t)=-\Delta \phi^\alpha(t)+F'(\phi^\alpha(t)),
&\text{in}\ \Omega,\\
&\partial_\mathbf{n}\mu^\alpha(t)=0,
&\text{on}\ \Gamma.
\label{aepmu}
\end{aligned}
\right.
\end{equation}
In a similar manner, the surface chemical potential $\mu^\alpha_\Gamma(t)$ can be uniquely determined by solving \
\begin{align}
 &\mu_\Gamma^\alpha(t)-\alpha \Delta_\Gamma \mu_\Gamma^\alpha(t)
  =-\kappa \Delta_\Gamma \psi^\alpha(t)+\psi^\alpha(t)+\partial_\mathbf{n}\phi^\alpha(t)
  +G'(\psi^\alpha(t)),\quad \text{on}\ \Gamma.
  \label{aepmua}
\end{align}
Now we prepare initial data for the auxiliary system \eqref{ACHt}. Assume $(\phi_0, \psi_0)\in \mathcal{V}^3$. Taking $t=0$ in \eqref{aepmu}, we infer from assumption (\textbf{A1}) and the Sobolev embedding theorem that
\begin{align}
&\|\nabla \mu^\alpha(0)\|_{L^2(\Omega)}\leq \|\nabla(-\Delta \phi_0+F'(\phi_0))\|_{L^2(\Omega)}\leq C,\nonumber\\
&\alpha^\frac12\|\Delta \mu^\alpha(0)\|_{L^2(\Omega)}\leq \|\nabla(-\Delta \phi_0+F'(\phi_0))\|_{L^2(\Omega)}\leq C,\nonumber
\end{align}
where the constant $C$ only depends on $\|\phi_0\|_{H^3(\Omega)}$ but is independent of $\alpha$ and $\kappa$.
Similarly, we have
\begin{align}
&\|\nabla_\Gamma\mu_\Gamma^\alpha(0)\|_{L^2(\Gamma)}\leq \|\nabla_\Gamma(-\kappa \Delta_\Gamma \psi_0+\psi_0+\partial_\mathbf{n}\phi_0+G'(\psi_0))\|_{L^2(\Gamma)}\leq C,\nonumber\\
&\alpha^\frac12\|\Delta_\Gamma\mu_\Gamma^\alpha(0)\|_{L^2(\Gamma)}\leq \|\nabla_\Gamma(-\kappa \Delta_\Gamma \psi_0+\psi_0+\partial_\mathbf{n}\phi_0+G'(\psi_0))\|_{L^2(\Gamma)}\leq C,\nonumber
\end{align}
where the constant $C$ only depends on $\|(\phi_0,\psi_0)\|_{\mathcal{V}^3}$ and the constant  $\bar{\kappa}$, but is independent of $\alpha$ and $\kappa$.

Next, testing the first equation in \eqref{ACHt} by $(A^0_\Omega)^{-1} \phi^\alpha_t$ and the fourth equation by $(A_\Gamma^0)^{-1}\psi^\alpha_t$, adding the resultants together, we deduce from assumption (\textbf{A2}) that
\begin{align}
&\frac{1}{2}\frac{\d}{\d t}\Big(\|\phi^\alpha_t\|_{(H^1(\Omega))^*}^2+ \alpha \|\phi^\alpha_t\|_{L^2(\Omega)}^2+\|\psi^\alpha_t\|_{(H^1(\Gamma))^*}^2+\alpha \|\psi^\alpha_t\|_{L^2(\Gamma)}^2\Big)\nonumber\\
&\qquad + \|\nabla \phi^\alpha_t\|_{L^2(\Omega)}^2+ \kappa\|\nabla_\Gamma \psi^\alpha_t\|_{L^2(\Gamma)}^2+\|\psi^\alpha_t\|_{L^2(\Gamma)}^2\nonumber\\
&\quad =-\int_\Omega F''(\phi^\alpha)(\phi^\alpha_t)^2\, \d x
-\int_\Gamma G''(\psi^\alpha)(\psi^\alpha_t)^2\,\d S\nonumber\\
&\quad \leq \widetilde{C}_F\|\phi^\alpha_t\|_{L^2(\Omega)}^2
+\widetilde{C}_G\|\psi^\alpha_t\|_{L^2(\Gamma)}^2.
\label{achhight}
\end{align}
It follows from the facts $\langle\phi^\alpha_t\rangle_\Omega=\langle\psi^\alpha_t\rangle_\Gamma=0$, the trace theorem, the interpolation inequality, the Poincar\'e and Young inequalities that for a.e. $t>0$,
\begin{align}
\|\phi^\alpha_t\|_{L^2(\Omega)}^2
&\leq C\|\phi^\alpha_t\|_{(H^1(\Omega))^*}\|\nabla \phi^\alpha_t\|_{L^2(\Omega)}\nonumber\\
&\leq \frac{1}{4\widetilde{C}_F} \|\nabla \phi^\alpha_t\|_{L^2(\Omega)}^2
+C\widetilde{C}_F\|\phi^\alpha_t\|_{(H^1(\Omega))^*}^2,
\label{rett1}\\
\|\psi^\alpha_t\|_{L^2(\Gamma)}^2
&\leq C\|\phi^\alpha_t\|_{H^{\frac12+\epsilon}(\Omega)}^2\nonumber\\
&\leq C\|\phi^\alpha_t\|_{(H^1(\Omega))^*}^{\frac12-\epsilon}\|\nabla \phi^\alpha_t\|_{L^2(\Omega)}^{\frac32+\epsilon}\nonumber\\
&\leq  \frac{1}{4\widetilde{C}_G} \|\nabla \phi^\alpha_t\|_{L^2(\Omega)}^2
+C\widetilde{C}_G\|\phi^\alpha_t\|_{(H^1(\Omega))^*}^2,
\label{rett2}
\end{align}
where $\epsilon\in (0,\frac12)$ and the constant $C$ is independent of $\alpha$ and $\kappa$.
 Integrating \eqref{achhight} with respect to time, we infer from the above estimates and \eqref{phiuni} that
\begin{align}
&\|\phi^\alpha_t(t)\|_{(H^1(\Omega))^*}^2+ \alpha \|\phi^\alpha_t(t)\|_{L^2(\Omega)}^2+\|\psi^\alpha_t(t)\|_{(H^1(\Gamma))^*}^2+\alpha \|\psi^\alpha_t(t)\|_{L^2(\Gamma)}^2\nonumber\\
&\qquad +\int_0^t\Big(\|\nabla \phi^\alpha_t(\tau)\|_{L^2(\Omega)}^2+ \kappa\|\nabla_\Gamma \psi^\alpha_t(\tau)\|_{L^2(\Gamma)}^2+\|\psi^\alpha_t(\tau)\|_{L^2(\Gamma)}^2\Big)\, \d \tau \nonumber\\
&\quad \leq \|\nabla \mu^\alpha(0)\|_{L^2(\Omega)}^2+\alpha\|\Delta \mu^\alpha(0)\|_{L^2(\Omega)}^2
+\|\nabla_\Gamma\mu_\Gamma^\alpha(0)\|_{L^2(\Gamma)}^2
\nonumber\\
&\qquad +\alpha\|\Delta_\Gamma\mu_\Gamma^\alpha(0)\|_{L^2(\Gamma)}^2+C\big(E(\phi_0, \psi_0)+C_F|\Omega|+C_G|\Gamma|\big)\nonumber\\
&\quad \leq C, \quad \forall\, t\geq 0,\label{estt}
\end{align}
where the constant $C$ depends on $\|(\phi_0, \psi_0)\|_{\mathcal{V}^3}$, $C_F$, $\widetilde{C}_F$, $C_G$, $\widetilde{C}_G$, $\Omega$, $\Gamma$ and $\bar{\kappa}$, but is independent of $\alpha$ and $\kappa$.

\medskip

\textbf{Third estimate: on the normal derivative $\partial_\mathbf{n} \phi^\alpha$}.
By the trace theorem, it follows that
\begin{align}
\|\partial_\mathbf{n}\phi^\alpha\|_{H^{-\frac12}(\Gamma)}\leq C\|\partial_\mathbf{n}\phi^\alpha\|_{L^2(\Gamma)}
\leq  C\|\phi^\alpha\|_{H^r(\Omega)},\quad \text{for some }r\in \left(\frac32, 2\right),
\label{espncla}
\end{align}
where $C>0$ only depends on $\Omega$. On the other hand, we also have
\begin{align}
\left| \int_\Gamma \partial_\mathbf{n} \phi^\alpha\, \d S \right|
&=\left|\langle\partial_\mathbf{n} \phi^\alpha,1\rangle_{H^{-\frac12}(\Gamma), H^{\frac12}(\Gamma)}\right|
\leq |\Gamma|^\frac12\|\partial_\mathbf{n} \phi^\alpha\|_{H^{-\frac12}(\Gamma)}.
\label{inttr}
\end{align}
The above estimates on the normal derivative $\partial_\mathbf{n} \phi^\alpha$ will be enough for the case with surface diffusion for any given $\kappa>0$. However, they are not sufficient in order to pass to the limit as $\kappa\to 0^+$. For this purpose, we recall the generalized first Green's formula in Section 3.2 and in particular the inequality \eqref{cZ1} such that for $\phi^\alpha\in H^1(\Omega)$ with $\Delta \phi^\alpha\in (H^1(\Omega))^*$, it holds
\begin{align}
\|\partial_\mathbf{n}\phi^\alpha\|_{H^{-\frac12}(\Gamma)}
&\leq c_\mathcal{R}\big(\|\nabla \phi^\alpha\|_{L^2(\Omega)}+\|\Delta \phi^\alpha\|_{(H^1(\Omega))^*}\big).
\label{ntres}
\end{align}
As a consequence, we infer from \eqref{inttr} that
\begin{align}
\left|\int_\Gamma \partial_\mathbf{n} \phi^\alpha\, \d S\right|&\leq c_\mathcal{R}|\Gamma|^\frac12\big(\|\nabla \phi^\alpha\|_{L^2(\Omega)}
+\|\Delta \phi^\alpha\|_{(H^1(\Omega))^*}\big).
\label{pphig}
\end{align}

\medskip

\textbf{Fourth estimate: on the mean value of chemical potentials}.
We distinguish several cases according to the possible dependence on $\kappa$.

(a) Assume that $(\phi_0, \psi_0)\in \mathcal{V}^1$ and the growth assumption (\textbf{A3}) is satisfied.
Then by  \eqref{espncla}--\eqref{inttr} and the Sobolev embedding theorem $(d=2,3)$, we obtain the following estimates
\begin{align}
|\langle \mu^\alpha \rangle_\Omega|
&\leq |\Omega|^{-1}|\Gamma| |\langle \partial_\mathbf{n}\phi^\alpha\rangle_\Gamma|+|\langle F'(\phi^\alpha)\rangle_\Omega|\nonumber\\
&\leq C\big(1+\|\phi^\alpha\|_{H^r(\Omega)}+\|\phi^\alpha\|_{H^1(\Omega)}^{p+1}\big),
\label{memu1}
\end{align}
\begin{align}
|\langle \mu^\alpha_\Gamma\rangle_\Gamma|
&\leq |\langle\psi^\alpha \rangle_\Gamma|+|\langle \partial_\mathbf{n}\phi^\alpha \rangle_\Gamma|+ |\langle G'(\psi^\alpha)\rangle_\Gamma|
\nonumber\\
& \leq |\langle\psi_0\rangle_\Gamma|+C\big(1+\|\phi^\alpha\|_{H^r(\Omega)}
+\|\psi^\alpha\|_{H^1(\Gamma)}^{q+1}\big),
\label{memug1}
\end{align}
where $r\in (\frac32, 2)$ and $C$ is independent of $\alpha$ and $\kappa$. By the lower-order estimate \eqref{esunif1}, we get
\begin{align}
|\langle \mu^\alpha \rangle_\Omega|+ |\langle \mu^\alpha_\Gamma\rangle_\Gamma|
\leq C(1+\|\phi^\alpha\|_{H^r(\Omega)}),\quad \forall\, t\geq 0,\label{memulo}
\end{align}
where $C$ depends on $\|(\phi_0, \psi_0)\|_{\mathcal{V}^1}$, $C_F$, $\widehat{C}_F$, $C_G$, $\widehat{C}_G$, $\Omega$, $\Gamma$, $\bar{\kappa}$, but is independent of $\alpha$ and $\kappa$.

(b) Assume that $(\phi_0, \psi_0)\in \mathcal{V}^1$ and (\textbf{A3}) is satisfied.
Moreover, we suppose the additional condition $c_\mathcal{R}|\Gamma|^\frac12|\Omega|^{-1}<1$, where $c_\mathcal{R}$ is the constant in \eqref{ntres} (cf. \eqref{cZ1}). Then from \eqref{pphig}, we deduce that
\begin{align}
\left|\int_\Omega \mu^\alpha \, \d x\right|
&\leq \left|\int_\Gamma \partial_\mathbf{n} \phi^\alpha\, \d S\right|
+\left|\int_\Omega F'(\phi^\alpha)\, \d x\right|\nonumber\\
&\leq c_\mathcal{R}|\Gamma|^\frac12\big(\|\nabla \phi^\alpha\|_{L^2(\Omega)}
+\|\Delta \phi^\alpha\|_{(H^1(\Omega))^*}\big)
+\left|\int_\Omega F'(\phi^\alpha)\, \d x\right|\nonumber\\
&\leq c_\mathcal{R}|\Gamma|^\frac12\big(\|\mu^\alpha\|_{(H^1(\Omega))^*}
+\alpha\|\phi^\alpha_t\|_{(H^1(\Omega))^*}
+\|F'(\phi^\alpha)\|_{(H^1(\Omega))^*}\big)\nonumber\\
&\quad + c_\mathcal{R}|\Gamma|^\frac12\|\nabla \phi^\alpha\|_{L^2(\Omega)} +\left|\int_\Omega F'(\phi^\alpha)\, \d x\right|\nonumber\\
&\leq c_\mathcal{R}|\Gamma|^\frac12\big(\|\nabla (A_\Omega^0)^{-1}(\mu^\alpha-\langle \mu^\alpha \rangle_\Omega)\|_{L^2(\Omega)}+ |\langle \mu^\alpha \rangle_\Omega|\big)\nonumber\\
&\quad +C(\alpha\|\phi^\alpha_t\|_{(H^1(\Omega))^*}
+\|F'(\phi^\alpha)\|_{(H^1(\Omega))^*})\nonumber\\
&\quad + c_\mathcal{R}|\Gamma|^\frac12\|\nabla \phi^\alpha\|_{L^2(\Omega)} +\left|\int_\Omega F'(\phi^\alpha)\, \d x\right|,\nonumber
\end{align}
which together with the lower-order estimate \eqref{esunif1} and Poincar\'e's inequality yields
\begin{align}
\left|\langle\mu^\alpha\rangle_\Omega\right|&\leq C|\Omega|^{-1}\big(1-c_\mathcal{R}|\Gamma|^\frac12|\Omega|^{-1}\big)^{-1}\big(\|\nabla \mu^\alpha\|_{L^2(\Omega)}+1\big),
\label{memu1b}
\end{align}
where $C$ depends on $\|(\phi_0, \psi_0)\|_{\mathcal{V}^1}$, $C_F$, $\widehat{C}_F$, $C_G$, $\Omega$, $\Gamma$, $\bar{\kappa}$, but is independent of $\alpha$ and $\kappa$. The uniform estimate \eqref{memu1b} also implies
\begin{align}
|\langle \partial_\mathbf{n} \phi^\alpha\rangle_\Gamma|
& \leq |\Gamma|^{-1}|\Omega|\big(\left|\langle\mu^\alpha\rangle_\Omega\right|+ |\langle F'(\phi^\alpha)\rangle_\Omega|\big)\nonumber\\
&\leq C\big(\|\nabla \mu^\alpha\|_{L^2(\Omega)}+1\big),
\label{meanpn1}
\end{align}
and
\begin{align}
\|\partial_\mathbf{n}\phi^\alpha \|_{H^{-\frac12}(\Gamma)}
&\leq c_\mathcal{R}(\|\nabla \phi^\alpha\|_{L^2(\Omega)}
+\|\Delta \phi^\alpha\|_{(H^1(\Omega))^*})\nonumber\\
&\leq C\big(\|\nabla \mu^\alpha\|_{L^2(\Omega)}+1\big).
\label{espnh-12}
\end{align}
Furthermore, we deduce from \eqref{esunif1} and \eqref{memu1b} that
\begin{align}
|\langle \mu^\alpha_\Gamma\rangle_\Gamma|
&\leq |\langle\psi^\alpha \rangle_\Gamma|+|\langle \partial_\mathbf{n}\phi^\alpha \rangle_\Gamma|+ |\langle G'(\psi^\alpha)\rangle_\Gamma|
\nonumber\\
& \leq C\big(\|\nabla \mu^\alpha\|_{L^2(\Omega)}+1\big),\label{memug1b}
\end{align}
where $C$ depends on $\|(\phi_0, \psi_0)\|_{\mathcal{V}^1}$, $C_F$, $\widehat{C}_F$, $C_G$, $\widehat{C}_G$, $\Omega$, $\Gamma$, $\bar{\kappa}$, but is independent of $\alpha$ and $\kappa$.

(c) Assume that $(\phi_0, \psi_0)\in \mathcal{V}^3$ and assumption (\textbf{A2}) is satisfied. Testing the equations for $\mu^\alpha$, $\mu_\Gamma^\alpha$ in \eqref{ACH} by $\overline{\phi^\alpha}:=\phi^\alpha-\langle \phi^\alpha\rangle_\Omega=\phi^\alpha-\langle \phi_0\rangle_\Omega$, $\overline{\psi^\alpha}:=\psi^\alpha-\langle\psi^\alpha\rangle_\Gamma
=\psi^\alpha-\langle\psi_0\rangle_\Gamma$, respectively, adding the resultants together, we obtain
\begin{align}
&\|\nabla \phi^\alpha(t)\|_{L^2(\Omega)}^2
+\kappa\|\nabla_\Gamma \psi^\alpha(t)\|_{L^2(\Gamma)}^2
+\|\overline{\psi^\alpha}(t)\|_{L^2(\Gamma)}^2
+\int_\Omega F'(\phi^\alpha(t))\overline{\phi^\alpha}(t)\, \d x
\nonumber\\
&\qquad
+ \int_\Gamma G'(\psi^\alpha(t))\overline{\psi^\alpha}(t)\, \d S\nonumber\\
&\quad =\int_\Omega (\mu^\alpha(t)-\langle \mu^\alpha(t)\rangle_\Omega)\overline{\phi^\alpha}(t)\, \d x
+\int_\Gamma (\mu_\Gamma^\alpha(t)-\langle\mu_\Gamma^\alpha(t)\rangle_\Gamma)\overline{\psi^\alpha}\, \d S\nonumber\\
&\qquad -\alpha\int_\Omega \phi_t^\alpha \overline{\phi^\alpha}(t)\, \d x
-\alpha\int_\Gamma \psi^\alpha_t \overline{\psi^\alpha}\, \d S
+(\langle\psi_0\rangle_\Gamma-\langle \phi_0\rangle_\Omega)\int_\Gamma\partial_\mathbf{n}\phi^\alpha(t)\, \d S \nonumber\\
&\quad \leq \|\mu^\alpha(t)-\langle \mu^\alpha(t)\rangle_\Omega\|_{L^2(\Omega)}\|\overline{\phi^\alpha}(t)\|_{L^2(\Omega)}
\nonumber\\
&\qquad +\|\mu_\Gamma^\alpha(t)-\langle\mu_\Gamma^\alpha(t)\rangle_\Gamma\|_{L^2(\Gamma)}
\|\overline{\psi^\alpha}\|_{L^2(\Gamma)}\nonumber\\
&\qquad + \alpha\| \phi_t^\alpha\|_{(H^1(\Omega))^*}\| \overline{\phi^\alpha}(t)\|_{H^1(\Omega)}
+\alpha\|\psi^\alpha_t\|_{(H^1(\Gamma))^*}\|\overline{\psi^\alpha}\|_{H^1(\Gamma)}\nonumber
\\
&\qquad +(|\langle\psi_0\rangle_\Gamma|+|\langle \phi_0\rangle_\Omega|)\left|\int_\Gamma\partial_\mathbf{n}\phi^\alpha(t)\, \d S\right|\nonumber\\
&\quad \leq C\|\nabla \mu^\alpha(t)\|_{L^2(\Omega)}\|\phi^\alpha(t)\|_{H^1(\Omega)} +C\|\nabla_\Gamma\mu_\Gamma^\alpha(t)\|_{L^2(\Gamma)}
\|\psi^\alpha(t)\|_{H^1(\Gamma)}\nonumber\\
&\qquad +(|\langle\psi_0\rangle_\Gamma|+|\langle \phi_0\rangle_\Omega|)\left|\int_\Gamma\partial_\mathbf{n}\phi^\alpha(t)\, \d S\right|.
\label{FFL2}
\end{align}
The above inequality together with \eqref{esunif1}--\eqref{phit}, the higher-order estimate \eqref{estt} and \eqref{espncla}--\eqref{inttr} implies
\begin{align}
&\int_\Omega F'(\phi^\alpha(t))\overline{\phi^\alpha}(t) \d x + \int_\Gamma G'(\psi^\alpha(t))\overline{\psi^\alpha}(t) \d S \leq C \big(1+\|\phi^\alpha\|_{H^r(\Omega)}\big),\quad \forall\, t\geq 0.
\label{FGm}
\end{align}
On the other hand, recalling  (\textbf{A2}) and the transformations \eqref{tF}--\eqref{tG}, we deduce the following simple facts
\begin{align}
&\widetilde{F}(0)=\widetilde{F}'(0)=\tilde{G}(0)=\widetilde{G}'(0)=0,\\
&\widetilde{F}''(y)\geq 1,\quad  \widetilde{G}''(y)\geq 1,\quad \forall\, y\in \mathbb{R},\label{convex}\\
&y\widetilde{F}'(y)\geq 0,\quad  y\widetilde{G}'(y)\geq 0, \quad \widetilde{F}'(y)\widetilde{G}'(y)\geq 0,
\quad \forall\, y\in \mathbb{R}.
\label{FFG}
\end{align}
The convexity condition \eqref{convex} easily yields (see, e.g., \cite{MZ05})
\begin{align*}
&\frac12|\widetilde{F}'(y)|(1+|y|)\leq \widetilde{F}'(y)(y-\langle \phi_0\rangle_\Omega)+C(F,\langle \phi_0\rangle_\Omega),\\
&\frac12|\widetilde{G}'(y)|(1+|y|)\leq \widetilde{G}'(y)(y-\langle \psi_0\rangle_\Gamma)+C(G,\langle \psi_0\rangle_\Gamma).
\end{align*}
Therefore, we infer from the above relations and \eqref{FGm} that
\begin{align}
|\langle F'(\phi^\alpha(t))\rangle_\Omega|+ |\langle G'(\psi^\alpha(t))\rangle_\Gamma|\leq C \big(1+\|\phi^\alpha\|_{H^r(\Omega)}\big),
\quad \forall\, t\geq 0.
\label{FGm1}
\end{align}
Then by \eqref{espncla}, \eqref{inttr}  and  \eqref{FGm1}, we obtain
\begin{align}
|\langle \mu^\alpha \rangle_\Omega|+ |\langle \mu^\alpha_\Gamma\rangle_\Gamma|
\leq C(1+\|\phi^\alpha\|_{H^r(\Omega)}),\quad \forall\, t\geq 0,\label{memuhi}
\end{align}
where the constant $C>0$ depends on $\|(\phi_0, \psi_0)\|_{\mathcal{V}^3}$, $C_F$, $\widetilde{C}_F$, $C_G$, $\widetilde{C}_G$, $\Omega$, $\Gamma$ and $\bar \kappa$, but is independent of $\alpha$.

\medskip
\textbf{Fifth estimate: further estimate on chemical potentials}.
Again we distinguish several cases according to the possible dependence on $\kappa$.

(a) Assume that $(\phi_0, \psi_0)\in \mathcal{V}^1$ and the growth assumption (\textbf{A3}) is satisfied.
Combining  \eqref{esunif1}, \eqref{memulo} and Poincar\'e's inequality, we deduce that
\begin{align}
&\|\mu^\alpha(t)\|_{L^2(0,T; H^1(\Omega))}
+\|\mu^\alpha_\Gamma(t)\|_{L^2(0,T; H^1(\Gamma))}\leq C(1+T)\big(1+\|\phi^\alpha\|_{L^2(0,T;H^r(\Omega))}\big),
\label{esL2H1mulo}
\end{align}
where the constant $C>0$ depends on $\|(\phi_0, \psi_0)\|_{\mathcal{V}^1}$, $C_F$, $\widehat{C}_F$, $C_G$, $\widehat{C}_G$, $\Omega$, $\Gamma$, $\bar{\kappa}$, but is independent of $\alpha$ and $\kappa$.

(b) Assume that $(\phi_0, \psi_0)\in \mathcal{V}^1$, (\textbf{A3}) is satisfied and  $c_\mathcal{R}|\Gamma|^\frac12|\Omega|^{-1}<1$ holds.
It follows from \eqref{esunif1}, \eqref{memu1b} and \eqref{memug1b} that
\begin{align}
&\|\mu^\alpha(t)\|_{L^2(0,T; H^1(\Omega))}
+\|\mu^\alpha_\Gamma(t)\|_{L^2(0,T; H^1(\Gamma))}\leq C(1+T),
\label{esL2H1mulok}
\end{align}
where the constant $C>0$ depends on $\|(\phi_0, \psi_0)\|_{\mathcal{V}^1}$, $C_F$, $\widehat{C}_F$, $C_G$, $\widehat{C}_G$, $\Omega$, $\Gamma$, $\bar{\kappa}$ and $c_\mathcal{R}$, but is independent of $\alpha$ and $\kappa$.

If we further assume $(\phi_0, \psi_0)\in \mathcal{V}^3$, then from \eqref{estt}, \eqref{memu1b} and \eqref{memug1b} we obtain
\begin{align}
&\|\mu^\alpha(t)\|_{L^\infty(0,T; H^1(\Omega))}
+\|\mu^\alpha_\Gamma(t)\|_{L^\infty(0,T; H^1(\Gamma))}\leq C,
\label{esL2H1mulokb}\\
&\|\mu^\alpha(t)\|_{L^2(0,T; H^3(\Omega))}
+\|\mu^\alpha_\Gamma(t)\|_{L^2(0,T; H^2(\Gamma))}\leq C(1+T),
\label{esL2H1mulokba}
\end{align}
where the constant $C>0$ may depend on $\|(\phi_0, \psi_0)\|_{\mathcal{V}^3}$, $C_F$, $\widetilde{C}_F$, $\widehat{C}_F$, $C_G$, $\widetilde{C}_G$, $\widehat{C}_G$, $\Omega$, $\Gamma$, $\bar{\kappa}$ and $c_\mathcal{R}$, but is independent of $\alpha$ and $\kappa$.

(c) Assume that $(\phi_0, \psi_0)\in \mathcal{V}^3$ and assumption (\textbf{A2}) is satisfied. We infer from \eqref{estt}, \eqref{memuhi} and Poincar\'e's inequality that
\begin{align}
&\|\mu^\alpha(t)\|_{H^1(\Omega)}+\|\mu^\alpha_\Gamma(t)\|_{H^1(\Gamma)}
\leq C(1+\|\phi^\alpha(t)\|_{H^r(\Omega)}),
\quad \forall\, t\geq 0,
\label{esL2H1muhi}
\end{align}
and
\begin{align}
&\|\mu^\alpha(t)\|_{L^2(0,T; H^3(\Omega))}+ \|\mu^\alpha_\Gamma(t)\|_{L^2(0,T; H^3(\Gamma))}
\leq C(1+T)\big(1+\|\phi^\alpha\|_{L^\infty(0,T;H^r(\Omega))}\big),
\label{esL2H23muhi}
\end{align}
where the constant $C>0$ depends on $\|(\phi_0, \psi_0)\|_{\mathcal{V}^3}$, $C_F$, $\widetilde{C}_F$, $C_G$, $\widetilde{C}_G$, $\Omega$, $\Gamma$ and $\kappa$, but is independent of $\alpha$.

\medskip
\textbf{Sixth estimate: $\mathcal{V}^2$-estimate for $(\phi^\alpha, \psi^\alpha)$}.
Using the facts $\langle\phi^\alpha_t\rangle_\Omega=\langle\psi^\alpha_t\rangle_\Gamma=0$,
we deduce that
\begin{align}
-(A^0_\Omega)^{-1}\phi_t^\alpha
&=\mu^\alpha-\langle\mu^\alpha\rangle_\Omega
=\mu^\alpha+\frac{|\Gamma|}{|\Omega|}\langle\partial_\mathbf{n}\phi^\alpha\rangle_\Gamma-
\langle F'(\phi^\alpha)\rangle_\Omega,\nonumber
\end{align}
and
\begin{align}
-(A^0_\Gamma)^{-1}\psi^\alpha_t
&=\mu^\alpha_\Gamma
-\langle\mu^\alpha_\Gamma\rangle_\Gamma\nonumber\\
&=\mu^\alpha_\Gamma-\langle\psi^\alpha\rangle_\Gamma
-\langle\partial_\mathbf{n}\phi^\alpha\rangle_\Gamma
-\langle G'(\psi^\alpha)\rangle_\Gamma.\nonumber
\end{align}
Therefore, the evolution system \eqref{ACH} can be transformed into the following nonlinear elliptic boundary value problem:
\begin{equation}
\left\{
\begin{aligned}
&-\Delta\phi^\alpha =-[\alpha+ (A^0_\Omega)^{-1}]\phi^\alpha_t
-\frac{|\Gamma|}{|\Omega|}\langle\partial_\mathbf{n}\phi^\alpha\rangle_\Gamma
-F'(\phi^\alpha)+\langle F'(\phi^\alpha)\rangle_\Omega,
&\text{in}\ \Omega,\\
&\phi^\alpha|_\Gamma=\psi^\alpha,
&\text{on}\ \Gamma,\\
&-\kappa\Delta_\Gamma \psi^\alpha+\psi^\alpha +\partial_\mathbf{n}\phi^\alpha
= -[\alpha+ (A^0_\Gamma)^{-1}]\psi^\alpha_t+ \langle\psi^\alpha\rangle_\Gamma+\langle \partial_\mathbf{n}\phi^\alpha\rangle_\Gamma\\
& \qquad\ \quad\qquad\qquad\qquad\qquad
-G'(\psi^\alpha)+\langle G'(\psi^\alpha)\rangle_\Gamma,
&\text{on}\ \Gamma.
\end{aligned}
\label{ellpal}
\right.
\end{equation}
Applying the elliptic estimate in Lemma \ref{esLepH2}, we obtain
\begin{align}
&\|\phi^\alpha(t)\|_{H^2(\Omega)}+\|\psi^\alpha(t)\|_{H^2(\Gamma)}\nonumber\\
&\quad \leq C\alpha\big(\|\phi^\alpha_t\|_{L^2(\Omega)}+ \|\psi^\alpha_t\|_{L^2(\Gamma)}\big)+C \big(\|\phi^\alpha_t\|_{(H^1(\Omega))^*}+ \|\psi^\alpha_t\|_{(H^1(\Gamma))^*}\big)\nonumber\\
&\qquad +C\|\langle\psi^\alpha\rangle_\Gamma\|_{L^2(\Gamma)}+C\big(\|\langle\partial_\mathbf{n}\phi^\alpha\rangle_\Gamma\|_{L^2(\Omega)}
+\|\langle\partial_\mathbf{n}\phi^\alpha\rangle_\Gamma\|_{L^2(\Gamma)}\big)\nonumber\\
&\qquad +C\|F'(\phi^\alpha)-\langle F'(\phi^\alpha)\rangle_\Omega\|_{L^2(\Omega)} +C\|G'(\psi^\alpha)-\langle G'(\psi^\alpha)\rangle_\Gamma\|_{L^2(\Gamma)},
\label{esalpH2}
\end{align}
where the constant $C>0$ depends on $\Omega$, $\Gamma$ and $\kappa$, but is independent of $\alpha$.

(a) Assume that $(\phi_0, \psi_0)\in \mathcal{V}^1$ and the growth assumption  (\textbf{A3}) is satisfied. By the Sobolev embedding theorem and H\"older's inequality, we obtain
\begin{align}
&\|F'(\phi^\alpha)\|_{L^2(\Omega)} +|\langle F'(\phi^\alpha)\rangle_\Omega|\leq C(1+\|\phi^\alpha\|_{H^1(\Omega)}^{p+1}),\nonumber\\
&\|G'(\psi^\alpha)\|_{L^2(\Gamma)} +|\langle G'(\psi^\alpha)\rangle_\Gamma| \leq C(1+\|\psi^\alpha\|_{H^1(\Gamma)}^{q+1}).\nonumber
\end{align}
Combining the above estimates with \eqref{esunif1}, \eqref{espncla} and \eqref{esalpH2}, we deduce from the interpolation inequality
 \begin{align}
 \|\phi^\alpha(t)\|_{H^r(\Omega)}\leq \epsilon \|\phi^\alpha(t)\|_{H^2(\Omega)}+C_\epsilon \|\phi^\alpha(t)\|_{H^1(\Omega)},\quad \forall\, \epsilon>0,\label{intpo}
 \end{align}
  that
\begin{align}
\|(\phi^\alpha(t), \psi^\alpha(t))\|_{L^2(0,T; \mathcal{V}^2)}&\leq C(1+T),\label{esL2H2}
\end{align}
where $C>0$ depends on $\|(\phi_0, \psi_0)\|_{\mathcal{V}^1}$, $C_F$, $\widehat{C}_F$, $C_G$, $\widehat{C}_G$, $\Omega$, $\Gamma$ and $\kappa$, but is independent of $\alpha$.

(b) Assume that $(\phi_0, \psi_0)\in \mathcal{V}^3$ and assumptions (\textbf{A3}) is satisfied. From \eqref{esunif1}, \eqref{estt}, \eqref{espncla}, \eqref{esalpH2} and \eqref{intpo}, we easily get
\begin{align}
\|(\phi^\alpha(t), \psi^\alpha(t))\|_{L^\infty(0,T; \mathcal{V}^2)}
&\leq C,\label{esLiH2aa}
\end{align}
where $C>0$ depends on $\|(\phi_0, \psi_0)\|_{\mathcal{V}^3}$, $C_F$, $\widetilde{C}_F$, $\widehat{C}_F$, $C_G$, $\widetilde{C}_G$, $\widehat{C}_G$, $\Omega$, $\Gamma$ and $\kappa$, but is independent of $\alpha$.

(c) Assume that $(\phi_0, \psi_0)\in \mathcal{V}^3$ and assumption (\textbf{A4}) is satisfied.
For sufficient regular $\phi^\alpha$, we observe that $\widetilde{F}(\phi^\alpha)|_\Gamma=\widetilde{F}(\psi^\alpha)$, where $\widetilde{F}$ is given by \eqref{tF}.
On the other hand, by assumption (\textbf{A4}), the sign condition \eqref{FFG} and the Cauchy--Schwarz inequality, we infer that
\begin{align}
&\int_\Gamma \widetilde{G}'(\psi^\alpha) \widetilde{F}'(\psi^\alpha)\, \d S
\geq \frac{1}{2\rho_1}\int_\Gamma |\widetilde{F}'(\psi^\alpha)|^2\, \d S-\frac{\rho_2^2}{2\rho_1}|\Gamma|.
\label{eee}
\end{align}
Then testing the equations for $\mu^\alpha$, $\mu_\Gamma^\alpha$ in \eqref{ACH} by $\widetilde{F}'(\phi^\alpha)$ and $\widetilde{F}'(\psi^\alpha)$, respectively, adding the resultants together, we obtain
\begin{align}
 &\int_\Omega \widetilde{F}''(\phi^\alpha)|\nabla \phi^\alpha|^2\, \d x
 +\int_\Omega |\widetilde{F}'(\phi^\alpha)|^2 \, \d x
 + \kappa\int_\Gamma \widetilde{F}''(\psi^\alpha)|\nabla_\Gamma \psi^\alpha|^2 \, \d S\nonumber\\
 &\qquad
 +\int_\Gamma \widetilde{G}'(\psi^\alpha) \widetilde{F}'(\psi^\alpha)\, \d S\nonumber\\
 &\quad
 =-\alpha\int_\Omega \phi^\alpha_t \widetilde{F}'(\phi^\alpha)\, \d x
 -\alpha\int_\Gamma \psi^\alpha_t \widetilde{F}'(\psi^\alpha)\, \d S
 + \int_\Omega \mu^\alpha  \widetilde{F}'(\phi^\alpha)\, \d x\nonumber\\
 &\qquad
 +\int_\Gamma \mu_\Gamma^\alpha \widetilde{F}'(\psi^\alpha)\, \d S
 -\int_\Gamma \psi^\alpha \widetilde{F}'(\psi^\alpha)\, \d S\nonumber\\
 &\qquad
 -\int_\Omega (F'(\phi^\alpha)-\widetilde{F}'(\phi^\alpha))\widetilde{F}'(\phi^\alpha) \, \d x  \nonumber\\
 &\qquad
 -\int_\Gamma (G'(\psi^\alpha)-\widetilde{G}'(\psi^\alpha))\widetilde{F}'(\psi^\alpha) \, \d S\nonumber\\
 &\quad \leq
 \frac12\int_\Omega |\widetilde{F}'(\phi^\alpha)|^2 \, \d x
  +\frac{3\alpha^2}{2} \|\phi_t^\alpha\|_{L^2(\Omega)}^2
  +\frac32\|\mu^\alpha\|_{L^2(\Omega)}^2 \nonumber\\
 &\qquad +\frac32\int_\Omega |F'(\phi^\alpha)-\widetilde{F}'(\phi^\alpha)|^2 \, \d x
  +\frac{1}{4\rho_1}\int_\Gamma |\widetilde{F}'(\psi^\alpha)|^2\, \d S
  \nonumber\\
 &\qquad + 4\rho_1\alpha^2 \|\psi^\alpha_t\|_{L^2(\Gamma)}^2 +4\rho_1\|\mu_\Gamma^\alpha\|_{L^2(\Gamma)}^2
 +4\rho_1\|\psi^\alpha\|^2_{L^2(\Gamma)}\nonumber\\
 &\qquad + 4\rho_1\int_\Gamma |G'(\psi^\alpha)-\widetilde{G}'(\psi^\alpha)|^2\, \d S.
 \label{dfdgin}
\end{align}
Exploiting \eqref{dfdgin} with \eqref{tF}, \eqref{esunif1}, \eqref{estt}, \eqref{esL2H1muhi} and \eqref{eee}, we have
\begin{align}
\|F'(\phi^\alpha)\|_{L^\infty(0,T; L^2(\Omega))}
+\|F'(\psi^\alpha)\|_{L^\infty(0,T; L^2(\Gamma))}
\leq C(1+\|\phi^\alpha\|_{L^\infty(0,T; H^r(\Omega))}).
\label{L2F}
\end{align}
Similarly, testing the equation for $\mu_\Gamma^\alpha$ in \eqref{ACH} by $\widetilde{G}'(\psi^\alpha)$, we get
\begin{align}
&  \kappa\int_\Gamma \widetilde{G}''(\psi^\alpha)|\nabla_\Gamma \psi^\alpha|^2 \, \d S +\int_\Gamma |\widetilde{G}'(\psi^\alpha)|^2 \,\d S\nonumber\\
&\quad = - \alpha\int_\Gamma \psi^\alpha_t \widetilde{G}'(\psi^\alpha) \,\d S
+\int_\Gamma \mu_\Gamma^\alpha \widetilde{G}'(\psi^\alpha)\, \d S
-\int_\Gamma \psi^\alpha \widetilde{G}'(\psi^\alpha)\, \d S\nonumber\\
&\qquad -\int_\Gamma \partial_\mathbf{n} \phi^\alpha \widetilde{G}'(\psi^\alpha)\, \d S
-\int_\Gamma (G'(\psi^\alpha)-\widetilde{G}'(\psi^\alpha))\widetilde{G}'(\psi^\alpha) \, \d S\nonumber\\
&\quad \leq \frac12\int_\Gamma |\widetilde{G}'(\psi^\alpha)|^2 \,\d S
+\frac52\alpha^2\|\psi_t^\alpha\|^2_{L^2(\Gamma)}
+\frac52 \|\mu_\Gamma^\alpha\|_{L^2(\Gamma)}^2
+\frac52\|\psi^\alpha\|^2_{L^2(\Gamma)}
\nonumber\\
&\qquad +\frac52\|\partial_\mathbf{n}\phi^\alpha\|_{L^2(\Gamma)}^2
+\frac52 \int_\Gamma |G'(\psi^\alpha)-\widetilde{G}'(\psi^\alpha)|^2\, \d S.
\label{dfdgina}
\end{align}
Using \eqref{tG}, \eqref{esunif1}, \eqref{espncla}, \eqref{esL2H1muhi} and \eqref{dfdgina}, we have
\begin{align}
\|G'(\psi^\alpha)\|_{L^\infty(0,T; L^2(\Gamma))}\leq C\big(1+\|\phi^\alpha\|_{L^\infty(0,T; H^r(\Omega))}\big).
\label{L2G}
\end{align}
Finally, we deduce from \eqref{espncla}, \eqref{FGm1},  \eqref{esalpH2}, \eqref{L2F}, \eqref{L2G} and the interpolation inequality \eqref{intpo} that
\begin{align}
\|(\phi^\alpha(t), \psi^\alpha(t))\|_{L^\infty(0,T; \mathcal{V}^2)}&\leq C,
\label{esLiH2}
\end{align}
where $C>0$ depends on $\|(\phi_0, \psi_0)\|_{\mathcal{V}^3}$, $C_F$, $\widetilde{C}_F$, $C_G$, $\widetilde{C}_G$, $\Omega$, $\Gamma$, $\rho_1$, $\rho_2$ and $\kappa$, but is independent of $\alpha$.

\medskip

\textbf{Seventh estimate: $\mathcal{V}^3$-estimate for $(\phi^\alpha, \psi^\alpha)$}.
Assume that $(\phi_0, \psi_0)\in \mathcal{V}^3$ and either (\textbf{A3}) or (\textbf{A4}) is satisfied.
Applying the elliptic estimate in Lemma \ref{esLepH2} once again, we obtain
\begin{align}
&\|\phi^\alpha(t)\|_{H^3(\Omega)}+\|\psi^\alpha(t)\|_{H^3(\Gamma)}\nonumber\\
&\quad \leq C\alpha\big(\|\phi^\alpha_t\|_{H^1(\Omega)}+ \|\psi^\alpha_t\|_{H^1(\Gamma)}\big)+C \big(\|\phi^\alpha_t\|_{(H^1(\Omega))^*}+ \|\psi^\alpha_t\|_{(H^1(\Gamma))^*}\big)\nonumber\\
&\qquad +C\|\langle\psi^\alpha\rangle_\Gamma\|_{H^1(\Gamma)}+C\big(\|\langle\partial_\mathbf{n}\phi^\alpha\rangle_\Gamma\|_{H^1(\Omega)}
+\|\langle\partial_\mathbf{n}\phi^\alpha\rangle_\Gamma\|_{H^1(\Gamma)}\big)\nonumber\\
&\qquad +C\|F'(\phi^\alpha)-\langle F'(\phi^\alpha)\rangle_\Omega\|_{H^1(\Omega)} +C\|G'(\psi^\alpha)-\langle G'(\psi^\alpha)\rangle_\Gamma\|_{H^1(\Gamma)}.
\label{esalpH3}
\end{align}
Then it follows from \eqref{esunif1}, \eqref{phiuni}, \eqref{estt}--\eqref{inttr}, \eqref{esLiH2aa}, \eqref{esLiH2} and the fact $\alpha\in (0,1]$ that
\begin{align}
\|(\phi^\alpha(t), \psi^\alpha(t))\|_{L^2(0,T; \mathcal{V}^3)}&\leq C(1+T),
\label{esL2H3}
\end{align}
where $C>0$ may depend on $\|(\phi_0, \psi_0)\|_{\mathcal{V}^3}$, $C_F$, $\widetilde{C}_F$, $\widehat{C}_F$,  $C_G$, $\widetilde{C}_G$, $\widehat{C}_G$, $\Omega$, $\Gamma$, $\rho_1$, $\rho_2$ and $\kappa$, but is independent of $\alpha$.

\subsection{Proof of Theorem \ref{main1}: the case with surface diffusion $\kappa>0$}

Combining the local well-posedness result for the regularized problem \eqref{ACH} (see Proposition \ref{exeACH}) and the uniform estimates obtained in the previous section, we are in a position to prove Theorem \ref{main1}, with $\kappa>0$ being an arbitrary but fixed constant.\smallskip

 (1) \textbf{Existence of global strong solutions}.
 Assume that $(\phi_0, \psi_0)\in \mathcal{V}^3$ and assumptions (\textbf{A1}), (\textbf{A2}) are satisfied. In addition, we suppose that either (\textbf{A3}) or (\textbf{A4}) is fulfilled.
  The uniform estimates \eqref{esunif1}, \eqref{phiuni}, \eqref{estt}, \eqref{esL2H1muhi}, \eqref{esL2H23muhi}, \eqref{esLiH2aa}, \eqref{esLiH2} and \eqref{esL2H3} imply that the unique local strong solution $(\phi^\alpha, \psi^\alpha)$ to problem \eqref{ACH} can be extended beyond $T_\alpha$ to the whole interval $[0, T]$ for arbitrary $T>0$. Moreover, these uniform estimates allow us to find a convergent subsequence of the approximating solutions still denoted by $(\phi^\alpha, \psi^\alpha, \mu^\alpha, \mu^\alpha_\Gamma)$ without abusing of notion as well as some limit functions $(\phi, \psi, \mu, \mu_\Gamma)$, such that as $\alpha\to 0^+$
\begin{equation}
\begin{aligned}
&(\phi^\alpha, \psi^\alpha) \rightharpoonup (\phi, \psi)\quad \qquad \text{weakly star in}\ L^\infty(0,T; \mathcal{V}^2),\\
&(\phi^\alpha, \psi^\alpha) \rightharpoonup (\phi, \psi)\quad\qquad \text{weakly in}\ L^2(0,T; \mathcal{V}^3),\\
&(\phi^\alpha_t,\psi^\alpha_t)  \rightharpoonup (\phi_t,\psi_t)\quad\quad \  \text{weakly in}\ L^2(0,T; \mathcal{V}^1),\\
&(\phi^\alpha_t,\psi^\alpha_t)  \rightharpoonup (\phi_t,\psi_t)\quad\quad \  \text{weakly star in}\ L^\infty(0,T; (H^1(\Omega))^*\times (H^1(\Gamma))^*),\\
&(\mu^\alpha,\mu^\alpha_\Gamma) \rightharpoonup (\mu, \mu_\Gamma) \ \ \qquad \text{weakly in}\ L^2(0,T; H^3(\Omega)\times H^3(\Gamma)),\\
&\partial_\mathbf{n}\phi^\alpha \rightharpoonup \partial_\mathbf{n}\phi \quad\ \qquad \qquad \text{weakly in}\ L^2(0,T; H^\frac32(\Gamma)),\\
&(\alpha \phi^\alpha_t, \alpha \psi^\alpha_t)  \to (0,0)\, \qquad\text{strongly in}\ L^2(0,T; \mathcal{H}),
\end{aligned}
\nonumber
\end{equation}
By the well-known compactness lemma (cf. e.g., \cite[Section 8, Corollary 4]{SI}), we obtain
\begin{align}
&(\phi^\alpha, \psi^\alpha) \to (\phi, \psi)\quad \text{strongly in}\ C([0,T]; \mathcal{V}^{2}),
\nonumber
\end{align}
which implies
$$
\phi|_{t=0}=\phi_0,\quad \text{ a.e. in}\ \Omega, \quad \psi|_{t=0}=\psi_0,\quad \text{ a.e. in}\ \Gamma,
$$
and
$$\phi^\alpha \to \phi,\ \ \psi^\alpha \to \psi,\quad \text{a.e. in }\Omega\times(0,T)\ \text{and}\ \Gamma\times(0,T),\text{ respectively.}$$
From the pointwise convergence of $(\phi^\alpha, \psi^\alpha)$ and the uniform estimates  \eqref{L2F}, \eqref{L2G}, we are able to show the convergence of nonlinear terms $F', G'$ such that
\begin{align*}
(F'(\phi^\alpha), G'(\psi^\alpha))\rightharpoonup (F'(\phi), G'(\psi)) \quad \text{weakly in}\ L^2(0,T; \mathcal{H}),
\end{align*}
and then identify $\mu$ and $\mu_\Gamma$ as
\begin{align}
&\mu=-\Delta \phi+F'(\phi) \qquad\qquad\qquad \quad\ \   \text{in}\ L^2(0,T; L^2(\Omega)),\label{lmu}\\
&\mu_\Gamma=-\kappa \Delta_\Gamma\psi+\psi+\partial_\mathbf{n}\phi+G'(\psi) \quad \text{in}\ L^2(0,T; L^2(\Gamma)).\label{lmug}
\end{align}
The above subsequential convergence results enable us to pass to the limit in the regularized system \eqref{ACH} as $\alpha \to 0^+$ to conclude that the limit functions $(\phi, \psi, \mu, \mu_\Gamma)$ satisfy the original system \eqref{CH} for a.e. $t\in (0,T)$.
Moreover, the spatial regularity of $(\phi, \psi)$ can be improved by using elliptic estimates. Applying Lemma \ref{esLepH2} to the following system (cf. \eqref{lmu}--\eqref{lmug})
\begin{equation}
\left\{
\begin{aligned}
&-\Delta \phi=\mu-F'(\phi),
&\text{in}\ \Omega,\\
&\phi|_\Gamma=\psi,&\text{on}\ \Gamma,\\
&-\kappa \Delta_\Gamma\psi+\psi+\partial_\mathbf{n}\phi
=\mu_\Gamma-G'(\psi),
&\text{on}\ \Gamma,
\end{aligned}
\label{CHEP}
\right.
\end{equation}
 we deduce that for a.e. $t\in (0,T)$,
\begin{align}
&\|\phi\|_{H^{s+2}(\Omega)}+\|\psi\|_{H^{s+2}(\Gamma)}\nonumber\\
&\quad \leq C(\|\mu\|_{H^s(\Omega)}+\|\mu_\Gamma\|_{H^s(\Gamma)}
+\|F'(\phi)\|_{H^s(\Omega)}+\|G'(\psi)\|_{H^s(\Gamma)}),\quad s=1,3.\label{esH3}
\end{align}
By assumption (\textbf{A1}), \eqref{esL2H1muhi}, \eqref{esL2H23muhi}, \eqref{esLiH2aa}, \eqref{esLiH2}, \eqref{esL2H3} and the Sobolev embedding theorem, we have
\begin{align*}
\|(\phi, \psi)\|_{L^\infty(0,T; \mathcal{V}^3)}\leq C \quad\text{and}\quad \|(\phi, \psi)\|_{L^2(0,T; \mathcal{V}^5)}\leq C.
\end{align*}
In summary, we can conclude that the limit function $(\phi, \psi)$ is a global strong solution to the original problem \eqref{CH}. Besides, its regularity enables us to derive the
energy identity \eqref{BELa}.\smallskip

(2) \textbf{Existence of global weak solutions}.
Assume that $(\phi_0, \psi_0)\in \mathcal{V}^1$ and assumptions (\textbf{A1})--(\textbf{A3}) are satisfied. By the classical density argument, we first find an approximating sequence $\{(\phi_0^n, \psi_0^n)\}_{n=1}^\infty\subset \mathcal{V}^3$ such that $(\phi_0^n, \psi_0^n)$ strongly converge to $(\phi_0, \psi_0)$ in $\mathcal{V}^1$ as $n\to +\infty$. For every $n\in \mathbb{N}$ and $\alpha \in (0,1]$,  problem \eqref{ACH} admits a global strong solution denoted by $(\phi^{n,\alpha}, \psi^{n,\alpha})$. Then by the uniform estimates \eqref{esunif1}, \eqref{phiuni}, \eqref{esL2H1mulo} and \eqref{esL2H2}, we have the following convergence as $n\to+\infty$ and $\alpha\to 0^+$ (up to a subsequence, not relabelled for simplicity):
\begin{align*}
&(\phi^{n,\alpha}, \psi^{n,\alpha}) \rightharpoonup (\phi, \psi)\qquad \quad  \text{weakly star in}\ L^\infty(0,T; \mathcal{V}^1),\\
&(\phi^{n,\alpha}, \psi^{n,\alpha}) \rightharpoonup (\phi, \psi)\qquad \quad \text{weakly in}\ L^2(0,T; \mathcal{V}^2),\\
&(\phi^{n,\alpha}_t, \psi^{n,\alpha}_t)  \rightharpoonup (\phi_t, \psi_t) \ \qquad  \text{weakly in}\ L^2(0,T; (H^1(\Omega))^*\times (H^1(\Gamma))^*),\\
&(\mu^{n,\alpha},  \mu^{n,\alpha}_\Gamma) \rightharpoonup (\mu, \mu_\Gamma) \ \ \qquad \text{weakly in}\ L^2(0,T; H^1(\Omega)\times H^1(\Gamma)),\\
&\partial_\mathbf{n}\phi^{n,\alpha} \rightharpoonup \partial_\mathbf{n}\phi\ \ \qquad \qquad \quad\ \text{weakly in}\ L^2(0,T; H^\frac12(\Gamma)),\\
&(\alpha \phi^{n,\alpha}_t,\alpha \psi^{n,\alpha}_t)  \to (0, 0)\, \qquad\text{strongly in}\ L^2(0,T; \mathcal{H}),
\end{align*}
and the convergence on initial data (cf. e.g., \cite[Lemma 3.1.7]{Z04})
\begin{align}
(\phi^{n,\alpha}, \psi^{n,\alpha})|_{t=0} \rightharpoonup (\phi, \psi)|_{t=0}\quad \text{weakly in}\ (H^1(\Omega))^*\times (H^1(\Gamma))^*,\nonumber
\end{align}
such that $(\phi, \psi)|_{t=0}=(\phi_0, \psi_0)$.
Using the Aubin--Lions compactness lemma (cf. e.g., \cite{SI}), we have
\begin{align}
&(\phi^{n,\alpha}, \psi^{n,\alpha}) \to (\phi, \psi)\quad \text{strongly in}\ L^2(0,T; \mathcal{V}^{1}),
\nonumber
\end{align}
which together with the growth assumption (\textbf{A3}) and the Sobolev embedding theorem yields
\begin{align*}
(F'(\phi^{n,\alpha}), G'(\psi^{n,\alpha}))\rightharpoonup (F'(\phi), G'(\psi)) \quad \text{weakly in}\ L^2(0,T; \mathcal{H}).
\end{align*}
Thus we still have the identities \eqref{lmu} and \eqref{lmug}. Using (\textbf{A3}) again, we deduce that
\begin{align}
&\|F'(\phi)\|_{H^1(\Omega)}\nonumber\\
&\quad \leq \|F''(\phi)\nabla \phi\|_{L^2(\Omega)}+\|F'(\phi)\|_{L^2(\Omega)}\nonumber\\
&\quad \leq C\|\nabla \phi\|_{L^{2p+2}(\Omega)}\|1+|\phi|^p\|_{L^\frac{2p+2}{p}(\Omega)}+ C\|1+|\phi|^{p+1}\|_{L^{2}(\Omega)}\nonumber\\
&\quad \leq C\|\phi\|_{H^2(\Omega)}\Big(1+\|\phi\|_{H^1(\Omega)}^p\Big)+C\|\phi\|_{H^1(\Omega)}^{p+1}+C,
\label{esFF}
\end{align}
and in a similar manner
\begin{align}
\|G'(\psi)\|_{H^1(\Gamma)}
&\leq C\|\psi\|_{H^2(\Gamma)}\Big(1+\|\psi\|_{H^1(\Gamma)}^q\Big)+C\|\psi\|_{H^1(\Gamma)}^{q+1}+C.
\label{esGG}
\end{align}
As a consequence, it follows from the above estimates and \eqref{esunif1}, \eqref{esL2H2}, \eqref{CHEP}, \eqref{esH3} that
$$\|(\phi, \psi)\|_{L^2(0,T;\mathcal{V}^3)}\leq C.$$
In summary, the limit function $(\phi, \psi)$ is a global weak solution to problem \eqref{CH} satisfying the weak formulation \eqref{we1}--\eqref{we1a}.
It is standard to show that the weak solution fulfills the energy inequality \eqref{aenergyin} by passing to the limit $n\to+\infty$, $\alpha \to 0^+$ in the energy equality for strong solutions $(\phi^{n, \alpha}, \psi^{n, \alpha})$.
\smallskip

(3) \textbf{Continuous dependence and uniqueness}.
Let $(\phi_i, \psi_i)$, $i=1,2$, be two global weak solutions to problem \eqref{CH} subject to initial data $(\phi_{0i}, \psi_{0i})\in \mathcal{V}^1$ with $\langle\phi_{01} \rangle_\Omega=\langle\phi_{02}\rangle_\Omega$, $\langle\psi_{01} \rangle_\Gamma=\langle\psi_{02}\rangle_\Gamma$. Denote the differences
\begin{align*}
&\overline{\phi}=\phi_1-\phi_2,\quad  \overline{\psi}=\psi_1-\psi_2,\quad \overline{\phi}_0= \phi_{01}- \phi_{02},\quad  \overline{\psi}_0=\psi_{01}-\psi_{02}.
\end{align*}
Then the pair $(\overline{\phi}, \overline{\psi})$ satisfies (keeping the weak formulations in Definition \ref{defweak} in mind)
\begin{equation}
\left\{
\begin{aligned}
&\overline{\phi}_t=\Delta \overline{\mu},\quad
\text{with}\ \ \overline{\mu}=-\Delta\overline{\phi}+F'(\phi_1)-F'(\phi_2),
&\text{in}\ \Omega\times (0,T),\\
&\partial_\mathbf{n}\overline{\mu}=0,
&\text{on}\ \Gamma\times (0,T),\\
&\overline{\phi}|_\Gamma=\overline{\psi},
&\text{on}\ \Gamma\times (0,T),\\
&\overline{\psi}_t=\Delta_\Gamma \overline{\mu}_\Gamma,
&\text{on}\ \Gamma\times (0,T),\\
&\quad \text{with}\ \  \overline{\mu}_\Gamma=-\kappa\Delta_\Gamma\overline{\psi}
+\overline{\psi}+\partial_\mathbf{n}\overline{\phi}+G'(\psi_1)-G'(\psi_2),
&\text{on}\ \Gamma\times (0,T),\\
&\overline{\phi}|_{t=0}=\overline{\phi}_0,
&\text{in}\ \Omega,\\
&\overline{\psi}^\alpha|_{t=0}=\overline{\psi}_0,
&\text{on}\ \Gamma,
\end{aligned}
\label{diffCH}
\right.
\end{equation}
Since $\langle\overline{\phi}(t)\rangle_\Omega=\langle\overline{\psi}(t)\rangle_\Gamma=0$ for all $t\in [0,T]$, we can test the first and fourth equations in \eqref{diffCH} by $(A_\Omega^0)^{-1} \overline{\phi}$ and $(A_\Gamma^0)^{-1} \overline{\psi}$, respectively, adding the resultants together, we have
\begin{align}
&\frac12\frac{\d}{\d t}\big(\|\overline{\phi}\|_{(H^1(\Omega))^*}^2
+ \|\overline{\psi}\|_{(H^1(\Gamma))^*}^2\big)
+ \|\nabla \overline{\phi}\|_{L^2(\Omega)}^2
+ \kappa \|\nabla_\Gamma \overline{\psi}\|_{L^2(\Gamma)}^2
+ \|\overline{\psi}\|_{L^2(\Gamma)}^2\nonumber\\
&\quad
=-\int_\Omega (F'(\phi_1)-F'(\phi_2)) \overline{\phi}\, \d x
-\int_\Gamma (G'(\psi_1)-G'(\psi_2)) \overline{\psi}\, \d S\nonumber\\
&\quad
=-\int_\Omega\int_0^1 F''(\phi_2+s(\phi_1-\phi_2))\overline{\phi}^2\, \d s\d x\nonumber\\
&\qquad -\int_\Gamma\int_0^1 G''(\psi_2+s(\psi_1-\phi_2))\overline{\psi}^2\, \d s\d S\nonumber\\
&\quad \leq \widetilde{C}_F\|\overline{\phi}\|_{L^2(\Omega)}^2
+\widetilde{C}_G\|\overline{\psi}\|_{L^2(\Gamma)}^2,\nonumber
\end{align}
where in the last line we have used (\textbf{A2}).
Exploiting the interpolation inequalities and the trace theorem
\begin{align*}
\|\overline{\phi}\|_{L^2(\Omega)}^2
&\leq C\|\overline{\phi}\|_{H^1(\Omega)}\|\overline{\phi}\|_{(H^1(\Omega))^*},\\ \|\overline{\psi}\|_{L^2(\Gamma)}^2
&\leq C\|\overline{\phi}\|_{H^{\frac12+\epsilon}(\Omega)}^2\leq C\|\overline{\phi}\|_{H^1(\Omega)}^{\frac32+\epsilon}\|\overline{\phi}\|_{(H^1(\Omega))^*}^{\frac12-\epsilon},\quad
\text{for some}\ \epsilon \in (0,\frac12),
\end{align*}
we deduce from Poincar\'e's inequality and Young's inequality that
\begin{align}
\frac{\mathrm{d}}{\mathrm{d}t}\big(\|\overline{\phi}\|_{(H^1(\Omega))^*}^2+ \|\overline{\psi}\|_{(H^1(\Gamma))^*}^2\big)+\|(\overline{\phi}, \overline{\psi})\|_{\mathcal{V}^1}^2 \leq C\|\overline{\phi}\|_{(H^1(\Omega))^*}^2.\nonumber
\end{align}
Thanks to Gronwall's lemma, we obtain
\begin{align}
&\|\overline{\phi}(t)\|_{(H^1(\Omega))^*}^2+ \|\overline{\psi}(t)\|_{(H^1(\Gamma))^*}^2
+\int_0^t \|(\overline{\phi}(\tau), \overline{\psi}(\tau))\|_{\mathcal{V}^1}^2 \, \d \tau\nonumber\\
&\quad \leq  C_T \big(\|\overline{\phi}_0\|_{(H^1(\Omega))^*}^2+ \|\overline{\psi}_0\|_{(H^1(\Gamma))^*}^2\big),
\quad \forall\, t\in(0,T),
\label{conti}
\end{align}
where $C_T>0$ depends on $T$, but is independent of the initial data.
The continuous dependence result \eqref{conti} easily yields the uniqueness of global weak solutions to problem \eqref{CH}.
We note that the inequality \eqref{conti} automatically holds for global strong solutions and in that case actually continuous dependence in stronger norms can be obtained in a similar manner.

The proof of Theorem \ref{main1} is complete.

\subsection{Proof of Theorem \ref{main1a}: the case without surface diffusion $\kappa=0$}
We observe that the surface diffusion term $-\kappa\Delta_\Gamma \psi$ (with $\kappa>0$) can be regarded as a regularization on the boundary.
Thus, we first solve the approximating system \eqref{ACH} with $\alpha, \kappa\in (0, 1]$, obtaining uniform estimates with respect to both  $\alpha$, $\kappa$  and then pass to the limit as $(\alpha,\kappa)\to (0^+, 0^+)$.
The procedure is similar to the previous section. Hence, we just sketch the essential steps.

Assume $(\phi_0, \psi_0)\in \mathcal{V}^3$.
 We solve the regularized problem \eqref{ACH} with $\alpha, \kappa\in (0,1]$.
 The approximating global strong solutions are denoted by $(\phi^{\alpha, \kappa},\psi^{\alpha,\kappa})$ with bulk/surface chemical potentials $(\mu^{\alpha,\kappa}, \mu^{\alpha,\kappa}_\Gamma)$.
Then by the uniform estimates \eqref{esunif1}, \eqref{phiuni}, \eqref{estt}, \eqref{espnh-12}, \eqref{esL2H1mulokb} and \eqref{esL2H1mulokba}, we can find a convergent subsequence, still denoted by $(\phi^{\alpha, \kappa},\psi^{\alpha,\kappa})$ without abusing of notion, such that as $(\alpha,\kappa)\to (0^+, 0^+)$
\begin{align*}
&(\phi^{\alpha, \kappa}, \psi^{\alpha, \kappa}) \rightharpoonup (\phi, \psi)\qquad\qquad  \quad \!\!\!\text{weakly star in}\ L^\infty(0,T; V^1),\\
&(\phi^{\alpha, \kappa}_t,\psi^{\alpha, \kappa}_t)  \rightharpoonup (\phi_t,\psi_t)\qquad\qquad  \! \text{weakly star in}\ L^\infty(0,T; (H^1(\Omega))^*\times (H^1(\Gamma))^*),\\
&(\phi^{\alpha, \kappa}_t,\psi^{\alpha, \kappa}_t)
\rightharpoonup (\phi_t, \psi_t)\qquad\qquad  \! \text{weakly in}\ L^2(0,T; H^1(\Omega)\times L^2(\Gamma)),\\
&(\mu^{\alpha, \kappa}, \mu^{\alpha, \kappa}_\Gamma) \rightharpoonup (\mu, \mu_\Gamma)
\, \qquad \qquad \text{weakly star in}\ L^\infty(0,T; H^1(\Omega)\times H^1(\Gamma)),\\
&(\mu^{\alpha, \kappa}, \mu^{\alpha, \kappa}_\Gamma) \rightharpoonup (\mu, \mu_\Gamma)
\, \qquad \qquad \text{weakly in}\ L^2(0,T; H^3(\Omega)\times H^2(\Gamma)),\\
&\partial_\mathbf{n}\phi^{\alpha, \kappa} \rightharpoonup \partial_\mathbf{n}\phi \qquad\qquad \qquad \quad\ \, \text{weakly star in}\ L^\infty(0,T; H^{-\frac12}(\Gamma)),\\
&(\alpha \phi^{\alpha, \kappa}_t,\alpha \psi^{\alpha, \kappa}_t)  \to (0, 0) \qquad\ \ \quad \text{strongly in}\ L^2(0,T; \mathcal{H}),\\
&\kappa\nabla_\Gamma \psi^{\alpha, \kappa}\to 0\qquad\qquad\qquad\qquad\!\!  \text{strongly in}\ L^\infty(0,T; L^2(\Gamma)).
\end{align*}
By the Aubin--Lions compactness lemma, up to a subsequence,
$$ (\phi^{\alpha, \kappa}, \psi^{\alpha, \kappa}) \to (\phi, \psi)\qquad\qquad  \text{strongly in}\ C([0,T]; V^{1-\epsilon}),\quad \epsilon\in (0,\frac12),$$
which again implies the almost everywhere convergence $(\phi^{\alpha, \kappa}, \psi^{\alpha, \kappa}) \to (\phi, \psi)$.
Then by the growth assumption (\textbf{A3}) (for the case $\kappa=0$) and the Sobolev embedding theorem, we easily have
\begin{align*}
(F'(\phi^{\alpha, \kappa}), G'(\psi^{\alpha, \kappa}))\rightharpoonup (F'(\phi), G'(\psi)) \quad \text{weakly in}\ L^2(0,T; \mathcal{H}).
\end{align*}
From the expression of $\mu^{\alpha, \kappa}$, we see that $\Delta\phi^{\alpha, \kappa}$ is uniformly bounded in $L^2(0,T; L^2(\Omega))$ and thus
\begin{align*}
&\Delta \phi^{\alpha, \kappa} \rightharpoonup  \Delta \phi \qquad \ \text{weakly in}\ L^2(0,T; L^2(\Omega)).
\end{align*}
Hence, we can identify the bulk chemical potential $\mu$ as
\begin{align}
&\mu=-\Delta \phi+F'(\phi) \qquad \text{in}\ L^2(0,T; L^2(\Omega)).\label{lmuk}
\end{align}
On the other hand, for any $\xi\in L^2(0,T;H^1(\Gamma))$, we have
\begin{align*}
&\int_0^T\int_\Gamma \mu^{\alpha, \kappa}_\Gamma \xi\, \d S \d t\nonumber\\
&\quad = \int_0^T\int_\Gamma \Big(\kappa \nabla_\Gamma \psi^{\alpha, \kappa} \cdot \nabla_\Gamma\xi +\psi^{\alpha, \kappa} \xi+ \partial_\mathbf{n}\phi^{\alpha, \kappa}\xi+ G'(\psi^{\alpha, \kappa})\xi\Big) \, \d S\d t,
\end{align*}
then the previous weak convergence results yield that as $(\alpha, \kappa)\to (0^+, 0^+)$, it holds
$$
\int_0^T\int_\Gamma \mu_\Gamma \xi\, \d S\d t
= \int_0^T\int_\Gamma \big(\psi +\partial_\mathbf{n}\phi+ G'(\psi)\big) \xi\, \d S\d t.
$$
Thus we can identify surface chemical potential $\mu_\Gamma$ as
\begin{align}
\mu_\Gamma= \psi+\partial_\mathbf{n}\phi+G'(\psi)\qquad\text{in}\ L^2(0,T; H^{-1}(\Gamma)).
\label{gmuk}
\end{align}
From the regularity of $\mu_\Gamma$, $\psi$ and $G'(\psi)$, by comparison, we see that $\partial_\mathbf{n}\phi \in L^2(0,T; L^2(\Gamma))$ and \eqref{gmuk} indeed holds in $L^2(0,T; L^2(\Gamma))$.

The above subsequential convergence results enable us to pass to the limit as $(\alpha, \kappa)\to (0^+, 0^+)$ to conclude that the limit functions $(\phi, \psi, \mu, \mu_\Gamma)$ satisfy the weak formulations \eqref{we1}--\eqref{we1a} for a.e. $t\in (0,T)$, while \eqref{we2a}, \eqref{we2} are satisfied a.e. in $\Omega\times(0,T)$ and on $\Gamma\times(0,T)$, respectively.
Besides, we also have $(\phi, \psi)|_{t=0}=(\phi_0, \psi_0)$.

The growth assumption (\textbf{A3}) for $\kappa=0$ yields
\begin{align}
\|G'(\psi)\|_{H^1(\Gamma)}
&\leq \|G''(\psi)\nabla_\Gamma \psi\|_{L^2(\Gamma)}+\|G'(\psi)\|_{L^2(\Gamma)}\nonumber\\
&\leq \widehat{C}_G \|\nabla_\Gamma \psi\|_{L^2(\Gamma)} + C\big(\|\psi\|_{L^2(\Gamma)}+1\big).
\label{esGGa}
\end{align}
Then applying the classical elliptic estimate to the Robin problem
\begin{equation}
\left\{
\begin{aligned}
&-\Delta \phi=\mu-F'(\phi),
&\text{in}\ \Omega,\\
&\phi|_\Gamma=\psi,
&\text{on}\ \Gamma,\\
&\partial_\mathbf{n}\phi+\psi=\mu_\Gamma-G'(\psi),
&\text{on}\ \Gamma,
\end{aligned}
\label{CHEPk}
\right.
\end{equation}
 by the trace theorem $\|\psi\|_{H^1(\Gamma)}\leq C\|\phi\|_{H^{r}(\Omega)}$ for some $r\in (\frac32, 2)$, \eqref{esFF}, \eqref{esGGa}, the interpolation inequality \eqref{intpo}, the Sobolev embedding theorem and Young's inequality, we have for a.e. $t\in (0,T)$,
\begin{align}
\|\phi\|_{H^2(\Omega)}
&\leq C\big(\|\mu\|_{L^2(\Omega)}+\|F'(\phi)\|_{L^2(\Omega)}
+\|\mu_\Gamma\|_{H^\frac12(\Gamma)}
+\|G'(\psi)\|_{H^\frac12(\Gamma)}\big)\nonumber\\
&\leq C\big(\|\mu\|_{L^2(\Omega)}+\|\mu_\Gamma\|_{H^1(\Gamma)}+ \|\phi\|_{L^{2p+2}(\Omega)}^{p+1}+ \|G'(\psi)\|_{H^1(\Gamma)}+1\big)\nonumber\\
&\leq C\big(\|\mu\|_{L^2(\Omega)}+\|\mu_\Gamma\|_{H^1(\Gamma)}+ \|\phi\|_{H^1(\Omega)}^{p+1}+ \|\phi\|_{H^{r}(\Omega)}+ \|\psi\|_{L^2(\Gamma)}+1\big)\nonumber\\
&\leq \frac12\|\phi\|_{H^2(\Omega)}+C\big(\|\mu\|_{L^2(\Omega)}
+\|\mu_\Gamma\|_{H^1(\Gamma)}\big)\nonumber\\
&\quad +C\big(\|\phi\|_{H^1(\Omega)}^{p+1}+ \|\phi\|_{H^{1}(\Omega)}+1\big).
\label{esH3k}
\end{align}
As a consequence,
\begin{align}
\|(\phi, \psi)\|_{L^\infty(0,T;V^2)}\leq C.\label{infH2}
\end{align}
From the elliptic estimate \cite[Chapter 2, Theorem 5.4]{LM68} (see also \cite[Part I, Chapter 3, Theorem 3.2]{BG}), the fact $(\mu, \mu_\Gamma)\in L^\infty(0,T; H^1(\Omega)\times H^1(\Gamma))$ and \eqref{infH2}, we further infer that
\begin{align}
\|\phi\|_{H^\frac52(\Omega)}
&\leq C\big(\|\mu\|_{H^\frac12(\Omega)}+\|F'(\phi)\|_{H^\frac12(\Omega)}
+\|\mu_\Gamma\|_{H^1(\Gamma)}
+\|G'(\psi)\|_{H^1(\Gamma)}\big)\nonumber\\
&\leq C,
\label{esH3ka}
\end{align}
namely, $(\phi,\psi)\in L^\infty(0,T; V^\frac52)$. In a similar manner, from
$(\mu, \mu_\Gamma)\in L^2(0,T; H^3(\Omega)\times H^2(\Gamma))$ and \eqref{infH2} we can deduce that  $(\phi, \psi)\in L^2(0,T; V^\frac72)$.

As far as the global weak solution is concerned, i.e., the initial datum satisfies $(\phi_0, \psi_0)\in \mathcal{V}^1$, we only have weaker uniform estimates \eqref{esunif1}, \eqref{phiuni}, \eqref{espnh-12} and \eqref{esL2H1mulok}.
By a similar density argument like before, we can conclude the existence of a global weak solution to problem \eqref{CH} with excepted regularity properties.

Finally, the continuous dependence result follows from the same argument as for \eqref{conti} in the case $\kappa>0$ (replacing $\mathcal{V}^1$-norm for $(\overline{\phi},\overline{\psi})$ with $V^1$-norm), which implies the uniqueness of global weak/strong solution when $\kappa=0$.

The proof of Theorem \ref{main1a} is complete.

\section{Long-time Behavior}
\setcounter{equation}{0}

\subsection{Uniform-in-time estimates}
First, we derive some uniform-in-time estimates for global weak/strong solutions to problem
\eqref{CH} that are helpful for the investigation of their long-time behavior.

First, it follows from the energy inequality \eqref{aenergyin} (for weak solutions) or the energy identity \eqref{BELa} (for strong solutions) that
\begin{lemma}
Suppose that the assumptions of Theorems \ref{main1} or \ref{main1a} are satisfied. Let $(\phi, \psi)$ be a global weak/strong solution to problem \eqref{CH}.
Then the following estimates hold:
\begin{align}
&\|(\phi(t),\psi(t)\|_{\mathds{V}^1_\kappa}\leq C,\quad \forall\, t\geq 0,
\label{uniH1}\\
&\int_0^{+\infty}\left(\|\nabla\mu(t)\|_{L^2(\Omega)}^2
+\|\nabla_\Gamma\mu_\Gamma(t)
\|_{L^2(\Gamma)}^2 \right)\,\d t\leq C,
\label{unimu}
\end{align}
where the constant $C$ may depend on $E(\phi_0, \psi_0)$, $C_F$, $C_G$, $\Omega$ and $\Gamma$.
\end{lemma}

Next, we establish a smoothing type estimate.

\begin{lemma}\label{comp}
Suppose that the assumptions of Theorems \ref{main1} or \ref{main1a} are satisfied. Let $(\phi, \psi)$ be a global weak/strong solution to problem \eqref{CH}.

(1) If $\kappa>0$, then we have
\begin{align}
\|\phi(t)\|_{H^3(\Omega)}+\|\psi(t)\|_{H^3(\Gamma)}
\leq C\left(\frac{1+t}{t}\right)^\frac12,\quad \forall\, t>0,
\label{regH3a}
\end{align}
where the constant $C$ may depend on $E(\phi_0, \psi_0)$, $\Omega$, $\Gamma$, $C_F$, $\widetilde{C}_F$, $\widehat{C}_F$, $C_G$,  $\widetilde{C}_G$, $\widehat{C}_G$ and $\kappa$.

(2) If $\kappa=0$, then we have
\begin{align}
\|\phi(t)\|_{H^\frac52(\Omega)}\leq  C\left(\frac{1+t}{t}\right)^\frac12,
\quad \forall\, t>0,
\label{regH2k0}
\end{align}
where the constant $C$ may depend on $E(\phi_0, \psi_0)$, $\Omega$, $\Gamma$, $C_F$, $\widetilde{C}_F$, $\widehat{C}_F$, $C_G$,  $\widetilde{C}_G$ and $\widehat{C}_G$.
\end{lemma}
\begin{proof}
We proceed with a standard approximating procedure (cf. for instance, \cite{AW07, LSU, MZ05}).
Given $h>0$, let us introduce the difference quotient of a function by
$$
\partial_t^h f=\frac1h\Big(f(t+h)-f(t)\Big),
\quad \forall \, t\geq 0.
$$
Owing to Definition \ref{defweak}, the difference quotient of a weak solution satisfies the weak formulation
\begin{align}
&\langle (\partial_t^h\phi)_t, \zeta\rangle_{(H^1(\Omega))^*, H^1(\Omega)}
+\int_\Omega \nabla \partial_t^h\mu\cdot \nabla \zeta\, \d x=0,
\label{we1dd} \\
&\langle (\partial_t^h\psi)_t, \eta\rangle_{(H^1(\Gamma))^*, H^1(\Gamma)}
+\int_\Gamma \nabla_\Gamma \partial_t^h \mu_\Gamma \cdot \nabla_\Gamma\eta\, \d S=0,
\label{we1add}
\end{align}
for every $\zeta\in H^1(\Omega)$ and $\eta\in H^1(\Gamma)$ and almost every $t\in (0,+\infty)$, with
\begin{align}
&\partial_t^h \mu=-\Delta \partial_t^h \phi + \frac{1}{h} (F'(\phi(t+h))-F'(\phi(t))),\label{we2add}
\end{align}
for almost every $(x,t)\in \Omega\times (0,+\infty)$ and
\begin{align}
&\partial_t^h \mu_\Gamma= -\kappa\Delta_\Gamma \partial_t^h \psi +  \partial_t^h\psi+ \partial_\mathbf{n} \partial_t^h \phi +\frac{1}{h}(G'(\psi(t+h))-G'(\psi(t))),
\label{we2dd}
\end{align}
for almost every $(x,t)\in \Gamma\times (0,+\infty)$.

\textbf{Case 1. $\kappa>0$}.  Using the facts $\langle\partial_t^h\phi\rangle_\Omega=\langle\partial_t^h\psi\rangle_\Gamma=0$, we test \eqref{we1dd} and \eqref{we1add} by $(A_\Omega^0)^{-1}\partial_t^h\phi$, $(A_\Gamma^0)^{-1}\partial_t^h\psi$, respectively, add the resultants together, then we deduce from \eqref{we2add} and \eqref{we2dd} that
\begin{align}
&\frac12\frac{\d}{\d t}\left(\|\partial_t^h\phi\|_{(H^1(\Omega))^*}^2
+\|\partial_t^h\psi\|_{(H^1(\Gamma))^*}^2\right)
+\|\nabla \partial_t^h \phi\|_{L^2(\Omega)}^2 \nonumber\\
&\qquad +\kappa\|\nabla_\Gamma \partial_t^h\psi\|_{L^2(\Gamma)}^2 + \|\partial_t^h\psi\|_{L^2(\Gamma)}^2\nonumber\\
&\quad = -\frac{1}{h} \int_\Omega (F'(\phi(t+h))-F'(\phi(t))) \partial_t^h\phi\,\d x\nonumber\\
&\qquad -\frac{1}{h} \int_\Gamma (G'(\psi(t+h))-G'(\psi(t))) \partial_t^h\psi\, \d S.
\label{hi1}
\end{align}
Using assumption (\textbf{A2}), Poincar\'e's inequality and Young's inequality, by a similar argument for \eqref{rett1}--\eqref{rett2}, we deduce from \eqref{hi1} that
\begin{align}
&\frac{\d}{\d t}\left(\|\partial_t^h\phi\|_{(H^1(\Omega))^*}^2
+\|\partial_t^h\psi\|_{(H^1(\Gamma))^*}^2\right)
+\|\nabla \partial_t^h\phi\|_{L^2(\Omega)}^2\nonumber\\
&\qquad +\kappa\|\nabla_\Gamma \partial_t^h\psi\|_{L^2(\Gamma)}^2+ \|\partial_t^h\psi\|_{L^2(\Gamma)}^2\nonumber\\
&\quad
\leq C\left(\|\partial_t^h\phi\|_{(H^1(\Omega))^*}^2
+\|\partial_t^h\psi\|_{(H^1(\Gamma))^*}^2\right),
\label{hi3}
\end{align}
where the constant $C>0$ depends on $\Omega$, $\Gamma$, $\widetilde{C}_F$, $\widetilde{C}_G$, but is independent of $\kappa$ and $h$.
We infer from $\phi_t\in L^2(0,+\infty; (H^1(\Omega))^*)$ that for almost every $\tau>0$, it holds
$$\|\partial_t^h\phi(\tau)\|_{(H^1(\Omega))^*}\leq \frac{1}{h}\int_{\tau}^{\tau+h} \|\phi_t(s)\|_{(H^1(\Omega))^*}\,\d s\to \|\phi_t(\tau)\|_{(H^1(\Omega))^*}\quad \text{as}\ h\to 0^+$$
and thus
$\|\partial_t^h\phi\|_{L^2(0,+\infty; (H^1(\Omega))^*)}\leq \|\phi_t\|_{L^2(0,+\infty; (H^1(\Omega))^*)}$ (cf. \cite{AW07}). In a similar manner, we have  $\|\partial_t^h\psi\|_{L^2(0,+\infty; (H^1(\Gamma))^*)}\leq \|\psi_t\|_{L^2(0,+\infty; (H^1(\Gamma))^*)}$.
Multiplying \eqref{hi3} by $t$ and integrating with respect to time, using the above estimates and \eqref{unimu}, we get
\begin{align}
&t\left(\|\partial_t^h\phi(t)\|_{(H^1(\Omega))^*}^2
+\|\partial_t^h\psi(t)\|_{(H^1(\Gamma))^*}^2\right)
+\int_0^t \tau \|(\partial_t^h\phi, \partial_t^h\psi)\|_{\mathcal{V}^1}^2 \,\d\tau \nonumber\\
&\quad \leq \int_0^t(1+C\tau)\left(\|\partial_t^h\phi(\tau)\|_{(H^1(\Omega))^*}^2
+\|\partial_t^h\psi(\tau)\|_{(H^1(\Gamma))^*}^2\right)\, \d\tau\nonumber\\
&\quad \leq C \big(E(\phi_0, \psi_0)+C_F|\Omega|+C_G|\Gamma|\big)(1+t),\quad \forall\, t>0,
\label{regNmu}
\end{align}
which gives
\begin{align}
\|\partial_t^h\phi(t)\|_{(H^1(\Omega))^*}^2
+\|\partial_t^h\psi(t)\|_{(H^1(\Gamma))^*}^2\leq C \left(\frac{1+t}{t}\right),\quad \forall\, t>0,
\label{regNpth}
\end{align}
where the constant $C>0$ is independent of $h$. Then passing to the limit $h\to 0^+$ in \eqref{regNpth}, we infer that
\begin{align}
\|\phi_t(t)\|_{(H^1(\Omega))^*}^2
+\|\psi_t(t)\|_{(H^1(\Gamma))^*}^2\leq C \left(\frac{1+t}{t}\right),\quad \forall\, t>0.
\label{regNpt}
\end{align}
Next, by the growth assumption (\textbf{A3}) and the trace theorem, we can estimate the mean values of $\mu$ and $\mu_\Gamma$ in the same way as \eqref{memu1}--\eqref{memug1}.
Combining them with \eqref{espncla}, \eqref{inttr} and Poincar\'e's inequality, we get
\begin{align}
&\|\mu(t)\|_{H^1(\Omega)}+ \| \mu_\Gamma(t)\|_{H^1(\Gamma)}\nonumber\\
&\quad \leq C\big(\|\nabla \mu(t)\|_{L^2(\Omega)} + \|\nabla_\Gamma\mu_\Gamma(t)\|_{L^2(\Gamma)}
+ \|\phi(t)\|_{H^r(\Omega)}+1\big),
\quad \forall\, t\geq 0.
\label{memuhibb}
\end{align}
Recalling the elliptic problem \eqref{CHEP}, we infer from Lemma \ref{esLepH2}, \eqref{uniH1}, \eqref{memuhibb} and the interpolation inequality \eqref{intpo} that
\begin{align}
&\|\phi(t)\|_{H^2(\Omega)}+\|\psi(t)\|_{H^2(\Gamma)}\nonumber\\
&\quad \leq C\big(\|\mu(t)\|_{L^2(\Omega)}+\|\mu_\Gamma(t)\|_{L^2(\Gamma)}+\|F'(\phi(t))\|_{L^2(\Omega)}+\|G'(\psi(t))\|_{L^2(\Gamma)}\big)\nonumber\\
&\quad \leq C\|\nabla \mu(t)\|_{L^2(\Omega)} + C\|\nabla \mu_\Gamma(t)\|_{L^2(\Gamma)}
+ C\|\phi(t)\|_{H^r(\Omega)}\nonumber\\
&\qquad + C\|\phi(t)\|_{H^1(\Omega)}^{p+1} + C\|\psi(t)\|_{H^1(\Gamma)}^{q+1} + C\nonumber\\
&\quad \leq \frac12\|\phi(t)\|_{H^2(\Omega)}+C\|\nabla \mu(t)\|_{L^2(\Omega)}+ C\|\nabla_\Gamma \mu_\Gamma(t)\|_{L^2(\Gamma)}+C,\nonumber
\end{align}
which together with \eqref{regNpt} implies
\begin{align}
\|\phi(t)\|_{H^2(\Omega)}+\|\psi(t)\|_{H^2(\Gamma)}\leq C\left(\frac{1+t}{t}\right)^\frac12,
\quad \forall\, t>0.
\label{regH2}
\end{align}
Furthermore, we deduce from \eqref{memuhibb}, \eqref{regH2} together with \eqref{esFF} and \eqref{esGG} that
\begin{align}
&\|\phi(t)\|_{H^3(\Omega)}+\|\psi(t)\|_{H^3(\Gamma)}\nonumber\\
&\quad  \leq C\|\nabla \mu(t)\|_{L^2(\Omega)}+ C\|\nabla_\Gamma\mu_\Gamma(t)\|_{L^2(\Gamma)} +C\|\phi(t)\|_{H^2(\Omega)}\nonumber\\
&\qquad + C\|\psi(t)\|_{H^2(\Gamma)}+C\nonumber\\
&\quad  \leq C\left(\frac{1+t}{t}\right)^\frac12,
\quad \forall\, t>0.
\label{regH3}
\end{align}
In \eqref{regH2} and \eqref{regH3}, the constant $C$ may depend on $\|(\phi_0, \psi_0)\|_{\mathcal{V}^1}$, $\Omega$, $\Gamma$, $C_F$, $\widetilde{C}_F$, $\widehat{C}_F$, $C_G$, $\widetilde{C}_G$, $\widehat{C}_G$ and $\kappa$.

If the growth assumption (\textbf{A3}) is replaced  by (\textbf{A4}), then using the same argument as for \eqref{FFL2}, \eqref{FGm1}, \eqref{memuhi}
and Poincar\'e's inequality, we still have the estimate \eqref{memuhibb}. Besides, in analogy to \eqref{L2F}, \eqref{L2G}, we infer from \eqref{memuhibb} that
\begin{align}
&\|F'(\phi)\|_{L^2(\Omega)}+\|G'(\psi)\|_{L^2(\Gamma)}\nonumber\\
&\quad \leq C\big(\|\nabla \mu(t)\|_{L^2(\Omega)} + \|\nabla_\Gamma\mu_\Gamma(t)\|_{L^2(\Gamma)}
+ \|\phi\|_{H^r(\Omega)}+1\big),\quad \forall\, t\geq 0.
\label{FGFGL2}
\end{align}
Combining  \eqref{uniH1}, \eqref{memuhibb}, \eqref{FGFGL2} with Lemma \ref{esLepH2}  and the interpolation inequality \eqref{intpo}, we can conclude \eqref{regH2} as well as \eqref{regH3}, with the constant $C$ depending on $E(\phi_0, \psi_0)$, $\Omega$, $\Gamma$,  $C_F$, $\widetilde{C}_F$, $C_G$, $\widetilde{C}_G$ and $\kappa$.

\textbf{Case 2. $\kappa=0$}. We first observe that the estimate \eqref{regNpt} on temporal derivatives is still valid under (\textbf{A2}), since the derivation of \eqref{hi3} does not depend on $\kappa$. On the other hand, using (\textbf{A3}) we also have the estimate \eqref{memuhibb} for chemical potentials. Recalling the $H^2$-estimate \eqref{esH3k} for the Robin problem \eqref{CHEPk}, then we obtain
\begin{align}
\|\phi(t)\|_{H^2(\Omega)}\leq  C\left(\frac{1+t}{t}\right)^\frac12,
\quad \forall\, t>0,
\label{regH2k0a}
\end{align}
which together with \eqref{esH3ka} and \eqref{memuhibb} yields the estimate \eqref{regH2k0}.

The proof is complete.
\end{proof}

When the surface diffusion is present on the boundary $\Gamma$ (i.e., $\kappa>0$), based on Theorem \ref{main1} and Lemma \ref{comp} we can deduce the following further characterization of problem \eqref{CH}.

\bc\label{semigroup}
Under the assumptions of Theorem \ref{main1} (1), for any initial datum $(\phi_0, \psi_0)\in \mathcal{V}^1$, the unique global weak solution $(\phi, \psi)$ to problem
\eqref{CH} defines a strongly continuous semigroup $\mathfrak{S}(t): \mathcal{V}^1 \to \mathcal{V}^1$ such that
$$ \mathfrak{S}(t)(\phi_0, \psi_0)=(\phi(t), \psi(t)), \quad \forall \, t\geq 0.$$
Let $\mathcal{S}$ be the family of operators such that $\mathcal{S}:=\{\mathfrak{S}(t)\}_{t\geq 0}$. Then $\mathcal{S}$ is a dynamical system in the sense of \cite[Definition 4.1.1]{HEN}.
Besides, $E(\phi(t), \psi(t)):\mathcal{V}^1\to \mathbb{R}$ serves as a Lyapunov functional for $\mathcal{S}$.
\ec
\begin{proof}
First, we note that $\mathfrak{S}(0)=I\in C(\mathcal{V}^1; \mathcal{V}^1)$. Next, let $(\phi_i, \psi_i)$, $i=1,2$, be two weak solutions to problem \eqref{CH} subject to initial data $(\phi_{0i}, \psi_{0i})\in \mathcal{V}^1$ with $\langle\phi_{01} \rangle_\Omega=\langle\phi_{02}\rangle_\Omega$, $\langle\psi_{01} \rangle_\Gamma=\langle\psi_{02}\rangle_\Gamma$. Combining \eqref{conti} and Lemma \ref{comp}, we have
\begin{align}
&\|\phi_1(t)-\phi_2(t)\|_{H^1(\Omega)} + \|\psi_1(t)-\psi_2(t)\|_{H^1(\Gamma)}\nonumber\\
&\quad \leq C\|\phi_1(t)-\phi_2(t)\|_{H^3(\Omega)}^\frac12\|\phi_1(t)-\phi_2(t)\|_{(H^1(\Omega))^*}^\frac12\nonumber\\
&\qquad +C\|\psi_1(t)-\psi_2(t)\|_{H^3(\Gamma)}^\frac12\|\psi_1(t)-\psi_2(t)\|_{(H^1(\Gamma))^*}^\frac12\nonumber\\
&\quad \leq C\left(\frac{1+t}{t}\right)^\frac14 e^{Ct}\big(\|\phi_{01}-\phi_{02}\|_{(H^1(\Omega))^*}^\frac12+\|\psi_{01}-\psi_{02}\|_{(H^1(\Gamma))^*}^\frac12\big),
\quad \forall\, t>0,
\nonumber
\end{align}
which yields that $\mathfrak{S}(t)\in C(\mathcal{V}^1; \mathcal{V}^1)$ for any fixed $t>0$.
Besides, we infer from the fact $(\phi, \psi)\in L^2(0,T; \mathcal{V}^3)\cap H^1(0,T; (H^1(\Omega))^*\times (H^1(\Gamma))^*)$ and the interpolation theory (cf. \cite{LM68}) that $(\phi, \psi)\in C([0,T]; \mathcal{V}^1)$, which implies that for every $(\phi_0, \psi_0)\in \mathcal{V}^1$, the map $t\to \mathfrak{S}(t)(\phi_0, \psi_0)$ is continuous.
The total free energy $E(\phi, \psi)$ (see \eqref{E}) is well defined on $\mathcal{V}^1$ due to (\textbf{A3}) and it is uniformly bounded from below because of (\textbf{A2}). Moreover, thanks to the smoothing property Lemma \ref{comp}, the energy identity \eqref{BELa} holds for $t>0$. As a consequence, we see that $E(\phi, \psi)$ is a strict Lyapunov functional for the dynamic system $\mathcal{S}$.

The proof is complete.
\end{proof}

\subsection{The stationary problem}
Since we are interested in the long-time behavior of of problem \eqref{CH}, we proceed to investigate its stationary points $(\phi^*, \psi^*, \mu^*, \mu_\Gamma^*)$, which are characterized by the following equations
\begin{equation}
\left\{
\begin{aligned}
&\Delta \mu^*=0, &\text{in}\ \Omega,\\
&\mu^*=-\Delta \phi^*+F'(\phi^*),\ &\text{in}\ \Omega,\\
&\partial_\mathbf{n}\mu^*=0, &\text{on}\ \Gamma,\\
&\phi^*|_\Gamma=\psi^*,&\text{on}\ \Gamma,\\
&\Delta_\Gamma\mu_\Gamma^*=0, &\text{on}\ \Gamma,\\
&\mu_\Gamma^*=-\kappa\Delta_\Gamma \psi^*+\psi^*+\partial_\mathbf{n}\phi^*+G'(\psi^*), &\text{on}\ \Gamma.
\end{aligned}
\right.
\label{STA1}
\end{equation}
It easily follows that the chemical potentials $\mu^*$ and $\mu_\Gamma^*$ are just constants.
Then the system \eqref{STA1} reduces to a semilinear nonlocal second order elliptic problem
\begin{equation}
\left\{
\begin{aligned}
& -\Delta \phi^*+F'(\phi^*)=\lambda_1,
& \text{in}\ \Omega,\\
& \phi^*|_\Gamma=\psi^*,
&\text{on}\ \Gamma,\\
& -\kappa\Delta_\Gamma \psi^*+\psi^*+\partial_\mathbf{n}\phi^*+G'(\psi^*)=\lambda_2,
&\text{on}\ \Gamma,
\end{aligned}
\label{equi3}
\right.
\end{equation}
where the constants $\lambda_1$, $\lambda_2$ can be determined by testing the above system by the pair $(1,1)$ such that
\begin{equation}
\left\{
\begin{aligned}
&\lambda_1=-|\Omega|^{-1}|\Gamma|\langle \partial_\mathbf{n}\phi^*\rangle_\Gamma+ \langle F'(\phi^*)\rangle_\Omega,\\
&\lambda_2=\langle \partial_\mathbf{n}\phi^*\rangle_\Gamma+\langle\psi^*\rangle_\Gamma+\langle G'(\psi^*)\rangle_\Gamma.
\end{aligned}
\right.
\label{llambda}
\end{equation}
The relation \eqref{llambda} implies that $(\lambda_1, \lambda_2)$ should be solved together with the solution $(\phi^*, \psi^*)$.

Besides, associating the stationary problem \eqref{STA1} with our evolution problem \eqref{CH}, for any given initial datum $(\phi_0, \psi_0)$ as stated in Theorems \ref{main1} or \ref{main1a}, due to the mass conservation property, we have the following two constraints
\begin{align}
&\langle\phi^*\rangle_\Omega=\langle\phi_0\rangle_\Omega,
\quad \langle\psi^*\rangle_\Gamma=\langle\psi_0\rangle_\Gamma.
\label{equi2}
\end{align}
Then we denote the set of associate stationary points by
\begin{align}
\mathcal{E}=\big\{(\phi^*, \psi^*)\in \mathds{V}^2_\kappa:\ (\phi^*, \psi^*) \ \ \text{solves the system}\ \ \text{\eqref{equi3}--\eqref{equi2}}\big\}.
\nonumber
\end{align}
\begin{remark}
The constants $\lambda_1$, $\lambda_2$ serve as two Lagrange multipliers for the bulk/surface mass constraint \eqref{equi2}.
Let us consider the Lagrangian
$$
L((\phi, \psi), \lambda_1, \lambda_2)
=E(\phi, \psi)-\lambda_1|\Omega|\big(\langle\phi\rangle_\Omega-\langle\phi_0\rangle_\Omega\big)
-\lambda_2|\Gamma|\big(\langle\psi\rangle_\Gamma-\langle\psi_0\rangle_\Gamma\big).
$$
Suppose $(\phi^*, \psi^*)\in \mathcal{E}$, it is straightforward to check that
\begin{align}
&\left.\frac{\d}{\d\varepsilon} L((\phi^*, \psi^*)+\varepsilon (\zeta, \eta), \lambda_1, \lambda_2)\right|_{\varepsilon=0}\nonumber\\
&\quad = \int_\Omega \big(\nabla \phi^*\cdot\nabla \zeta +F'(\phi^*)\zeta-\lambda_1 \zeta\big)\,\d x\nonumber\\
&\qquad + \int_\Gamma \big(\kappa \nabla_\Gamma\psi^*\cdot\nabla_\Gamma \eta+\psi^*\eta+G'(\psi^*)\eta-\lambda_2\eta\big) \,\d S\nonumber\\
&\quad =0,\quad \forall\, (\zeta, \eta)\in \mathds{V}^1_\kappa.
\label{Fdre}
\end{align}
As a consequence, every stationary point $(\phi_*,\psi_*) \in \mathcal{E}$ is a critical point of the energy $E(\phi, \psi)$ subject to the mass constraint \eqref{equi2}.
\end{remark}
\begin{remark}\label{equalM}
The constants $\lambda_1$, $\lambda_2$ also satisfy some further relations.
Testing the system \eqref{equi3} by $(\phi^*,\psi^*)$ in $\mathcal{H}$, from \eqref{llambda} and \eqref{equi2} we observe that the pair $(\lambda_1, \lambda_2)$ fulfills the linear system
\begin{align}
\begin{pmatrix}
1 & 1 \\
\langle\phi_0\rangle_\Omega & \langle\psi_0\rangle_\Gamma
\end{pmatrix}
\begin{pmatrix}
|\Omega| & 0 \\
0 & |\Gamma|
\end{pmatrix}
\begin{pmatrix}
\lambda_1 \\
\lambda_2
\end{pmatrix}=
\begin{pmatrix}
l_1 \\
l_2
\end{pmatrix},
\label{equi4}
\end{align}
with
\begin{align}
&l_1=\int_\Omega F'(\phi^*) \,\d x+\int_\Gamma \big(\psi^*+G'(\psi^*)\big) \,\d S,\nonumber\\
&l_2=\int_\Omega \big(|\nabla \phi^*|^2 +F'(\phi^*)\phi^*\big) \,\d x
+\int_\Gamma\big(\kappa|\nabla_\Gamma \psi^*|^2+|\psi^*|^2+G'(\psi^*)\psi^*\big) \,\d S.\nonumber
\end{align}

(1) The condition $\langle\phi_0\rangle_\Omega\neq \langle\psi_0\rangle_\Gamma$ is necessary for the unique solvability of \eqref{equi4}. If $\langle\phi_0\rangle_\Omega=\langle\psi_0\rangle_\Gamma$ and $(\phi^*, \psi^*)\in \mathds{V}^1_\kappa$ is a stationary solution, then the compatibility condition $l_2=l_1\langle\phi_0\rangle_\Omega$ should be satisfied.

(2) When $\langle\phi_0\rangle_\Omega= \langle\psi_0\rangle_\Gamma:=M_0$, it is obvious that the stationary problem \eqref{equi3}--\eqref{equi2} admits a trivial solution $(\phi^*,\psi^*)=(M_0,M_0)$ with $\lambda_1= F'(M_0)$ and $\lambda_2=M_0+G'(M_0)$.
\end{remark}

Existence of the stationary point $(\phi^*, \psi^*)$ (i.e., nonemptyness of the set $\mathcal{E}$) can be obtained under suitable assumptions on the nonlinearities $F$, $G$.
For instance, by the direct method in calculus of variations, it is standard to prove the following result:
\begin{proposition}\label{exemini}
Let $\kappa\geq 0$, $\Omega\subset \mathbb{R}^d$ ($d=2,3$) be a bounded domain with smooth boundary $\Gamma$.
Assume that (\textbf{A1})--(\textbf{A3}) are satisfied.
Then for any $(\hat{\phi}, \hat{\psi})\in \mathds{V}^1_\kappa$, the energy functional $E(\phi, \psi)$ admits at least one minimizer $(\phi^*, \psi^*)$ in the set
$\{(\phi, \psi)\in \mathds{V}^1_\kappa:\ \langle\phi\rangle_\Omega=\langle \hat{\phi}\rangle_\Omega,
\ \langle\psi\rangle_\Gamma=\langle \hat{\psi}\rangle_\Gamma\}$
such that
\begin{align}
\int_\Omega \big(\nabla \phi^*\cdot\nabla \zeta +F'(\phi^*)\zeta\big)\,\d x+ \int_\Gamma \big(\kappa \nabla_\Gamma\psi^*\cdot\nabla_\Gamma \eta+\psi^*\eta+G'(\psi^*)\eta\big) \,\d S=0,
\nonumber
\end{align}
for any $(\zeta, \eta)\in \dot{\mathds{V}}^1_\kappa$.
\end{proposition}
\begin{remark}\label{stauni}
By the elliptic regularity theory, we deduce that $(\phi^*, \psi^*)\in \mathds{V}^2_\kappa$ and thus $(\phi^*, \psi^*) \in \mathcal{E}$ (where we take $(\hat{\phi}, \hat{\psi})=(\phi_0, \psi_0)$).
Moreover, if the potential functions $F, G$ are smooth, a bootstrap argument easily yields that $(\phi^*, \psi^*)\in C^\infty(\Omega)\times C^\infty(\Gamma)$.
\end{remark}

\subsection{Characterization of the $\omega$-limit set}
  The property that $\mathcal{E}$ is a nonempty set can also be guaranteed by a dynamical approach based on Theorems \ref{main1}, \ref{main1a} and Lemma \ref{comp} for the evolution problem \eqref{CH}.
  Let $(\phi_0, \psi_0)$ be an initial datum for problem \eqref{CH} considered in Theorems \ref{main1}, \ref{main1a} and $(\phi(t), \psi(t))$ be its corresponding global weak/strong solution.
  We introduce the associated $\omega$-limit set denoted by $\omega(\phi_0, \psi_0)$:
\begin{align}
&\omega(\phi_0, \psi_0):=\big\{(\phi^*, \psi^*)\in \mathds{V}^2_\kappa:\ \text{there exists}\ \{t_n\}\nearrow +\infty\ \text{such that}\nonumber\\
&\qquad\qquad \qquad \quad (\phi(t_n), \psi(t_n))\to (\phi^*, \psi^*)\ \text{in}\ \mathds{V}^{2}_\kappa\big\}.
\label{omega}
\end{align}
Then we have
\begin{proposition}\label{ome}
Assume that the assumptions of Theorem \ref{main1} (or Theorem \ref{main1a}) are satisfied.
For any initial datum $(\phi_0, \psi_0)$ therein, its $\omega$-limit set $\omega(\phi_0, \psi_0)$ is a nonempty, compact set in $\mathds{V}^2_\kappa$ and satisfies that $\omega(\phi_0, \psi_0)\subset \mathcal{E}$.
Furthermore,  the energy functional $E(\phi, \psi)$ is constant on $\omega(\phi_0, \psi_0)$.
\end{proposition}
\begin{proof}
For the case $\kappa>0$ and $(\phi_0, \psi_0)\in \mathcal{V}^1$ as in Theorem \ref{main1}, the conclusion immediately follows from Corollary \ref{semigroup} and well-known
results for dynamical systems (e.g., \cite[Theorem 4.3.3]{HEN} or \cite[Theorem 9.2.7]{CH}).
Below we present an argument that also applies to other cases under consideration.
From the uniform estimates obtained in Lemma \ref{comp} and the compact embedding theorem, we easily see that $\omega(\phi_0, \psi_0)$ is a nonempty, compact set in $\mathds{V}^2_\kappa$.
Next, thanks to Lemma \ref{comp}, the energy identity \eqref{BELa} holds for both global weak/strong solution when $t>0$. This monotonicity property and (\textbf{A2}) yields
\begin{align}
\lim_{t\to +\infty} E(\phi(t), \psi(t))=E_\infty,\label{limE}
\end{align}
for some finite constant $E_\infty$.
By the definition \eqref{omega}, it follows that $E(\phi, \psi)\equiv E_{\infty}$ on the set $\omega(\phi_0, \psi_0)$.

It remains to prove $\omega(\phi_0, \psi_0)\subset \mathcal{E}$. This follows from an argument in \cite{HT01}. For any $(\phi^*,\psi^*)\in \omega(\phi_0, \psi_0)$, it satisfies the mass constraint \eqref{equi2} and there exists an unbounded increasing sequence $t_n\to+\infty$ such that $\|(\phi(t_n), \psi(t_n))-(\phi^*, \psi^*)\|_{\mathds{V}_\kappa^2}\to 0$ as $n\to+\infty$.
Without loss of generality, we assume $t_{n+1}\geq
t_n+1,\ n\in\mathbb{N}$. Then from \eqref{unimu}, we see that
\begin{align}
&\lim_{n\to+\infty}\int_0^1\left(\|\nabla \mu(t_n+\tau)\|_{L^2(\Omega)}^2
+\|\nabla_\Gamma \mu_\Gamma(t_n+\tau)
\|_{L^2(\Gamma)}^2 \right)\, \d \tau=0,
\label{convmu1}
\end{align}
as well as
\begin{align}
&\lim_{n\to+\infty}\int_0^1\left(\|\phi_t(t_n+\tau)\|_{(H^1(\Omega))^*}^2
+\|\psi_t(t_n+\tau)
\|_{(H^1(\Gamma))^*}^2 \right)\, \d \tau=0.
\nonumber
\end{align}
The latter strong convergence property further implies
\begin{align}
 \|\phi(t_n+\tau_1)-\phi(t_n+\tau_2)\|_{(H^1(\Omega))^*}+\|\psi(t_n+\tau_1)-\psi(t_n+\tau_2)\|_{(H^1(\Gamma))^*}
 \rightarrow 0,
 \nonumber
 \end{align}
uniformly for all $\tau_1, \tau_2\in [0,1]$. Then from Lemma \ref{comp} and the interpolation inequality, it holds
\begin{align}
 \lim_{n\to+\infty}\|(\phi(t_n+\tau)-\phi^*, \psi(t_n+\tau)-\psi^*)\|_{\mathds{V}_\kappa^2}
 =0,\quad \forall\, \tau\in [0,1].
 \label{sestrc}
 \end{align}
For any $(\zeta, \eta)\in \dot{\mathds{V}}_\kappa^1$, using \eqref{convmu1}, \eqref{sestrc}, the Lebesgue dominated convergence
theorem and Poincar\'e's inequality, we deduce that
\begin{align}
&\left| \int_\Omega (\nabla \phi^* \cdot \nabla \zeta+ F'(\phi^*)\zeta-\langle F'(\phi^*)\rangle_\Omega\zeta)\,\d x\right. \nonumber\\
&\qquad \left. +\int_\Gamma \big[\kappa \nabla_\Gamma\psi^*\cdot\nabla_\Gamma \eta+\psi^*\eta+G'(\psi^*)\eta-(\langle\psi^*\rangle_\Gamma+\langle G'(\psi^*)\rangle_\Gamma)\eta\big] \,\d S \right|\nonumber\\
&\quad = \lim_{n\to+\infty}\left|\int^{1}_{0} \int_\Omega\left(\nabla \phi(t_n+\tau)\cdot \nabla \zeta+F'(\phi(t_n+\tau))\zeta-\langle F'(\phi(t_n+\tau))\rangle_\Omega\zeta\right) \,\d x \d \tau\right.\nonumber\\
&\qquad\qquad \ +\int_0^1\int_\Gamma \left[\kappa \nabla_\Gamma\psi(t_n+\tau)\cdot\nabla_\Gamma \eta+\psi(t_n+\tau)\eta+G'(\psi(t_n+\tau))\eta\right]\,\d S\d\tau\nonumber\\
&\qquad\qquad \ \left.-\int_0^1\int_\Gamma \left(\langle\psi(t_n+\tau)\rangle_\Gamma+\langle G'(\psi(t_n+\tau))\rangle_\Gamma\right)\eta \,\d S\d\tau \right|\nonumber\\
&= \lim_{n\to+\infty}\left|\int^{1}_{0}\int_\Omega \left(\mu(t_n+\tau)-\langle\mu(t_n+\tau)\rangle_\Omega\right)\zeta \,\d x\d \tau \right.\nonumber\\
&\qquad \qquad \ \left.
+\int^{1}_{0}\int_\Gamma \left(\mu_\Gamma(t_n+\tau)-\langle\mu_\Gamma(t_n+\tau)\rangle_\Gamma\right)\eta \,\d S\d \tau
\right|\nonumber\\
&\leq  \lim_{n\to+\infty}\int^{1}_{0}
  \|\mu(t_n+\tau)-\langle \mu(t_n+\tau)\rangle_\Omega\|_{L^2(\Omega)}\|\zeta\|_{L^2(\Omega)}\, \d\tau\nonumber\\
  &\qquad +\lim_{n\to+\infty} \int^{1}_{0}
  \|\mu_\Gamma(t_n+\tau)-\langle \mu_\Gamma(t_n+\tau)\rangle_\Gamma\|_{L^2(\Gamma)}\|\eta\|_{L^2(\Gamma)}\, \d\tau\nonumber\\
  &\leq C\lim_{n\to+\infty}\left(\int^{1}_{0}
  \|\nabla \mu(t_n+\tau)\|_{L^2(\Omega)}^2 \,\d\tau \right)^\frac12 \|\zeta\|_{L^2(\Omega)}\nonumber\\
  &\qquad
  +C\lim_{n\to+\infty}\left(\int^{1}_{0}
  \|\nabla_\Gamma \mu_\Gamma(t_n+\tau)\|_{L^2(\Gamma)}^2 \,\d\tau \right)^\frac12 \|\eta\|_{L^2(\Gamma)}\nonumber\\
  &= 0.\nonumber
\end{align}
As a consequence, we can conclude that $(\phi^*, \psi^*)\in \mathcal{E}$. The proof is complete.
\end{proof}

\begin{remark}
(1) Propositions \ref{exemini}, \ref{ome} imply that the set of stationary points $\mathcal{E}$ is nonempty under the assumptions (\textbf{A1}), (\textbf{A2}) together with either (\textbf{A3}) or (\textbf{A4}).

(2) The uniqueness of stationary points is unclear, since the energy $E(\phi, \psi)$ is allowed to be nonconvex.
It will be an interesting problem to study the structure of $\mathcal{E}$ and profiles of the stationary points (cf. \cite{Z86, GN95, WM1998} for the Cahn--Hilliard equation subject to classical Neumann boundary conditions).
\end{remark}

Since the set $\mathcal{E}$ may have a rather complicated structure, in general, we cannot conclude that each global weak/strong solution of problem \eqref{CH} obtained in Theorems \ref{main1}, \ref{main1a} converges to a single equilibrium as $t\to +\infty$, in other words, the nonempty set $\omega(\phi_0, \psi_0)$ consists of only one element.
This difficulty can be overcome by the well-known \L ojasiewicz--Simon approach (cf. \cite{S83, Ch03}) under the additional assumption (\textbf{AN}) on analyticity of the potential functions $F$, $G$. We refer to \cite{AW07,CFJ06, GaW08, GMS, GMS11, HR99, WH07, WZ04} for applications to the Cahn--Hilliard equation subject to various different type of boundary conditions.

Denote
$$
\mathcal{K}_2=\{(\phi, \psi)\in \mathds{V}_\kappa^2:\ \langle\phi\rangle_\Omega=\langle\hat{\phi}\rangle_\Omega,\ \langle\psi\rangle_\Gamma=\langle\hat{\psi}\rangle_\Gamma\},
$$
where the pair $(\hat{\phi}, \hat{\psi})\in \mathds{V}_\kappa^1$ is arbitrary but given.
By a similar argument as in \cite[Lemma 6.2]{CFJ06}, we can verify that the energy functional $E(\phi, \psi)$ is twice continuously Fr\'echet differentiable in $\mathcal{K}_2$
such that
for any $(\phi, \psi)\in \mathcal{K}_2$ and $(\zeta, \eta)\in \dot{\mathds{V}}^1_\kappa$,
\begin{align}
&\langle E'(\phi, \psi), (\zeta, \eta)\rangle_{(\mathds{V}^1_\kappa)^*, \mathds{V}_\kappa^1}\nonumber\\
&\quad
= \int_\Omega \left(\nabla \phi\cdot\nabla \zeta + F'(\phi)\zeta\right)\,\d x\nonumber\\
&\qquad + \int_\Gamma\left(\kappa\nabla_\Gamma\psi\cdot\nabla_\Gamma \eta+\psi\eta+G'(\psi)\eta\right) \,\d S,
\label{Fdre1}
\end{align}
while for any $(\phi, \psi)\in \mathcal{K}_2$ and $(\zeta_i, \eta_i)\in \dot{\mathds{V}}^1_\kappa$, $i=1,2$,
\begin{align}
&\langle E''(\phi, \psi)(\zeta_1, \eta_1), (\zeta_2, \eta_2)\rangle_{(\mathds{V}^1_\kappa)^*, \mathds{V}_\kappa^1}\nonumber\\
&\quad
= \int_\Omega \left(\nabla \zeta_1\cdot\nabla \zeta_2 + F''(\phi)\zeta_1\zeta_2\right)\,\d x\nonumber\\
&\qquad + \int_\Gamma\left(\kappa\nabla_\Gamma\eta_1\cdot\nabla_\Gamma \eta_2+\eta_1\eta_2+G''(\psi)\eta_1\eta_2\right) \,\d S.
\label{Fdre2}
\end{align}
Then we introduce the following gradient inequality of \L ojasiewicz--Simon type:
\begin{lemma}[\L ojasiewicz--Simon inequality]
\label{LS}
Suppose that the assumption (\textbf{AN}) is satisfied.
Let $(\phi^*,\psi^*)\in \mathcal{K}_2$ be a critical point of the energy functional $E(\phi, \psi)$ over $\mathcal{K}_2$.
There exist constants $\theta\in(0,\frac{1}{2}]$, $\Lambda>0$ and $\beta>0$ depending on $(\phi^*,\psi^*)$ such that for any $(\phi, \psi)\in \mathcal{K}_2$
and $\|(\phi, \psi)-(\phi_*,\psi_*)\|_{\mathds{V}_\kappa^1}< \beta$, it holds
\begin{align}
\Lambda \|E'(\phi, \psi)\|_{(\mathds{V}^1_\kappa)^*} \geq |E(\phi, \psi)-E(\phi_*, \psi_*)|^{1-\theta}.
\label{ls}
\end{align}
\end{lemma}
\begin{remark}\label{rm6.5}
The proof of Lemma \ref{LS} can be carried out in the same manner as \cite[Proposition 6.6]{CFJ06} by adapting the abstract result \cite[Corollary 3.11]{Ch03} (see also \cite{WZ04,WH07,GaW08}). The main difference here is that we have mass conservation both in the bulk $\Omega$ and on the boundary $\Gamma$, which leads to a different phase space to work with, i.e., $\mathcal{K}_2$ defined above. Besides, by the Sobolev embedding theorem $\mathds{V}_\kappa^2\hookrightarrow C(\overline{\Omega})$ for $d=2,3$, here we don't need any growth assumption on $F$, $G$. If we assume in addition that (\textbf{A3}) is satisfied, then in the statement of Lemma \ref{LS} we can replace $\mathcal{K}_2$ by the set $\{(\phi, \psi)\in \mathds{V}_\kappa^1:\ \langle\phi\rangle_\Omega=\langle\hat{\phi}\rangle_\Omega,\ \langle\psi\rangle_\Gamma=\langle\hat{\psi}\rangle_\Gamma\}$ with a weaker topology $\mathds{V}_\kappa^1$.
\end{remark}

\subsection{Proof of Theorem \ref{main2}: uniqueness of asymptotic limit as $t\to+\infty$}
We are in a position to prove Theorem \ref{main2}, by adapting the argument in \cite{HR99}.

Let us consider a trajectory $(\phi(t), \psi(t))$ defined by the global weak (or strong) solutions obtained in Theorems \ref{main1} or \ref{main1a}.
Denote by $E_\infty= E(\phi, \psi)|_{(\phi, \psi)\in \omega(\phi_0, \psi_0)}$ to be the limit value of energy as $t\to+\infty$ (see Proposition \ref{ome} and \eqref{limE}).
From \eqref{BELa}, we see that $E(\phi(t),\psi(t))\geq E_\infty$ for $t\geq 0$.
If there exists certain $t_1>0$ such that $E(\phi(t_1),\psi(t_1))= E_\infty$, then $\|\phi_t(t)\|_{(H^1(\Omega))^*}=\|\psi_t(t)\|_{(H^1(\Gamma))^*}=0$ for $t\geq t_1$ and thus the evolution stops.
Therefore, below we only consider the non-trivial case $E(\phi(t),\psi(t))> E_\infty$ for $t\geq 0$.

 For every element $(\phi_\infty, \psi_\infty)\in\omega(\phi_0, \psi_0)\subset \mathcal{K}_2$ (here we take $(\hat{\phi}, \hat{\psi})=(\phi_0, \psi_0)$ in the definition of the set $\mathcal{K}_2$), thanks to Lemma \ref{LS}, there exist constants  $\beta>0$, $\Lambda>0$ and $\theta\in(0,\frac{1}{2}]$ depending on $(\phi_\infty, \psi_\infty)$ such that
the gradient inequality \eqref{ls} holds for every
\begin{align}
 (\phi, \psi) \in \mathbf{B}_{\beta}(\phi_\infty, \psi_\infty) :=\Big\{(\phi, \psi)\in \mathcal{K}_2: \|(\phi, \psi)-(\phi_\infty, \psi_\infty)\|_{\mathds{V}_\kappa^2}<\beta\Big\}.\non
\end{align}
The union of  balls $\{\mathbf{B}_{\beta}(\phi_\infty, \psi_\infty):\ (\phi_\infty, \psi_\infty)\in\omega(\phi_0, \psi_0)\}$ forms an open cover of $\omega(\phi_0, \psi_0)$ and because of the compactness of $\omega(\phi_0, \psi_0)$ in $\mathds{V}^2_\kappa$ (cf. Proposition \ref{ome}), we can find a finite sub-covering $\{\mathbf{B}_{\beta_i}(\phi_\infty^i, \psi_\infty^i):i=1,2,...,m\}$ for $\omega(\phi_0, \psi_0)$ in $\mathds{V}^2_\kappa$, where the constants $\beta_i, \Lambda_i, \theta_i$ corresponding to $(\phi_\infty^i, \psi_\infty^i)$ in Lemma \ref{LS} are indexed by $i$.
From the definition of $\omega(\phi_0, \psi_0)$, there exists a sufficient large $t_0>0$ such that
\begin{align}
 (\phi(t), \psi(t))\in\mathcal{U}:=\bigcup_{i=1}^m\mathbf{B}_{\beta_i}(\phi_\infty^i, \psi_\infty^{i}), \quad \text{for}\;\; t\geq t_0.
 \end{align}
Taking
$$
\theta=\min_{i=1,...,m}\{\theta_i\}\in(0,\frac{1}{2}],\quad  \Lambda=\max_{i=1,...,m}\{\Lambda_i\},
$$
  we deduce from  Lemma \ref{LS} that for all $t\geq t_0$,
\begin{align}
\Lambda \|E'(\phi(t), \psi(t))\|_{(\mathds{V}^1_\kappa)^*}
\geq |E(\phi(t), \psi(t))-E_\infty|^{1-\theta}.\label{ls1}
\end{align}

Since $(\phi(t),\psi(t))\in \mathcal{V}^3$ if $\kappa>0$ and $(\phi(t),\psi(t))\in V^\frac52$ if $\kappa=0$ are uniformly bounded for $t\geq t_0$, then by the definition \eqref{Fdre1} and Poincar\'e's inequality, we have
\begin{align}
\|\nabla \mu(t)\|_{L^2(\Omega)}+\|\nabla_\Gamma \mu_\Gamma(t)\|_{L^2(\Gamma)}\geq C \|E'(\phi(t), \psi(t))\|_{(\mathds{V}^1_\kappa)^*},\quad \forall\, t\geq t_0.
\label{ls2}
\end{align}
Thus, it follows from \eqref{BELa}, \eqref{ls1} and \eqref{ls2} that for all $t\geq t_0$,
\begin{align}
&\int_t^{\infty} \big(\|\nabla \mu(\tau)\|_{L^2(\Omega)}^2
+\|\nabla_\Gamma \mu_\Gamma(\tau)\|_{L^2(\Gamma)}^2\big) \,\d\tau\nonumber\\
&\quad = E(\phi(t), \psi(t))-E_\infty\nonumber\\
&\quad \leq C\big(\|\nabla \mu(t)\|_{L^2(\Omega)}^2+\|\nabla_\Gamma\mu_\Gamma(t)\|_{L^2(\Gamma)}^2\big)^\frac{1}{2(1-\theta)}.
\label{Des}
\end{align}
Thanks to the elementary result \cite[Lemma 7.1]{FS00}, we deduce from \eqref{Des} that
\begin{align}
\int_{t_0}^{+\infty}\big(\|\nabla \mu(t)\|_{L^2(\Omega)}
+\|\nabla_\Gamma \mu_\Gamma(t)\|_{L^2(\Gamma)}\big)\,\d t<+\infty,\nonumber
\end{align}
which implies
$$
\phi_t\in L^1(t_0, +\infty; (H^1(\Omega))^*),\quad \psi_t\in L^1(t_0, +\infty; (H^1(\Gamma))^*).
$$
The above $L^1$-integrability property on time derivatives entails that the whole trajectory $(\phi(t), \psi(t))$ strongly converges to a single equilibrium  $(\phi_\infty, \psi_\infty)\in \omega(\phi_0, \psi_0)$ in $(H^1(\Omega))^*\times (H^1(\Gamma))^*$, and by definition $\omega(\phi_0, \psi_0)=\{(\phi_\infty, \psi_\infty)\}$.
Next, using the uniform-in-time estimate obtained in Lemma \ref{comp} and the interpolation inequality, we can conclude the strong convergence in higher-order norm, i.e., \eqref{conv1}.

Concerning the convergence rate, it follows from \eqref{ls1} that
\begin{align}
\label{ener1}
\frac{\d}{\d t}(E(\phi(t),\psi(t))-E_\infty)+\|\nabla \mu(t)\|_{L^2(\Omega)}^2+\|\nabla_\Gamma \mu_\Gamma(t)\|_{L^2(\Gamma)}^2\leq 0,\quad\forall\,t\geq t_0.
\end{align}
Similar to \eqref{Des}, we deduce that
\begin{align}
\big(E(\phi(t),\psi(t))-E_\infty\big)^{2(1-\theta)}
&\leq C(\|\nabla \mu(t)\|_{L^2(\Omega)}^2+\|\nabla_\Gamma \mu_\Gamma(t)\|_{L^2(\Gamma)}^2)\nonumber\\
&= - C\frac{\d}{\d t}\big(E(\phi(t),\psi(t))-E_\infty\big).
\nonumber
\end{align}
which easily implies the decay of total free energy (cf. \cite[Lemma 2.6]{HJ01})
\begin{align}
\nonumber E(\phi(t),\psi(t))-E_\infty\leq C(1+t)^{-\frac{1}{1-2\theta}},\quad \forall\, t\geq t_0.
\end{align}
Thus, we obtain that
\begin{align}
 &\int_t^{+\infty} \big(\|\phi_t(\tau)\|_{(H^1(\Omega))^*}+\|\psi_t(\tau)\|_{(H^1(\Gamma))^*} \big)\,\d\tau\nonumber\\
 &\quad = \int_t^{+\infty}\big(\|\nabla \mu(\tau)\|_{L^2(\Omega)}+\|\nabla_\Gamma \mu_\Gamma(\tau)\|_{L^2(\Gamma)}\big)\,\d\tau\nonumber\\
 &\quad = \sum_{j=0}^{+\infty} \int_{2^j t}^{2^{j+1}t} \big(\|\nabla \mu(\tau)\|_{L^2(\Omega)}+\|\nabla_\Gamma \mu_\Gamma(\tau)\|_{L^2(\Gamma)}\big)\,\d\tau \nonumber\\
 &\quad \leq C \sum_{j=0}^{+\infty} (2^jt)^{\frac{1}{2}}\left(\int_{2^j t}^{2^{j+1}t}\mathcal{D}(\tau)d\tau\right)^{\frac{1}{2}}\nonumber\\
 &\quad \leq Ct^\frac12(E(\phi(t),\psi(t))-E_\infty)^{\frac{1}{2}}\sum_{j=0}^{+\infty} 2^{j-1}\nonumber\\
 &\quad \leq C(1+t)^{-\frac{\theta}{1-2\theta}},\quad\forall\, t\geq t_0,\nonumber
 \end{align}
which yields the lower-order convergence rate:
\begin{align}
\|\phi(t)-\phi_{\infty}\|_{(H^1(\Omega))^*}+ \|\psi(t)-\psi_{\infty}\|_{(H^1(\Gamma))^*}
\leq  C(1+t)^{-\frac{\theta}{1-2\theta}},\quad\forall\, t\geq t_0.
\label{ratelow}
\end{align}
Then the higher-order convergence rate \eqref{convrate} can be deduced from  \eqref{BELa}, \eqref{regH3a}, \eqref{regH2k0} and \eqref{ratelow}, by employing the same energy method as in \cite[Part II, Section 4]{WH07}.

The proof of Theorem \ref{main2} is complete. \hfill $\square$

\subsection{Proof of Theorem \ref{main3}: stability criterion}

We proceed to prove Theorem \ref{main3}.

\textbf{Part I. Stability of local energy minimizers}.
The key observation is that by the \L ojasiewicz--Simon inequality \eqref{ls}, the ``drop" of the total free energy $E(\phi(t), \psi(t))$ can control the length of the trajectory $(\phi(t), \psi(t))$ in $(H^1(\Omega))^*\times(H^1(\Gamma))^*$.

Let $(\phi^*, \psi^*)\in \mathcal{K}_1$ be a given local energy minimizer and $\epsilon >0$ be an arbitrary small but fixed constant. We consider the initial datum $(\phi_0, \phi_0)\in \mathcal{K}_1$ that satisfies $\|(\phi_0, \psi_0)-(\phi^*, \psi^*)\|_{\mathcal{V}^1}<\delta$, where $\delta>0$ is a small constant to be determined later.

For any $(\phi_1, \psi_1), (\phi_2,\psi_2)\in \mathcal{K}_1$, $\|(\phi_i, \psi_i)-(\phi^*, \psi^*)\|_{\mathcal{V}^1}\leq 1$, $i=1,2$, by the growth assumption (\textbf{A3}) and the Sobolev embedding theorem, we have
$$ |E(\phi_1, \psi_1)-E(\phi_2,\psi_2)|\leq C_0\|(\phi_1, \psi_1)-(\phi_2,\psi_2)\|_{\mathcal{V}^1},$$
where $C_0>0$ depends on $\|(\phi^*, \psi^*)\|_{\mathcal{V}^1}$, $\Omega$, $\Gamma$, $\widehat{C}_F$, $\widehat{C}_G$ and $\kappa$.
Without loss of generality, we can take $C_0\geq 1$.

Set $$\delta_1=\frac{1}{4C_0} \min\big\{1, \sigma, \beta, \epsilon\big\},$$
where $\sigma>0$ is the constant in the definition of local energy minimizer $(\phi^*, \psi^*)$ and $\beta>0$ is the constant given in Lemma \ref{LS} depending on $(\phi^*, \psi^*)$.
By the continuity of the global weak solution $(\phi, \psi)\in C([0,+\infty); \mathcal{V}^1)$ (see Corollary \ref{semigroup}), we can find $t_0>0$ such that
$$t_0:=\inf_{\|(\phi_0, \psi_0)-(\phi^*, \psi^*)\|_{\mathcal{V}^1}\leq \delta_1} \min\Big\{t>0:\ \|(\phi(t), \psi(t))-(\phi^*, \psi^*)\|_{\mathcal{V}^1}=2\delta_1\Big\}.$$
If $t_0=+\infty$, then we can simply take $\delta =\delta_1$ and the proof is done.

Below we consider the case $t_0\in (0,+\infty)$.
For any $(\phi_0, \phi_0)\in \mathcal{K}_1$ with $\|(\phi_0, \psi_0)-(\phi^*, \psi^*)\|_{\mathcal{V}^1}\leq\delta_1$, by Lemma \ref{comp}, the corresponding solution $(\phi, \psi)$ satisfies
$$\|(\phi(t), \psi(t))\|_{\mathcal{V}^3}\leq C_1,\quad \forall\, t\geq t_0,$$
where $C_1>0$ depends on $\|(\phi^*, \psi^*)\|_{\mathcal{V}^1}$, $t_0$, $\Omega$, $\Gamma$, $C_F$, $\widetilde{C}_F$, $\widehat{C}_F$, $C_G$,  $\widetilde{C}_G$, $\widehat{C}_G$ and $\kappa$.

For $\delta \in (0,\delta_1]$, we denote
$$t_1:=\inf_{\|(\phi_0, \psi_0)-(\phi^*, \psi^*)\|_{\mathcal{V}^1}\leq \delta} \min\Big\{t>0:\ \|(\phi(t), \psi(t))-(\phi^*, \psi^*)\|_{\mathcal{V}^1}=  3\delta_1\Big\}.$$
It is obvious that $t_1> t_0$ by their definitions. Our aim is to prove there exists a $\delta\in (0,\delta_1]$ such that $t_1=+\infty$.

Suppose on the contrary that for any $\delta\in (0,\delta_1]$ the corresponding time $t_1$ defined above satisfies $t_1<+\infty$.
Then for every $\delta$, there exists an initial datum $(\phi_0, \psi_0)\in \mathcal{K}_1$, $\|(\phi_0, \psi_0)-(\phi^*, \psi^*)\|_{\mathcal{V}^1}\leq \delta$ such that
$\|(\phi(t_1), \psi(t_1))-(\phi^*, \psi^*)\|_{\mathcal{V}^1}= 3\delta_1$.
Next, for any $t\in [t_0, t_1]$, it holds
\begin{align}
&\|(\phi(t), \psi(t))-(\phi^*, \psi^*)\|_{\mathcal{V}^1}\nonumber\\
&\quad \leq  \|(\phi(t), \psi(t))-(\phi(t_0), \psi(t_0))\|_{\mathcal{V}^1} +\|(\phi(t_0), \psi(t_0))-(\phi^*, \psi^*)\|_{\mathcal{V}^1}\nonumber\\
&\quad \leq  C \|(\phi(t), \psi(t))-(\phi(t_0), \psi(t_0))\|_{(\mathcal{V}^1)^*}^\frac12\|(\phi(t), \psi(t))-(\phi(t_0), \psi(t_0))\|_{\mathcal{V}^3}^\frac12+ 2\delta_1\nonumber\\
&\quad \leq C_2\left(\int_{t_0}^{t} \big(\|\phi_t(\tau)\|_{(H^1(\Omega))^*}+\|\psi_t(\tau)\|_{(H^1(\Gamma))^*} \big)\,\d\tau\right)^\frac12+ 2\delta_1\nonumber\\
&\quad \leq C_2\left(\int_{0}^{t} \big(\|\phi_t(\tau)\|_{(H^1(\Omega))^*}+\|\psi_t(\tau)\|_{(H^1(\Gamma))^*} \big)\,\d\tau\right)^\frac12+ 2\delta_1,
\end{align}
where $C_2$ depends on $C_1$, $\Omega$ and $\Gamma$.
On the other hand, thanks to the definition of $\delta_1$, we see that the \L ojasiewicz--Simon inequality \eqref{ls} holds for the trajectory $(\phi(t), \psi(t))$ on $[0,t_1]$. Like in the proof of Theorem \ref{main2}, we only need to consider the case $E(\phi(t), \psi(t))> E(\phi^*,\psi^*)$ for any $t\in [0, t_1]$. Then we have
\begin{align}
& -\frac{\d}{\d t} (E(\phi(t), \psi(t))-E(\phi^*,\psi^*))^\theta\nonumber\\
&\quad  = (E(\phi(t), \psi(t))-E(\phi^*,\psi^*))^{\theta-1}\frac{\d}{\d t} E(\phi(t), \psi(t))\nonumber\\
&\quad \geq C_3\big(\|\phi_t(\tau)\|_{(H^1(\Omega))^*}+\|\psi_t(\tau)\|_{(H^1(\Gamma))^*} \big),
\end{align}
where $C_3$ depends on $(\phi^*, \psi^*)$, $\Omega$, $\Gamma$, $C_F$, $\widetilde{C}_F$, $\widehat{C}_F$, $C_G$,  $\widetilde{C}_G$, $\widehat{C}_G$ and $\kappa$.
As a consequence, we obtain
\begin{align}
&\|(\phi(t), \psi(t))-(\phi^*, \psi^*)\|_{\mathcal{V}^1}\nonumber\\
&\quad \leq C_2 C^\frac12_3 (E(\phi_0, \psi_0)-E(\phi^*,\psi^*))^\frac{\theta}{2}+2\delta_1\nonumber\\
&\quad \leq C_2 C^\frac12_3C_0^\frac{\theta}{2}\|(\phi_0, \psi_0)-(\phi^*,\psi^*)\|_{\mathcal{V}^1}^\frac{\theta}{2}+2\delta_1, \quad \forall\, t\in [t_0, t_1].\label{estt1}
\end{align}
In view of \eqref{estt1}, we can choose $\delta\in (0, \delta_1]$ sufficiently small such that $C_2 C^\frac12_3C_0^\frac{\theta}{2}\delta^\frac{\theta}{2}\leq \frac12\delta_1$.
Then it follows that
\begin{align}
&\|(\phi(t_1), \psi(t_1))-(\phi^*, \psi^*)\|_{\mathcal{V}^1}\leq \frac{5}{2}\delta_1<3\delta_1,\nonumber
\end{align}
which contradicts the definition of $t_1$.

Hence, for any given $\epsilon>0$, there exits a $\delta>0$ such that $t_1=+\infty$. This leads to the conclusion that $(\phi^*, \psi^*)$ is locally Lyapunov stable.

\medskip

\textbf{Part II. Instability of stationary points that are not local energy minimizers}.
Let $(\phi^*, \psi^*)\in \mathcal{K}_1$ be a stationary point that is a weak solution of the stationary problem \eqref{equi3}--\eqref{llambda}.
By assumption, $(\phi^*, \psi^*)$ does not attain any local minimum of $E(\phi, \psi)$ over the set $\mathcal{K}_1$.
Then there exists a sequence $\{(\phi_0^n, \psi_0^n)\}\subset \mathcal{K}_1$ such that
$$E(\phi_0^n, \psi_0^n)< E(\phi^*, \psi^*)\quad  \text{and}\quad  \|(\phi_0^n, \psi_0^n)-(\phi^*, \psi^*)\|_{\mathcal{V}^1}\to 0\quad\text{as}\ n\to +\infty.$$
Taking $(\phi_0^n, \psi_0^n)$ as the initial datum, thanks to Theorem \ref{main1}, problem \eqref{CH} admits a unique global weak solution $(\phi^n(t), \psi^n(t))$.
Then for every $n\in \mathbb{N}$ it follows from Theorem \ref{main2} that there exists a stationary point $(\phi_\infty^n, \psi_\infty^n)\in \mathcal{K}_1\cap \mathcal{V}^3$ such that
\begin{align}
\|(\phi^n(t), \psi^n(t))-(\phi^n_\infty, \psi^n_\infty)\|_{\mathcal{V}^1}\to 0\quad\text{as}\ t\to +\infty, \quad \forall\, n\in \mathbb{N}.\label{concon}
\end{align}

Suppose on the contrary that $(\phi^*, \psi^*)$ is Lyapunov stable. Then for $\epsilon_0=\frac12\beta$ (where $\beta$ is the constant in Lemma \ref{LS} determined by $(\phi^*, \psi^*)$, which is the critical point of $E(\phi, \psi)$ in $\mathcal{K}_1$, cf. Remark \ref{rm6.5}), there exists a constant $\delta_0>0$ such that if $\|(\phi_0, \psi_0)-(\phi^*, \psi^*)\|_{\mathcal{V}^1}<\delta_0$, it holds $$\sup_{t\in[0,+\infty)}\|(\phi(t), \psi(t))-(\phi^*, \psi^*)\|_{\mathcal{V}^1}<\epsilon_0.$$
Hence, we see that there exits an integer $n_0\in \mathbb{N}$ such that for all $n\geq n_0$, $\|(\phi_0^n, \psi_0^n)-(\phi^*, \psi^*)\|_{\mathcal{V}^1}<\delta_0$
and moreover, due to \eqref{concon}, it holds $\|(\phi_\infty^n, \psi_\infty^n)-(\phi^*, \psi^*)\|_{\mathcal{V}^1}<\beta$.
Applying Lemma \ref{LS}, we easily have $E(\phi_\infty^n, \psi_\infty^n)=E(\phi^*, \psi^*)$ for $n\geq n_0$, since $E'(\phi_\infty^n, \psi_\infty^n)=0$ by the definition of $(\phi_\infty^n, \psi_\infty^n)$. However, from the above construction and the decreasing property of the energy $E(\phi, \psi)$ along the trajectory $(\phi^n(t), \psi^n(t))$, we have $$E(\phi_\infty^n, \psi_\infty^n)\leq E(\phi_0^n, \psi_0^n)<E(\phi^*, \psi^*).$$
This leads to a contradiction.

The proof of Theorem \ref{main3} is complete. \hfill $\square$

\section{Appendix}
\label{APP}
\setcounter{equation}{0}

\subsection{Calculation of energy variations}

In this section, we provide detailed computations for the variation of the total action functional $\mathcal{A}^\text{total}$ (see \eqref{LAP}). \medskip

\textbf{Step 1. Variation in the bulk}. The deformation gradient of the bulk flow map $x(X,t)$ (see \eqref{flowmap}) is given by
$$\textsf{F}(X,t)=\frac{\partial x(X,t)}{\partial X},$$
which provides the information about how the
configuration is deformed with respect to the reference configuration. As a consequence,
the $d\times d$ tensor $\textsf{F}(X,t)$ carries all the information about structures and patterns of the material.
If we do a push forward for $\textsf{F}(X,t)$, i.e., expressing the deformation gradient by the Eulerian coordinate system such that
$$
\textsf{F}(X,t)=\widetilde{\textsf{F}}(x(X,t),t),
$$
then by the chain rule, one can see that $\widetilde{\textsf{F}}$ satisfies the following fundamental equation:
$$
\widetilde{\textsf{F}}_t + (\mathbf{w}\cdot \nabla_x) \widetilde{\textsf{F}}=(\nabla_x \mathbf{w}) \widetilde{\textsf{F}}.
$$
Besides, since the mass is conserved in the bulk (see \eqref{con1}), a change of coordinates yields that (cf. \cite[Chapter 2, Proposition 12]{FJ13})
\begin{equation}
\phi(x(X,t),t)=\frac{\phi_0(X)}{\mathrm{det} \textsf{F}},\quad  \forall\, X\in \Omega_0^X, \ t\geq 0.
\label{mass1}
\end{equation}

Now we compute the variation of the bulk action functional $\mathcal{A}^{\text{bulk}}$ with respective to the position $x=x(X,t)$.
The continuity equation \eqref{con1} and the relation \eqref{mass1} turn out to be fundamental kinematic assumptions that are necessary for the subsequent energetic variational analysis.
After pulling back to the Lagrangian coordinates, the bulk action functional $\mathcal{A}^{\text{bulk}}(x(X,t))$ (recall \eqref{ac}) can be written as
\begin{align}
\mathcal{A}^{\text{bulk}}(x(X,t))
& = \int_0^T  L^{\text{bulk}}(x(t)) \d x \d t\non\\
& = -\int_0^T\int_{\Omega_0^X} W_\text{b}\left(\frac{\phi_0(X)}{\mathrm{det} \textsf{F}},\ \nabla_X  \Big(\frac{\phi_0(X)}{\mathrm{det} \textsf{F}}\Big)\textsf{F}^{-1}\right)\mathrm{det} \textsf{F}\, \d X \d t.\label{bulkacL}
\end{align}
For any smooth vector $y(X,t)=\tilde{y}(x(X,t),t)$ satisfying $\tilde{y}\cdot \mathbf{n}=0$, we denote $$x^\epsilon= x+\epsilon y,\quad \textsf{F}^\epsilon(X,t)=\frac{\partial x^\epsilon(X,t)}{\partial X}.$$
Then we have
\begin{align}
& \left.\frac{\d}{\d\epsilon} \right|_{\epsilon=0}\mathcal{A}^{\text{bulk}} (x^\epsilon(X,t))\nonumber\\
&\  = -\left.\frac{\d}{\d\epsilon} \right|_{\epsilon=0}
\int_0^T\int_{\Omega_0^X} W_\text{b}\left(\frac{\phi_0(X)}{\mathrm{det} \textsf{F}^\epsilon},\
\nabla_X  \Big(\frac{\phi_0(X)}{\mathrm{det} \textsf{F}^\epsilon}\Big) (\textsf{F}^\epsilon)^{-1}\right)\mathrm{det} \textsf{F}^\epsilon\, \d X \d t\nonumber\\
&\  =-\int_0^T\int_{\Omega_0^X} \left[\frac{\partial W_\text{b}}{\partial\phi}\left(\frac{\phi_0(X)}{\mathrm{det} \textsf{F}},\ \nabla_X  \Big(\frac{\phi_0(X)}{\mathrm{det} \textsf{F}}\Big) \textsf{F}^{-1}\right)\right] \nonumber\\
&\qquad \quad \times \left(-\frac{\phi_0(X)}{(\mathrm{det} \textsf{F})^2}\right)
(\mathrm{det} \textsf{F}) \mathrm{tr}\left(\textsf{F}^{-1} \frac{\partial y}{\partial X}\right)\mathrm{det} \textsf{F}\, \d X\d t\nonumber\\
&\quad\ \ -\int_0^T\int_{\Omega_0^X} \left[\frac{\partial W_\text{b}}{\partial \nabla \phi}\left(\frac{\phi_0(X)}{\mathrm{det} \textsf{F}},\ \nabla_X  \Big(\frac{\phi_0(X)}{\mathrm{det} \textsf{F}}\Big) \textsf{F}^{-1}\right)\right]\nonumber\\
&\qquad \quad \times \left[-\nabla_X \Big(\frac{\phi_0(X)}{\mathrm{det} \textsf{F}} \Big)  \textsf{F}^{-1}\frac{\partial y}{\partial X}\textsf{F}^{-1}\right]\mathrm{det}\textsf{F}\, \d X \d t\nonumber\\
&\quad\ \ -\int_0^T\int_{\Omega_0^X} \left[\frac{\partial W_\text{b}}{\partial \nabla \phi}\left(\frac{\phi_0(X)}{\mathrm{det} \textsf{F}},\ \nabla_X  \Big(\frac{\phi_0(X)}{\mathrm{det} \textsf{F}}\Big) \textsf{F}^{-1}\right)\right]\nonumber\\
&\qquad \quad \times \left[ \nabla_X\left(-\frac{\phi_0(X)}{(\mathrm{det} \textsf{F})^2}(\mathrm{det} \textsf{F}) \mathrm{tr}\left(\textsf{F}^{-1} \frac{\partial y}{\partial X}\right)\right)\textsf{F}^{-1}\right] \mathrm{det} \textsf{F}\, \d X \d t\nonumber\\
&\quad\ \  -\int_0^T\int_{\Omega_0^X} W_\text{b}\left(\frac{\phi_0(X)}{\mathrm{det} \textsf{F}},\ \nabla_X  \Big(\frac{\phi_0(X)}{\mathrm{det} \textsf{F}}\Big) \textsf{F}^{-1}\right)(\mathrm{det} \textsf{F}) \mathrm{tr}\left(\textsf{F}^{-1} \frac{\partial y}{\partial X}\right)\, \d X \d t.\nonumber
\end{align}
After pushing forward the above result to the Eulerian coordinates and performing integration by parts, we have
\begin{align}
& \left.\frac{\d}{\d\epsilon} \right|_{\epsilon=0}\mathcal{A}^{\text{bulk}} (x+\epsilon y)\nonumber\\
&\ \ = -\int_0^T\int_{\Omega_t^x} \left[\frac{\partial W_\text{b}}{\partial\phi} (-\phi)(\nabla_x\cdot \tilde{y})+  \frac{\partial W_\text{b}}{\partial \nabla_x \phi}\cdot \left(-\frac{\partial\tilde{y}}{\partial x}\nabla_x \phi-\nabla_x(\phi \nabla_x \cdot \tilde{y})\right)\right.\nonumber\\
&\qquad \qquad +W_\text{b}(\phi, \nabla_x \phi) (\nabla_x \cdot \tilde{y})\Big]\, \d x\d t\nonumber\\
&\ \ = -\int_0^T\int_{\Omega_t^x} \left[\nabla_x\left(\phi\frac{\partial W_\text{b}}{\partial\phi} \right) + \nabla_x\cdot \left(\frac{\partial W_\text{b}}{\partial \nabla_x \phi} \otimes \nabla_x\phi\right) \right]\cdot \tilde{y}\, \d x\d t\nonumber\\
 &\qquad +\int_0^T\int_{\Omega_t^x} \left[
 \nabla_x \left(\phi\nabla_x\cdot \frac{\partial W_\text{b}}{\partial \nabla_x \phi} \right)+ \nabla_x W_\text{b}(\phi, \nabla_x \phi)\right]\cdot \tilde{y}\, \d x\d t\nonumber\\
 &\qquad - \int_0^T \int_{\Gamma_t^x}\left(-\phi\frac{\partial W_\text{b}}{\partial\phi}+W_\text{b}(\phi, \nabla_x \phi)+ \phi \nabla_x \cdot \frac{\partial W_\text{b}}{\partial \nabla_x \phi} \right)(\tilde{y}\cdot \mathbf{n})\, \d S_x \d t\nonumber\\
 &\qquad + \int_0^T \int_{\Gamma_t^x} \left[\left(\frac{\partial W_\text{b}}{\partial \nabla_x \phi}\cdot \mathbf{n}\right)\phi(\nabla_x\cdot\tilde{y}) + \left(\frac{\partial W_\text{b}}{\partial \nabla_x \phi}\cdot \mathbf{n}\right) (\tilde{y}\cdot\nabla_x)\phi\right]\, \d S_x \d t\nonumber\\
 &\ \ = -\int_0^T\int_{\Omega_t^x} (\phi \nabla_x \mu_\text{b}) \cdot \tilde{y}\, \d x\d t
 + \int_0^T \int_{\Gamma_t^x} \left(\frac{\partial W_\text{b}}{\partial \nabla_x \phi}\cdot \mathbf{n}\right)\nabla_x\cdot( \phi\tilde{y})\, \d S_x \d t,\label{bdaa}
\end{align}
where we recall the definition of the bulk chemical potential
\begin{equation}
\mu_\text{b}=\frac{\delta W_\text{b}(\phi, \nabla_x\phi)}{\delta \phi}
=-\nabla_x\cdot \frac{\partial W_\text{b}}{\partial \nabla_x \phi}+\frac{\partial W_\text{b}}{\partial\phi},\non
\end{equation}
and use the following direct computation
\begin{align}
&\nabla_x\left(\phi\frac{\partial W_\text{b}}{\partial\phi} \right) + \nabla_x\cdot \left(\frac{\partial W_\text{b}}{\partial \nabla_x \phi} \otimes \nabla_x\phi\right)-
 \nabla_x \left(\phi\nabla_x\cdot \frac{\partial W_\text{b}}{\partial \nabla_x \phi} \right)\nonumber\\
 &\qquad  -\nabla_x W_\text{b}(\phi, \nabla_x \phi)\nonumber\\
 &\ \ = \frac{\partial W_\text{b}}{\partial\phi} \nabla_x \phi+ \phi \nabla_x\frac{\partial W_\text{b}}{\partial\phi}+ \left(\nabla_x\cdot \frac{\partial W_\text{b}}{\partial \nabla_x \phi}\right)\nabla_x\phi+ \frac{\partial W_\text{b}}{\partial \nabla_x \phi}\nabla_x^2 \phi\nonumber\\
 &\qquad -\left(\nabla_x\cdot \frac{\partial W_\text{b}}{\partial \nabla_x \phi}\right)\nabla_x\phi-\nabla_x\left(\nabla_x\cdot \frac{\partial W_\text{b}}{\partial \nabla_x \phi}\right)\phi\nonumber\\
 &\qquad -\frac{\partial W_\text{b}}{\partial\phi} \nabla_x \phi-\frac{\partial W_\text{b}}{\partial \nabla_x \phi}\nabla_x^2 \phi\nonumber\\
 &\ \ = \phi\nabla_x \left(-\nabla_x\cdot \frac{\partial W_\text{b}}{\partial \nabla_x \phi}+\frac{\partial W_\text{b}}{\partial\phi}\right).\nonumber
\end{align}

\medskip

 \textbf{Step 2. Variation on the boundary}. Within this step, we only consider the case that the boundary $\Gamma\subset \mathbb{R}^3$ is a closed two dimensional surface. The computations can be done in a similar way when $\Gamma\subset \mathbb{R}^2$ is a closed curve.
 Recalling \cite[Definition 2.9]{KLG}, we introduce the flow map on the surface $\Gamma$
 \begin{equation}
 \left\{
 \begin{aligned}
 &\frac{\d x_\text{s}}{\d t}=(x_\text{s})_t(X, t)=\mathbf{v}_\text{s}(x_\text{s}(X,t),t),\\
 &x_\text{s}|_{t=0}=X_\text{s},
 \end{aligned}
 \right.
 \nonumber
 \end{equation}
 where $X_\text{s}=(X_1, X_2 ,X_3)^T\in \Gamma_0^X$, $x_\text{s}=(x_1, x_2 ,x_3)^T\in \Gamma_t^x$, $\mathbf{v}_\text{s}$ is a velocity determined by the surface flow map $x_\text{s}$.
 The corresponding deformation tensor is given by
 $$\textsf{F}_\text{s}=\frac{\partial x_\text{s}(X_\text{s},t)}{\partial X_\text{s}}.$$
 Here and after within Step 2, we drop the subscript ``s" that stands for ``surface" for the sake of simplicity.

  Next, we introduce the partition of unity (cf. e.g., \cite[Section 2.4]{KLG}): there are $\Gamma_m\subset \Gamma_0^X$, $\Phi_m\in C^\infty(\mathbb{R}^2)$, $U_m\subset \mathbb{R}^2$, $\Psi_m\in C^\infty(\mathbb{R}^3)$, $m=1,2,...,N$ such that
 \begin{align*}
 &\bigcup_{m=1}^N\Gamma_m=\Gamma_0^X,\quad \Gamma_m=\Phi_m(U_m),\quad \bigcup_{m=1}^N U_m:=U,\\
 & \text{supp}\Psi_m\subset \Gamma_m,\quad \|\Psi_m\|_{L^\infty}=1,\quad \sum_{m=1}^N\Psi_m=1\ \text{on} \ \Gamma_0^X.
 \end{align*}
 For arbitrary $X\in \Gamma_0^X$, assume that $X\in \Gamma_m$ for some $m$. Since we can write $X=\Phi_m(\xi)$ for some $\xi=(\xi_1,\xi_2)^T\in U_m$, we set
 $$\hat{x}=\hat{x}(\xi,t)=x(\Phi_m(\xi),t)=x(X, t),$$
 which implies
 \begin{equation}
 \left\{
 \begin{aligned}
 &\frac{\d \hat{x}}{\d t}=\hat{x}_t(\xi, t)=\mathbf{v}(\hat{x}(\xi,t),t),\\
 &\hat{x}|_{t=0}=\Phi_m(\xi).
 \end{aligned}
 \right.
 \nonumber
 \end{equation}
 Denote $\Phi:=\Phi_m$ if $X\in \Gamma_m$. Then for every $X\in\Gamma_0^X$, the surface flow map $\hat{x}=\hat{x}(\xi, t)$ in local coordinates satisfies
 \begin{equation}
 \left\{
 \begin{aligned}
 &\frac{\d \hat{x}}{\d t}=\hat{x}_t(\xi, t)=\mathbf{v}(\hat{x}(\xi,t),t),\\
 &\hat{x}|_{t=0}=\Phi(\xi).
 \end{aligned}
 \right.
 \nonumber
 \end{equation}
 Now we define the corresponding deformation tensor
 \begin{align}
 \widehat{\textsf{F}}(\xi,t):=\frac{\partial \hat{x}(\xi,t)}{\partial \xi}=
 \begin{pmatrix}
 \displaystyle{\frac{\partial\hat{x}_1}{\partial \xi_1}} & \displaystyle{\frac{\partial\hat{x}_1}{\partial \xi_2}}\\[1em]
 \displaystyle{\frac{\partial\hat{x}_2}{\partial \xi_1}} & \displaystyle{\frac{\partial\hat{x}_2}{\partial \xi_2}}\\[1em]
 \displaystyle{\frac{\partial\hat{x}_3}{\partial \xi_1}} & \displaystyle{\frac{\partial\hat{x}_3}{\partial \xi_2}}
 \end{pmatrix}.
 \nonumber
 \end{align}
 For any function $f(\cdot, \cdot)\in C(\mathbb{R}^3\times \mathbb{R})$, using the change of variables, we have (see \cite[Section 2.4]{KLG})
 \begin{align}
 &\int_{\Gamma_t^x} f(x,t) \d S_x\nonumber\\
 &\ \ =\int_{\Gamma_0^X} f(x(X,t),t) \mathrm{det} \textsf{F}\, \d S_X\nonumber\\
 &\ \ =\sum_{m=1}^N\int_{\Gamma_m} \Psi_m(X)f(x(X,t),t) \mathrm{det} \textsf{F}\, \d S_X \nonumber\\
 &\ \ =\sum_{m=1}^N\int_{U_m} \Psi_m(\Phi_m(\xi))f(x(\Phi_m(\xi),t),t) \mathrm{det} \textsf{F} \sqrt{\mathrm{det}\big((\nabla_\xi \Phi_m)^T(\nabla_\xi \Phi_m)\big)}\, \d \xi\nonumber\\
 &\ \ =\sum_{m=1}^N\int_{U_m} \Psi_m(\Phi_m(\xi))f(x(\Phi_m(\xi),t),t) \sqrt{\mathrm{det}\big(\widehat{\textsf{F}}^T\widehat{\textsf{F}}\big)}\, \d \xi\nonumber\\
 &\ \ :=\int_{U} \widehat{\Psi}(\xi)f(\hat{x}(\xi,t),t) \sqrt{\mathrm{det}\big(\widehat{\textsf{F}}^T\widehat{\textsf{F}}\big)}\, \d \xi.\nonumber
 \end{align}
Besides, thanks to the continuity equation \eqref{con2} on the boundary, using the surface flow map $\hat{x}=\hat{x}(\xi,t)$, we have (cf. \cite[Lemma 3.3]{KLG})
\begin{equation}
\phi(\hat{x}(\xi,t),t)=\frac{\tilde{\phi}_0(\xi)}{\sqrt{\mathrm{det} \big(\widehat{\textsf{F}}^T\widehat{\textsf{F}}\big)(\xi,t)}},\quad  \forall\, \xi\in U, \ t\geq 0,
\label{mass2}
\end{equation}
where $$\tilde{\phi}_0(\xi)=\phi_0(\hat{x}(\xi,0))\sqrt{\mathrm{det} \big(\widehat{\textsf{F}}^T\widehat{\textsf{F}}\big)(\xi,0)}.$$
Consider the action functional on the boundary \eqref{ac} (pulled back to the Lagrangian coordinates and written in local coordinates)
\begin{align}
\mathcal{A}^{\text{surf}}(\hat{x}(\xi,t))
&= -\int_0^T\int_{U} \widehat{\Psi}(\xi) \widehat{W}_\text{s} \sqrt{\mathrm{det}\big(\widehat{\textsf{F}}^T\widehat{\textsf{F}}\big)}\, \d \xi\d t,\nonumber
\end{align}
where
\begin{align}
\widehat{W}_\text{s} := W_\text{s}\left(\frac{\tilde{\phi}_0(\xi)}{\sqrt{\mathrm{det}\big( \widehat{\textsf{F}}^T\widehat{\textsf{F}}\big)(\xi,t)} },\ \nabla_\xi\left(  \frac{\tilde{\phi}_0(\xi)}{\sqrt{\mathrm{det}\big( \widehat{\textsf{F}}^T\widehat{\textsf{F}}\big)(\xi,t)} } \right) (\widehat{\textsf{F}}^T\widehat{\textsf{F}})^{-1}\widehat{\textsf{F}}^T \right).
\nonumber
\end{align}
For any smooth $y(X,t)=\tilde{y}(x(X,t),t)$ (with $X\in \Gamma_0^X$) satisfying $\tilde{y}\cdot \mathbf{n}=0$,
we denote
\begin{align}
&x^\epsilon= x+\epsilon y,\quad \textsf{F}^\epsilon(X,t)=\frac{\partial x^\epsilon(X,t)}{\partial X},\nonumber
\end{align}
while in the local coordinates, we  write
\begin{align*}
&\hat{x}^\epsilon(\xi, t)= \hat{x}(\xi, t)+\epsilon \hat{y}(\xi, t)=x^\epsilon(\Phi_m(X),t),\\[1em]
&\widehat{\textsf{F}}^\epsilon(\xi,t):=\frac{\partial \hat{x}^\epsilon(\xi,t)}{\partial \xi}=
 \begin{pmatrix}
 \displaystyle{\frac{\partial\hat{x}^\epsilon_1}{\partial \xi_1}} & \displaystyle{\frac{\partial\hat{x}^\epsilon_1}{\partial \xi_2}}\\[1em]
 \displaystyle{\frac{\partial\hat{x}^\epsilon_2}{\partial \xi_1}} & \displaystyle{\frac{\partial\hat{x}^\epsilon_2}{\partial \xi_2}}\\[1em]
 \displaystyle{\frac{\partial\hat{x}^\epsilon_3}{\partial \xi_1}} & \displaystyle{\frac{\partial\hat{x}^\epsilon_3}{\partial \xi_2}}
 \end{pmatrix}
 \end{align*}
 and correspondingly,
 \begin{align}
 &\widehat{W}^\epsilon_\text{s} := W_\text{s}\left(\frac{\tilde{\phi}_0(\xi)}{\sqrt{\mathrm{det}\big( (\widehat{\textsf{F}}^\epsilon)^T\widehat{\textsf{F}}^\epsilon\big)(\xi,t)} },\ \nabla_\xi\left(  \frac{\tilde{\phi}_0(\xi)}{\sqrt{\mathrm{det} \big((\widehat{\textsf{F}}^\epsilon)^T\widehat{\textsf{F}}^\epsilon\big)(\xi,t)} } \right) \Big((\widehat{\textsf{F}}^\epsilon)^T \widehat{\textsf{F}}^\epsilon\Big)^{-1}(\widehat{\textsf{F}}^\epsilon)^T \right).\nonumber
\end{align}
Using the facts
\begin{align}
&\left.\frac{\d}{\d\epsilon} \right|_{\epsilon=0}\sqrt{\mathrm{det}\big((\widehat{\textsf{F}}^\epsilon)^T\widehat{\textsf{F}}^\epsilon\big)}\nonumber\\
&\quad =\frac12 \sqrt{\mathrm{det}\big(\widehat{\textsf{F}}^T\widehat{\textsf{F}}\big)}\mathrm{tr}\left[ \big(\widehat{\textsf{F}}^T\widehat{\textsf{F}}\big)^{-1}\left( \Big(\frac{\partial \hat{y}}{\partial \xi}\Big)^T \widehat{\textsf{F}}+ \widehat{\textsf{F}}^T\Big(\frac{\partial \hat{y}}{\partial \xi}\Big)\right)\right]\nonumber
\end{align}
and
\begin{align}
&\left.\frac{\d}{\d\epsilon} \right|_{\epsilon=0}\Big((\widehat{\textsf{F}}^\epsilon)^T \widehat{\textsf{F}}^\epsilon\Big)^{-1}(\widehat{\textsf{F}}^\epsilon)^T\nonumber\\
&\quad = \Big(\widehat{\textsf{F}}^T \widehat{\textsf{F}}\Big)^{-1}\Big(\frac{\partial \hat{y}}{\partial \xi}\Big)^T
- \Big(\widehat{\textsf{F}}^T \widehat{\textsf{F}}\Big)^{-1}\left[ \Big(\frac{\partial \hat{y}}{\partial \xi}\Big)^T \widehat{\textsf{F}}+ \widehat{\textsf{F}}^T\Big(\frac{\partial \hat{y}}{\partial \xi}\Big)\right]\Big(\widehat{\textsf{F}}^T \widehat{\textsf{F}}\Big)^{-1}\widehat{\textsf{F}}^T\nonumber\\
&\quad =- \Big(\widehat{\textsf{F}}^T \widehat{\textsf{F}}\Big)^{-1} \widehat{\textsf{F}}^T\Big(\frac{\partial \hat{y}}{\partial \xi}\Big)\Big(\widehat{\textsf{F}}^T \widehat{\textsf{F}}\Big)^{-1}\widehat{\textsf{F}}^T,
\nonumber
\end{align}
we compute that
\begin{align}
& \left.\frac{\d}{\d\epsilon} \right|_{\epsilon=0}\mathcal{A}^{\text{surf}} (\hat{x}^\epsilon(\xi,t))\nonumber\\
&\quad  = -\left.\frac{\d}{\d\epsilon} \right|_{\epsilon=0}
\int_0^T\int_{U} \widehat{\Psi}(\xi) \widehat{W}_\text{s}^\epsilon \sqrt{\mathrm{det}\big((\widehat{\textsf{F}}^\epsilon)^T\widehat{\textsf{F}}^\epsilon\big)}\, \d \xi\d t\nonumber\\
&\quad = -\int_0^T\int_{U} \widehat{\Psi}(\xi) \frac{\partial W_\text{s}}{\partial \phi} \Big(-\tilde{\phi}_0(\xi)\Big)\frac12 \mathrm{tr}\left[ \big(\widehat{\textsf{F}}^T\widehat{\textsf{F}}\big)^{-1}\left( \Big(\frac{\partial \hat{y}}{\partial \xi}\Big)^T \widehat{\textsf{F}}+ \widehat{\textsf{F}}^T\Big(\frac{\partial \hat{y}}{\partial \xi}\Big)\right)\right] \, \d \xi\d t\nonumber\\
&\qquad -\int_0^T\int_{U} \widehat{\Psi}(\xi) \frac{\partial W_\text{s}}{\partial\nabla_\Gamma^x \phi}\cdot \nabla_\xi\left\{  -\frac{\tilde{\phi}_0(\xi)}{\sqrt{\mathrm{det}\big(\widehat{\textsf{F}}^T\widehat{\textsf{F}}\big)}} \frac12 \mathrm{tr}\left[ \big(\widehat{\textsf{F}}^T\widehat{\textsf{F}}\big)^{-1}\left( \Big(\frac{\partial \hat{y}}{\partial \xi}\Big)^T \widehat{\textsf{F}}+ \widehat{\textsf{F}}^T\Big(\frac{\partial \hat{y}}{\partial \xi}\Big)\right)\right] \right\}\nonumber\\
&\qquad\qquad  \times (\widehat{\textsf{F}}^T\widehat{\textsf{F}})^{-1}\widehat{\textsf{F}}^T \sqrt{\mathrm{det}\big(\widehat{\textsf{F}}^T\widehat{\textsf{F}}\big)}\, \d \xi\d t
\nonumber\\
&\qquad -\int_0^T\int_{U} \widehat{\Psi}(\xi) \frac{\partial W_\text{s}}{\partial\nabla_\Gamma^x \phi}\cdot
\nabla_\xi\left(\frac{\tilde{\phi}_0(\xi)}{\sqrt{\mathrm{det}\big(\widehat{\textsf{F}}^T\widehat{\textsf{F}}\big)}} \right)\left[-\Big(\widehat{\textsf{F}}^T \widehat{\textsf{F}}\Big)^{-1} \widehat{\textsf{F}}^T\Big(\frac{\partial \hat{y}}{\partial \xi}\Big)\Big(\widehat{\textsf{F}}^T \widehat{\textsf{F}}\Big)^{-1}\widehat{\textsf{F}}^T\right]
\nonumber\\
&\qquad\qquad  \times \sqrt{\mathrm{det}\big(\widehat{\textsf{F}}^T\widehat{\textsf{F}}\big)}
\, \d \xi\d t
\nonumber\\
&\qquad -\int_0^T\int_{U} \widehat{\Psi}(\xi)\widehat{W}_\text{s}\frac12 \sqrt{\mathrm{det}\big(\widehat{\textsf{F}}^T\widehat{\textsf{F}}\big)}\mathrm{tr}\left[ \big(\widehat{\textsf{F}}^T\widehat{\textsf{F}}\big)^{-1}\left( \Big(\frac{\partial \hat{y}}{\partial \xi}\Big)^T \widehat{\textsf{F}}+ \widehat{\textsf{F}}^T\Big(\frac{\partial \hat{y}}{\partial \xi}\Big)\right)\right]\, \d \xi\d t.\nonumber
\end{align}
Then pushing forward the above result to the Eulerian coordinates and performing integration
by parts, we infer from \cite[Lemma 2.6]{KLG} and $\tilde{y}\cdot \mathbf{n}=0$ that
\begin{align}
& \left.\frac{\d}{\d\epsilon} \right|_{\epsilon=0}\mathcal{A}^{\text{surf}} (x^\epsilon)\nonumber\\
&\quad
=-\int_0^T\int_{\Gamma_t^x}\left[ \frac{\partial W_\text{s}}{\partial \phi} (-\phi)(\nabla_\Gamma^x \cdot \tilde{y})+ W_\text{s}(\phi, \nabla_\Gamma^x\phi)(\nabla_\Gamma^x \cdot \tilde{y}) \right]\, \d S_x\d t
\nonumber\\
&\qquad
-\int_0^T\int_{\Gamma_t^x} \frac{\partial W_\text{s}}{\partial\nabla^x_\Gamma \phi} \cdot
\left[\nabla_\Gamma^x\big(-\phi\nabla_\Gamma^x\cdot \tilde{y}\big)-\nabla_\Gamma^x\phi\nabla_\Gamma^x \tilde{y}\right]\,
\d S_x\d t
\nonumber\\
&\quad = -\int_0^T\int_{\Gamma_t^x} \left[\nabla_\Gamma^x\left(\phi\frac{\partial W_\text{s}}{\partial \phi} \right) +\nabla_\Gamma^x\cdot\left(\frac{\partial W_\text{s}}{\partial\nabla^x_\Gamma \phi}\otimes \nabla_\Gamma^x \phi\right)  \right] \cdot \tilde{y}\, \d S_x\d t\nonumber\\
&\qquad +\int_0^T\int_{\Gamma_t^x} \left[ \nabla_\Gamma^x\cdot \left(\phi \nabla_\Gamma^x\cdot\frac{\partial W_\text{s}}{\partial\nabla^x_\Gamma \phi}\right) +  \nabla_\Gamma^x W_\text{s}(\phi, \nabla_\Gamma^x\phi)\right]\cdot \tilde{y} \, \d S_x\d t
\nonumber\\
&\quad = -\int_0^T\int_{\Gamma_t^x} (\phi \nabla_\Gamma^x \mu_\text{s}) \cdot \tilde{y}\, \d S_x \d t,
 \label{bdab}
\end{align}
where $\mu_\text{s}$ stands for the chemical potential on the boundary such that
\begin{align}
\mu_\text{s} =\frac{\delta W_\text{s}(\phi, \nabla_\Gamma^x\phi)}{\delta \phi}=-\nabla_\Gamma^x\cdot \frac{\partial W_\text{s}}{\partial \nabla_\Gamma^x \phi}+\frac{\partial W_\text{s}}{\partial\phi}.\non
\end{align}

\textbf{Step 3. Variation of the total action functional}.
As a consequence of \eqref{bdaa} and \eqref{bdab}, for the total action functional $\mathcal{A}^{\text{total}}=\mathcal{A}^{\text{bulk}}+\mathcal{A}^{\text{surf}}$, we deduce that
\begin{align}
&\delta_{(x, x_\text{s})} \mathcal{A}^{\text{total}}\nonumber\\
 &\ \ :=\left.\frac{\d}{\d\epsilon} \right|_{\epsilon=0}\mathcal{A}^{\text{bulk}} (x+\epsilon y) +\left.\frac{\d}{\d\epsilon} \right|_{\epsilon=0}\mathcal{A}^{\text{surf}} (x_\text{s}+\epsilon y_\text{s})  \nonumber\\
 &\ \ \ = -\int_0^T\int_{\Omega_t^x} (\phi \nabla_x \mu_\text{b}) \cdot \tilde{y}\, \d x \d t
 + \int_0^T \int_{\Gamma_t^x} \left(\frac{\partial W_\text{b}}{\partial \nabla_x \phi}\cdot \mathbf{n}\right)\nabla_x\cdot( \phi\tilde{y})\, \d S_x \d t\non\\
 &\quad\ \ -\int_0^T\int_{\Gamma_t^x} (\phi \nabla_\Gamma^x \mu_\text{s}) \cdot \tilde{y}_\text{s}\, \d S_x\d t.
 \label{LAPa}
\end{align}
It remains to treat the second term on the right-hand side of \eqref{LAPa}.
From equations \eqref{con1} and \eqref{con2a}, we have the following variational relations with $\delta x=y(X,t)=\tilde{y}(x(X,t),t)$ and $\delta x_\text{s}=y_\text{s}(X,t)=\tilde{y}_\text{s}(x_\text{s}(X,t),t)$:
\begin{equation}
\begin{aligned}
&\delta \phi+\nabla_x\cdot( \phi\delta x)=0,
&\text{in}\ \Omega,\nonumber\\
&\delta (\phi|_\Gamma) + \nabla_\Gamma^x\cdot[(\phi|_\Gamma) \delta x_\text{s}]=0,
&\text{on}\ \Gamma.\nonumber
\end{aligned}
\end{equation}
Besides, for sufficiently smooth phase function $\phi$, we infer that $\partial_t(\phi|_\Gamma)=(\partial_t\phi)|_\Gamma$, which implies $\delta (\phi|_\Gamma)=(\delta \phi)|_\Gamma$.
This compatibility condition on the boundary $\Gamma$ yields that
$$
\big[\nabla_x\cdot( \phi\delta x)\big]\big|_\Gamma=\nabla_\Gamma^x\cdot\big[ (\phi|_\Gamma) \delta x_\text{s}\big].
$$
Then we see that the integrand of the second term on the right-hand side of \eqref{LAPa} can be interpreted as
\begin{align}
\left.\left(\frac{\partial W_\text{b}}{\partial \nabla_x \phi}\cdot \mathbf{n}\right)\right|_\Gamma \big[\nabla_x\cdot( \phi\tilde{y})\big]\big|_\Gamma
 = \left.\left(\frac{\partial W_\text{b}}{\partial \nabla_x \phi}\cdot \mathbf{n} \right)\right|_\Gamma \big[\nabla_\Gamma^x\cdot( (\phi|_\Gamma) \tilde{y}_\text{s})\big], \quad \text{on}\ \Gamma.\nonumber
\end{align}
From the above relation and \eqref{LAPa}, using integration by parts, we deduce  that
\begin{align}
&\delta_{(x, x_\text{s})}\mathcal{A}^{\text{total}}\nonumber\\
 &\ \  = -\int_0^T\int_{\Omega_t^x} (\phi \nabla_x \mu_\text{b}) \cdot \tilde{y}\, \d x\d t
 + \int_0^T \int_{\Gamma_t^x} \left(\frac{\partial W_\text{b}}{\partial \nabla_x \phi}\cdot \mathbf{n}\right)\nabla_\Gamma^x\cdot(\phi\tilde{y}_\text{s})\, \d S_x \d t\non\\
 &\quad\ \ -\int_0^T\int_{\Gamma_t^x} (\phi \nabla_\Gamma^x \mu_\text{s}) \cdot \tilde{y}_\text{s}\, \d S_x \d t\non\\
 &\ \  = -\int_0^T\int_{\Omega_t^x} (\phi \nabla_x \mu_\text{b}) \cdot \tilde{y}\, \d x\d t\nonumber\\
 &\quad\ \   -\int_0^T\int_{\Gamma_t^x} \left[\phi \nabla_\Gamma^x \left(\mu_\text{s}+ \frac{\partial W_\text{b}}{\partial \nabla_x \phi}\cdot \mathbf{n}\right)\right] \cdot \tilde{y}_\text{s}\, \d S_x\d t,
  \label{LAPb}
\end{align}
which leads to the conclusion \eqref{LAP}.

\begin{remark}
(1) The formula \eqref{LAPb} implies that the bulk-boundary interaction is in terms of the chemical potential (via force balance relations). The term $\frac{\partial W_\mathrm{b}}{\partial \nabla_x \phi}\cdot \mathbf{n}$ involving the normal derivative can be regarded as the contribution of the bulk chemical potential acting on the boundary $\Gamma$.

(2) Combining the above calculations and the argument in \cite{FJ13,KLG}, we are able to derive hydrodynamical systems for two-phase flows including nontrivial boundary evolution (e.g., the moving contact line problem), which again fulfill the conservation of mass, energy dissipation and force balance relations. This will be presented in a future work.
\end{remark}

\subsection{Proof of Lemma \ref{exLpa}}
For $T\in (0,+\infty)$, we denote $Q_T=\Omega\times[0,T]$ and $\Sigma_T=\Gamma\times[0,T]$. For any $p\in (1,+\infty)$,  $s, r\geq 0$, $W^{(s,r)}_p(Q_T)=W^{s,p}(0,T; L^p(\Omega))\cap L^p(0,T; W^{r,p}(\Omega))$, $W^{(s,r)}_p(\Sigma_T)=W^{s,p}(0,T; L^p(\Gamma))\cap L^p(0,T; W^{r,p}(\Gamma))$, stand for the anisotropic Sobolev spaces (see, e.g., \cite{LSU}).

The proof of Lemma \ref{exLpa} is inspired by the argument in \cite[Lemmas 2.1, 2.2]{MZ05}.
However, different from \cite{MZ05}, in order to keep certain compatibility between the boundary condition and the initial datum on $\Gamma\times\{t=0\}$ for the decomposed system, we first introduce a parabolic shift function defined on $\Gamma$ (cf. \cite{PRZ06}). For $(\phi_0,\psi_0)\in \mathcal{V}^m$ with certain integer $m\geq 2$, we take
\begin{align}
\rho(t)=e^{\kappa A_\Gamma} \psi_0,\quad t\geq0.\label{rho}
\end{align}
If for some $p\in (1,+\infty)$, $\psi_0\in W^{2(1-\frac1p),p}(\Gamma)$ (valid by choosing suitably large $m$), then by the classical $L^p$-theory for the heat equation, we have $\rho(t)\in W^{(1, 2)}_p(\Sigma_T)$ and
\begin{align}
\|\rho(t)\|_{W^{(1, 2)}_p(\Sigma_T)}\leq C\|\psi_0\|_{W^{2(1-\frac1p),p}(\Gamma)}, \quad \forall\, t\in [0,T]. \label{rhoes}
\end{align}

For any $\alpha\in (0,1]$, we consider the following auxiliary non-homogeneous linear problems:
\begin{equation}
\left\{
\begin{aligned}
&\phi_t=\Delta \widetilde{\mu}, \quad \text{with}\ \widetilde{\mu}=-\Delta \phi+\alpha\phi_t+h(t),
&\text{in}\ \Omega\times(0,T),\\
&\partial_\mathbf{n}\widetilde{\mu}=0,\quad\phi|_\Gamma=\rho,
&\text{on}\ \Gamma\times(0,T),\\
&\phi|_{t=0}=\phi_0,
&\text{in}\ \Omega,
\end{aligned}
\label{LL1}
\right.
\end{equation}
and
\begin{equation}
\left\{
\begin{aligned}
&\phi_t=\Delta \widetilde{\mu}, \quad \text{with}\ \widetilde{\mu}=-\Delta \phi+\alpha\phi_t,
&\text{in}\ \Omega\times(0,T),\\
&\partial_\mathbf{n}\widetilde{\mu}=0,\quad \phi|_\Gamma=\psi-\rho,
&\text{on}\ \Gamma\times(0,T),\\
&\phi|_{t=0}=0,
&\text{in}\ \Omega,
\end{aligned}
\label{LL1a}
\right.
\end{equation}
where $\psi=\psi(x,t)$ is a certain regular function defined on $\Gamma\times[0,T]$ such that $\psi|_{t=0}=\psi_0(x)$ and $\rho$ is given by \eqref{rho}.

If we assume that the source term $h(t)$ in \eqref{LL1} satisfies $\langle h(t)\rangle_\Omega=0$, then using the fact that $A_\Omega^0$ is invertible on $\dot{L^2}(\Omega)$, we can write the above systems into the following equivalent form (cf.  \cite[Section 2]{MZ05}):
\begin{equation}
\left\{
\begin{aligned}
&[\alpha+ (A^0_\Omega)^{-1}]\phi_t=\Delta \phi-\frac{|\Gamma|}{|\Omega|}\langle \partial_\mathbf{n}\phi\rangle_\Gamma-h(t),
&\text{in}\ \Omega\times(0,T),\\
&\phi|_\Gamma=\rho,
&\text{on}\ \Gamma\times(0,T),\\
&\phi|_{t=0}=\phi_0,
&\text{in}\ \Omega,
\end{aligned}
\label{LL2}
\right.
\end{equation}
and
\begin{equation}
\left\{
\begin{aligned}
&[\alpha+ (A^0_\Omega)^{-1}]\phi_t=\Delta \phi-\frac{|\Gamma|}{|\Omega|}\langle \partial_\mathbf{n}\phi\rangle_\Gamma,
&\text{in}\ \Omega\times(0,T),\\
&\phi|_\Gamma=\psi-\rho,
&\text{on}\ \Gamma\times(0,T),\\
&\phi|_{t=0}=0,
&\text{in}\ \Omega.
\end{aligned}
\label{LL2a}
\right.
\end{equation}
 The following well-posedness results in the $L^p$-framework follows form the argument in \cite[Lemma 2.1, Corollary 2.1]{MZ05} with minor modifications (i.e., due to the nonhomogeneous Dirichlet boundary in the system \eqref{LL2}):
\begin{lemma}\label{exeLL1}
Assume that $\alpha\in (0,1]$ and $p \in [2,+\infty)$. $\rho$ is given by \eqref{rho} satisfying $\rho\in W^{(1-\frac{1}{2p},2-\frac{1}{p})}_p(\Sigma_T)$.

(i) If $h\in L^p(Q_T)$ with $\langle h(t)\rangle_\Omega=0$ and $\phi_0\in W^{2(1-\frac1p), p}(\Omega)$ with $\phi_0|_\Gamma=\rho|_{t=0}$, then problem \eqref{LL2} admits a unique solution $\phi\in W^{(1,2)}_p(Q_T)$ and the following estimate holds:
\begin{align}
\|\phi\|_{W^{(1,2)}_p(Q_T)}\leq C\Big(\|\phi_0\|_{W^{2(1-\frac1p), p}(\Omega)}+\|\rho\|_{W^{(1-\frac{1}{2p},2-\frac{1}{p})}_p(\Sigma_T)}+\|h\|_{L^p(Q_T)}\Big),\label{esLL1}
\end{align}
where the constant $C>0$ depends on $T$ and $\alpha$, but is independent of $\phi$,  $\phi_0$ and $h$.

(ii)  If $\psi, \rho\in W^{(1-\frac{1}{2p},2-\frac{1}{p})}_p(\Sigma_T)$, then problem \eqref{LL2a} admits a unique solution $\phi\in W^{(1,2)}_p(Q_T)$ with $\langle \phi \rangle_\Omega=0$ and the following estimates hold:
\begin{align}
&\|\phi\|_{W^{(1,2)}_p(Q_T)}\leq C\left(\|\psi\|_{W^{(1-\frac{1}{2p},2-\frac{1}{p})}_p(\Sigma_T)}
+\|\rho\|_{W^{(1-\frac{1}{2p},2-\frac{1}{p})}_p(\Sigma_T)}\right),
\label{esLL2}\\
&\int_0^t(\partial_\mathbf{n}\phi(\tau), \psi(\tau)-\rho(\tau))_{L^2(\Gamma)} d\tau
\nonumber\\
&\quad =\frac12\|\phi(t)\|^2_{(H^1(\Omega))^*}+\frac{\alpha}{2}\|\phi(t)\|_{L^2(\Omega)}^2+\int_0^t \|\nabla \phi(\tau)\|_{L^2(\Omega)}^2 \d\tau,\quad \forall\, t\in [0,T],
\label{esLL2a}
\end{align}
where the constant $C>0$ in \eqref{esLL2} depends on $T$ and $\alpha$, but is independent of $\psi$.
\end{lemma}

Now let us go back to the linear problem \eqref{Lpa}. Put
$$
\widetilde{h}_1=h_1- \langle h_1\rangle_\Omega,\quad \widetilde{h}_2=h_2-  \langle h_2\rangle_\Gamma.
$$
We observe that $\widetilde{h}_1\in L^p(Q_T)$, $\widetilde{h}_2\in L^p(\Sigma_T)$ with null average in $\Omega$ and $\Gamma$, respectively. Since $\langle\phi_t\rangle_\Omega=\langle\psi_t\rangle_\Gamma=0$, then problem \eqref{Lpa} can be written into the following form:
\begin{equation}
\left\{
\begin{aligned}
&[\alpha+ (A^0_\Omega)^{-1}]\phi_t=\Delta\phi-\frac{|\Gamma|}{|\Omega|}\langle\partial_\mathbf{n}\phi\rangle_\Gamma-\widetilde{h}_1,
&\text{in}\ \Omega\times(0,T),\\
&\phi|_\Gamma=\psi,
&\text{on}\ \Gamma\times(0,T),\\
&[\alpha+ (A^0_\Gamma)^{-1}]\psi_t=\kappa\Delta_\Gamma \psi-\psi + \langle\psi\rangle_\Gamma-\partial_\mathbf{n}\phi+\langle \partial_\mathbf{n}\phi\rangle_\Gamma-\widetilde{h}_2,
&\text{on}\ \Gamma\times(0,T),\\
&\phi|_{t=0}=\phi_0(x),
&\text{in}\ \Omega,\\
&\psi|_{t=0}=\psi_0(x):=\phi_0(x)|_{\Gamma},
&\text{on}\ \Gamma.
\end{aligned}
\label{Lpag}
\right.
\end{equation}
We denote by $\mathfrak{T}: W^{(1-\frac{1}{2p},2-\frac{1}{p})}_p(\Sigma_T)\to W^{(1,2)}_p(Q_T)$ the solution operator for problem \eqref{LL2a} and then the normal derivative operator $\partial_\mathbf{n}\mathfrak{T}$ can be viewed as a generalized parabolic Dirichlet--Neumann map. The Dirichlet boundary datum in the auxiliary system \eqref{LL2a} is given by $\psi-\rho$ with $\rho$ being given as in \eqref{rho} and the function $\psi$ will be determined below (see \eqref{trLLg}).
Then we introduce the decomposition
$$\phi=u+\mathfrak{T} (\psi-\rho).$$
 As a consequence,  \eqref{Lpag} can be transformed into the following form for the new unknown variables $(u, \psi)$:
\begin{equation}
\left\{
\begin{aligned}
&[\alpha+ (A^0_\Omega)^{-1}]u_t=\Delta u-\frac{|\Gamma|}{|\Omega|}\langle\partial_\mathbf{n}u\rangle_\Gamma-\widetilde{h}_1,
&\text{in}\ \Omega\times(0,T),\\
&u|_\Gamma=\rho,
&\text{on}\ \Gamma\times(0,T),\\
&[\alpha+ (A^0_\Gamma)^{-1}]\psi_t=\kappa\Delta_\Gamma \psi-\psi + \langle\psi\rangle_\Gamma-\partial_\mathbf{n}(\mathfrak{T} (\psi-\rho))\\
&\qquad \qquad\qquad\quad \ \  +\langle \partial_\mathbf{n}(\mathfrak{T} (\psi-\rho))\rangle_\Gamma-\widehat{h}_2,
&\text{on}\ \Gamma\times(0,T),\\
& \text{with}\ \widehat{h}_2=\partial_\mathbf{n}u-\langle \partial_\mathbf{n}u\rangle_\Gamma+\widetilde{h}_2,
&\text{on}\ \Gamma\times(0,T),\\
&u|_{t=0}=\phi_0(x),
&\text{in}\ \Omega,\\
&\psi|_{t=0}=\psi_0(x):=\phi_0(x)|_{\Gamma},
&\text{on}\ \Gamma.
\end{aligned}
\label{Lpag1}
\right.
\end{equation}
One advantage of the above reformulation is that the system \eqref{Lpag1} is indeed decoupled for $(u, \psi)$. Namely, the equation for $u$ turns out to be independent of $\psi$ and thus it can be solved directly by applying Lemma \ref{exeLL1} (i) such that there exists a unique solution $u$ satisfying
\begin{align}
\|u\|_{W^{(1,2)}_p(Q_T)}\leq C\Big(\|\phi_0\|_{W^{2(1-\frac{1}{p}), p}(\Omega)}+\|\rho\|_{W^{(1-\frac{1}{2p},2-\frac{1}{p})}_p(\Sigma_T)}+\|\widetilde{h}_1\|_{L^p(Q_T)}\Big).\label{esLL1u}
\end{align}
Next, we apply the linear bounded operator $[\alpha+ (A^0_\Gamma)^{-1}]^{-1}$ to the equation of $\psi$ in \eqref{Lpag1} and obtain that
\begin{align}
\alpha\psi_t-\kappa\Delta_\Gamma \psi+\psi= \mathfrak{K}\psi-\alpha[\alpha+ (A^0_\Gamma)^{-1}]^{-1}\widehat{h}_2,\quad\ \ \ \ \text{on}\ \Gamma\times(0,T),\label{trLLg}
\end{align}
where
\begin{align*}
&\mathfrak{K}\psi= -[\alpha+ (A^0_\Gamma)^{-1}]^{-1}(A^0_\Gamma)^{-1} \Big(\kappa\Delta_\Gamma \psi-\psi + \langle\psi\rangle_\Gamma-\partial_\mathbf{n}(\mathfrak{T} (\psi-\rho))+\langle \partial_\mathbf{n}(\mathfrak{T} (\psi-\rho))\rangle_\Gamma\Big)\nonumber\\
&\qquad\quad +\langle\psi\rangle_\Gamma-\partial_\mathbf{n}(\mathfrak{T} (\psi-\rho))+\langle \partial_\mathbf{n}(\mathfrak{T} (\psi-\rho))\rangle_\Gamma.
\end{align*}
Inserting the solution $u$ into the reduced equation \eqref{trLLg} for $\psi$, we easily see from \eqref{esLL1u} that $\widehat{h}_2\in L^p(\Sigma_T)$.
Besides, it follows from \eqref{esLL2} and the trace theorem that
\begin{align}
\|\partial_\mathbf{n}(\mathfrak{T} (\psi-\rho))\|_{L^p(\Sigma_T)}\leq C\Big(\|\psi\|_{W^{(1-\frac{1}{2p}, 2-\frac{1}{p})}_p(\Sigma_T)}+\|\rho\|_{W^{(1-\frac{1}{2p}, 2-\frac{1}{p})}_p(\Sigma_T)}\Big),\nonumber
\end{align}
which further implies
$$
\|\mathfrak{K}\psi\|_{L^p(\Sigma_T)}\leq C\Big(\|\psi\|_{W^{(1-\frac{1}{2p}, 2-\frac{1}{p})}_p(\Sigma_T)}+\|\rho\|_{W^{(1-\frac{1}{2p}, 2-\frac{1}{p})}_p(\Sigma_T)}\Big).
$$
Hence, \eqref{trLLg} can be viewed as a linear heat equation for $\psi$ on $\Gamma$ with compact perturbations. By $L^p$-estimates for parabolic equations, we have
\begin{align}
\|\psi\|_{W^{(1,2)}_p(\Sigma_T)}
& \leq C\Big(\|\psi_0\|_{W^{2(1-\frac{1}{p}), p}(\Gamma)}+\|\psi\|_{W^{(1-\frac{1}{2p}, 2-\frac{1}{p})}_p(\Sigma_T)}\Big)\nonumber\\
&\quad +C\Big(\|\rho\|_{W^{(1-\frac{1}{2p}, 2-\frac{1}{p})}_p(\Sigma_T)}+\|\widehat{h}_2\|_{L^p(\Sigma_T)}\Big).
\label{esLLgp}
\end{align}
Recalling the estimate \eqref{rhoes} for $\rho$ and using the interpolation inequality
$\|\psi\|_{W^{(1-\frac{1}{2p}, 2-\frac{1}{p})}_p(\Sigma_T)}\leq \epsilon \|\psi\|_{W^{(1,2)}_p(\Sigma_T)}+C_\epsilon\|\psi\|_{L^2(\Sigma_T)}$ with $\epsilon>0$ being sufficiently small, we get
\begin{align}
\|\psi\|_{W^{(1,2)}_p(\Sigma_T)}
\leq C\big(\|\psi_0\|_{W^{2(1-\frac{1}{p}), p}(\Gamma)}+\|\psi\|_{L^2(\Sigma_T)}
+\|\widehat{h}_2\|_{L^p(\Sigma_T)}\big).\label{esLLgp1}
\end{align}
In order to estimate the lower-order term $\|\psi\|_{L^2(\Sigma_T)}$ on the right-hand side of \eqref{esLLgp1}, we test \eqref{trLLg} by $\psi_t-\Delta_\Gamma\psi$ on $\Gamma\times (0,t)$ and obtain
\begin{align}
&\frac{\alpha+\kappa}{2}\|\nabla_\Gamma\psi(t)\|_{L^2(\Gamma)}^2
+\frac12\|\psi(t)\|_{L^2(\Gamma)}^2\nonumber\\
&\qquad +\int_0^t\Big(\alpha\|\psi_t\|^2_{L^2(\Gamma)}+\kappa\|\Delta_\Gamma \psi\|_{L^2(\Gamma)}^2+\|\nabla_\Gamma\psi\|_{L^2(\Gamma)}^2\Big)\,\d \tau\nonumber\\
&\quad =\frac{\alpha+\kappa}{2}\|\nabla_\Gamma\psi_0\|_{L^2(\Gamma)}^2
+\frac12\|\psi_0\|_{L^2(\Gamma)}^2
\nonumber\\
&\qquad +\int_0^t\int_\Gamma \Big(\mathfrak{K}\psi-\alpha[\alpha+ (A^0_\Gamma)^{-1}]^{-1}\widehat{h}_2\Big) (\psi_t-\Delta_\Gamma\psi)\, \d S \d\tau,\quad \forall\, t\in [0,T].\nonumber
\end{align}
By H\"older's inequality and the Cauchy--Schwarz inequality, it holds
\begin{align}
&\left|\int_0^t\int_\Gamma \Big(\alpha[\alpha+ (A^0_\Gamma)^{-1}]^{-1}\widehat{h}_2\Big) (\psi_t-\Delta_\Gamma \psi)\, \d S \d\tau\right| \nonumber\\
&\quad \leq C\int_0^t\|\widehat{h}_2\|_{L^2(\Gamma)}\|\psi_t-\Delta_\Gamma\psi\|_{L^2(\Gamma)} \,\d\tau\nonumber\\
&\quad \leq \frac{\kappa}{4}\int_0^t\|\Delta_\Gamma \psi\|_{L^2(\Gamma)}^2\,\d \tau +\frac{\alpha}{4}\int_0^t\|\psi_t\|^2_{L^2(\Gamma)}\,\d \tau
+C\int_0^t\|\widehat{h}_2\|_{L^2(\Gamma)}^2\, \d\tau.
\nonumber
\end{align}
From the trace theorem
$\|\partial_\mathbf{n}(\mathfrak{T} (\psi-\rho))\|_{L^2(\Gamma)}\leq C\|\mathfrak{T} (\psi-\rho)\|_{H^r(\Omega)}$ for some $r\in (\frac32,2)$,
we deduce form \eqref{esLL2}, \eqref{esLLgp1} and the interpolation inequality that
\begin{align}
&\left|\int_0^t\int_\Gamma (\mathfrak{K}\psi) (\psi_t-\Delta_\Gamma\psi)\, \d S \d\tau\right|\nonumber\\
&\quad \leq \frac{\kappa}{8}\int_0^t\|\Delta_\Gamma \psi\|_{L^2(\Gamma)}^2\, \d \tau
+\frac{\alpha}{8}\int_0^t\|\psi_t\|^2_{L^2(\Gamma)}\, \d \tau  \nonumber\\
&\qquad + C\|\mathfrak{T} (\psi-\rho)\|_{W^{(1,2)}_2(Q_T)}^2
+C\int_0^t \|\psi\|_{L^2(\Gamma)}^2 \, \d \tau \nonumber\\
&\quad \leq \frac{\kappa}{8}\int_0^t\|\Delta_\Gamma \psi\|_{L^2(\Gamma)}^2\, \d \tau
+\frac{\alpha}{8}\int_0^t\|\psi_t\|^2_{L^2(\Gamma)}\, \d \tau
+ C\|\psi\|_{W^{(\frac{3}{4},\frac{3}{2})}_2(\Sigma_T)}^2  \nonumber\\
&\qquad + C\|\rho\|_{W^{(\frac{3}{4},\frac{3}{2})}_2(\Sigma_T)}^2 +C\int_0^t \|\psi\|_{L^2(\Gamma)}^2 \, \d \tau \nonumber\\
&\quad \leq \frac{\kappa}{4}\int_0^t\|\Delta_\Gamma \psi\|_{L^2(\Gamma)}^2\, \d \tau
+\frac{\alpha}{4}\int_0^t\|\psi_t\|^2_{L^2(\Gamma)}\,\d \tau
+ C\int_0^t \|\psi\|_{H^1(\Gamma)}^2\, \d\tau\nonumber\\
&\qquad +C\|\psi_0\|^2_{H^1(\Gamma)}+C\int_0^t\|\widehat{h}_2\|_{L^2(\Gamma)}^2\, \d\tau.\nonumber
\end{align}
Using the fact $\langle\psi(t)\rangle_\Gamma=\langle\psi_0\rangle_\Gamma$ and Poincar\'e's inequality, we infer from the above estimates that
\begin{align}
&\frac{\alpha+\kappa}{2}\|\nabla_\Gamma\psi(t)\|_{L^2(\Gamma)}^2
+\frac12\|\psi(t)\|_{L^2(\Gamma)}^2
+\frac{\kappa}{2}\int_0^t\|\Delta_\Gamma \psi\|_{L^2(\Gamma)}^2\,\d \tau\nonumber\\
&\qquad +\frac{\alpha}{2}\int_0^t\|\psi_t\|^2_{L^2(\Gamma)}\,\d \tau\nonumber\\
&\quad \leq \frac{\alpha+\kappa}{2}\|\nabla_\Gamma\psi_0\|_{L^2(\Gamma)}^2
+\frac12\|\psi_0\|_{L^2(\Gamma)}^2+C\|\psi_0\|^2_{H^1(\Gamma)}+ C\int_0^t\|\widehat{h}_2\|_{L^2(\Gamma)}^2\, \d\tau
\nonumber\\
&\qquad
+ C \int_0^t \Big(\|\nabla_\Gamma \psi\|_{L^2(\Gamma)}^2+\|\psi\|_{L^2(\Gamma)}^2\Big)\, \d\tau,\quad \forall\, t\in [0,T],\nonumber
\end{align}
where $C$ is a constant depending on $\kappa$, $\alpha$ and $T$.
Applying Gronwall's lemma, we deduce that
\begin{align}
\|\psi(t)\|_{H^1(\Gamma)}^2&\leq C_T\Big(\|\psi_0\|_{H^1(\Gamma)}^2+\|\widehat{h}_2\|_{L^2(\Sigma_T)}^2\Big),\quad \forall\, t\in [0,T],\nonumber
\end{align}
which together with \eqref{esLLgp1} yields the $L^p$-estimate
\begin{align}
\|\psi\|_{W^{(1,2)}_p(\Sigma_T)}
\leq C\Big(\|\psi_0\|_{W^{2(1-\frac{1}{p}), p}(\Gamma)}+\|\widehat{h}_2\|_{L^p(\Sigma_T)}\Big).
\label{esLLgpa}
\end{align}
By the classical Leray--Schauder fixed point theorem, we can conclude from \eqref{esLLgpa} that problem \eqref{trLLg} has a unique strong solution $\psi$.

Therefore, under the assumptions of Lemma \ref{exLpa}, for $d=2,3$, we can choose $p\in (3, \frac{10}{3})$ such that $H^2(\Omega)\hookrightarrow W^{2(1-\frac{1}{p}), p}(\Omega)$ and $L^2(0, T; H^1(\Omega))\cap L^\infty(0,T; L^2(\Omega))\hookrightarrow L^p(\Omega_T)$. Then we conclude from the above argument and estimates \eqref{esLL1u}, \eqref{esLLgpa} that problem \eqref{Lpag} admits a unique strong solution $(\phi, \psi)\in W^{(1,2)}_p(Q_T)\times W^{(1,2)}_p(\Sigma_T)$ satisfying
 \begin{align}
& \|\phi\|_{W^{(1,2)}_p(Q_T)}+\|\psi\|_{W^{(1,2)}_p(\Sigma_T)}\nonumber\\
 &\quad \leq C\Big(\|\phi_0\|_{W^{2(1-\frac{1}{p}),p}(\Omega)}
 +\|\psi_0\|_{W^{2(1-\frac{1}{p}),p}(\Gamma)}
 +\|\widetilde{h}_1\|_{L^p(Q_T)}+\|\widetilde{h}_2\|_{L^p(\Sigma_T)}\Big),
 \label{esLLgp2}
\end{align}
where $C$ depends on $T$, $\alpha$, $\kappa$, $\Omega$ and $\Gamma$.

It remains to show the estimates \eqref{esLIN}--\eqref{esLINh}. Testing the equations for $\phi$ and $\psi$ in \eqref{Lpag} by $\phi_t$, $\psi_t$, respectively, integrating over $\Omega$ and $\Gamma$, integrating with respect to time and adding the resultants together, we have
\begin{align}
&\frac12\|\nabla \phi(t)\|_{L^2(\Omega)}^2+\frac{\kappa}{2}\|\nabla_\Gamma \psi(t)\|_{L^2(\Gamma)}+\frac12\|\psi(t)\|_{L^2(\Gamma)}^2\nonumber\\
&\qquad + \alpha\int_0^t\left(\|\phi_t(\tau)\|_{L^2(\Omega)}^2+\|\psi_t(\tau)\|_{L^2(\Gamma)}^2\right)\,\d\tau\nonumber\\
&\qquad +\int_0^t\left(\|\phi_t(\tau)\|_{(H^1(\Omega))^*}^2+\|\psi_t(\tau)\|_{(H^1(\Gamma))^*}^2\right) \,\d\tau\nonumber\\
&\quad = \frac12\|\nabla \phi_0\|_{L^2(\Omega)}^2+\frac{\kappa}{2}\|\nabla_\Gamma \psi_0\|_{L^2(\Gamma)}^2+\frac12\|\psi_0\|_{L^2(\Gamma)}^2\nonumber\\
&\qquad +\int_0^t(h_1, \phi_t(\tau))_{L^2(\Omega)}+(h_2, \psi_t(\tau))_{L^2(\Gamma)} \,\d\tau.\nonumber
\end{align}
The last term on the right-hand side can be estimated by using the Cauchy--Schwarz inequality
\begin{align}
&\int_0^t(h_1, \phi_t(\tau))_{L^2(\Omega)}+(h_2, \psi_t(\tau))_{L^2(\Gamma)} \,\d\tau\nonumber\\
&\quad \leq \frac{\alpha}{2}\int_0^t\left(\|\phi_t(\tau)\|_{L^2(\Omega)}^2+\|\psi_t(\tau)\|_{L^2(\Gamma)}^2\right)\,\d\tau\nonumber\\
&\qquad +\frac{1}{2\alpha}\int_0^t(\|h_1\|_{L^2(\Omega)}^2+\|h_2\|_{L^2(\Gamma)}^2)\,\d\tau,\nonumber
\end{align}
and as a result,
\begin{align}
&\|(\phi(t), \psi(t))\|_{\mathcal{V}^1}^2+ \alpha\|(\phi_t,\psi_t)\|_{L^2(0,t; \mathcal{H})}^2+ \|(\phi_t,\psi_t)\|_{L^2(0,t; (H^1(\Omega))^*\times(H^1(\Gamma))^*)}^2\nonumber\\
&\quad \leq  \|(\phi_0, \psi_0)\|_{\mathcal{V}^1}^2 + \alpha^{-1} \|(h_1,h_2)\|_{L^2(0,t; \mathcal{H})}^2, \quad \forall\, t\in [0,T].
\label{esLin1}
\end{align}
For a.e. $t\in (0,T)$, we consider the linear elliptic problem (cf. \eqref{Lpag})
\begin{equation}
\left\{
\begin{aligned}
&-\Delta\phi =-[\alpha+ (A^0_\Omega)^{-1}]\phi_t-\frac{|\Gamma|}{|\Omega|}\langle\partial_\mathbf{n}\phi\rangle_\Gamma
-\widetilde{h}_1,
&\text{in}\ \Omega,\\
&\phi|_\Gamma=\psi,
&\text{on}\ \Gamma,\\
&-\Delta_\Gamma \psi+\psi +\partial_\mathbf{n}\phi= -[\alpha+ (A^0_\Gamma)^{-1}]\psi_t+ \langle\psi\rangle_\Gamma+\langle \partial_\mathbf{n}\phi\rangle_\Gamma-\widetilde{h}_2, &\text{on}\ \Gamma.
\end{aligned}
\label{esLinelli}
\right.
\end{equation}
It follows from Lemma \ref{esLepH2} and  estimates similar to \eqref{espncla}, \eqref{inttr} and \eqref{intpo} that
\begin{align}
\|(\phi,\psi)\|_{\mathcal{V}^2}
&\leq C(\alpha \|(\phi_t, \psi_t)\|_{\mathcal{H}}+\|(\phi_t,\psi_t)\|_{H^1(\Omega))^*\times(H^1(\Gamma))^*})\nonumber\\
&\quad +C(|\langle \partial_\mathbf{n}\phi\rangle_\Gamma|+\|\phi\|_{H^1(\Omega)}+|\langle\psi_0\rangle_\Gamma|+\|(h_1,h_2)\|_{\mathcal{H}})\nonumber\\
&\leq \frac12\|\phi\|_{H^2(\Omega)}+C(\alpha \|(\phi_t, \psi_t)\|_{\mathcal{H}}+\|(\phi_t,\psi_t)\|_{H^1(\Omega))^*\times(H^1(\Gamma))^*})\nonumber\\
&\quad +C(\|\phi\|_{H^1(\Omega)}+|\langle\psi_0\rangle_\Gamma|+\|(h_1,h_2)\|_{\mathcal{H}}),
\label{esLinV2}
\end{align}
for a.e. $t\in [0,T]$, where the constant $C$ is independent of $\alpha$.
Besides, in view of \eqref{Lpa} the following estimates on the mean values hold:
\begin{align}
|\langle \widetilde{\mu}\rangle_\Omega|
&\leq |\Omega|^{-1}|\Gamma| |\langle \partial_\mathbf{n}\phi\rangle_\Gamma|+|\langle h_1\rangle_\Omega|\nonumber\\
&\leq C(\|\phi\|_{H^2(\Omega)}+\|h_1\|_{L^2(\Omega)}),
\label{esmmuo}
\end{align}
\begin{align}
|\langle \widetilde{\mu}_\Gamma\rangle_\Gamma|
&\leq |\langle\psi\rangle_\Gamma|+|\langle \partial_\mathbf{n}\phi\rangle_\Gamma|+ |\langle h_2\rangle_\Gamma|\nonumber\\
&\leq C(|\langle\psi_0\rangle_\Gamma|+\|\phi\|_{H^2(\Omega)}+\|h_2\|_{L^2(\Gamma)}).
\label{esmmug}
\end{align}
Then it follows from the relations
\begin{align}
&\|\nabla \widetilde{\mu}\|_{ L^2(\Omega)}= \|\phi_t\|_{H^1(\Omega))^*},\quad \|\nabla_\Gamma \widetilde{\mu}_\Gamma \|_{L^2(\Gamma)}= \|\psi_t\|_{(H^1(\Gamma))^*},\nonumber
\end{align}
 and Poincar\'e's inequality  that
\begin{align}
\|\widetilde{\mu}\|_{H^1(\Omega)}
   &\leq C(\|\nabla \widetilde{\mu}\|_{L^2(\Omega)}+|\langle \widetilde{\mu}\rangle_\Omega|)\nonumber\\
   &\leq C(\|\phi_t\|_{(H^1(\Omega))^*}+\|\phi\|_{H^2(\Omega)}+\|h_1\|_{L^2(\Omega)}),\label{esLin3}
   \end{align}
   \begin{align}
\|\widetilde{\mu}_\Gamma\|_{H^1(\Gamma)}
    &\leq  C(\|\nabla_\Gamma \widetilde{\mu}_\Gamma\|_{L^2(\Gamma)}+|\langle \widetilde{\mu}_\Gamma \rangle_\Gamma|)\nonumber\\
&  \leq  C(\|\psi_t\|_{(H^1(\Gamma))^*}+|\langle\psi_0\rangle_\Gamma|
+\|\phi\|_{H^2(\Omega)}+\|h_2\|_{L^2(\Gamma)}).\label{esLin4}
\end{align}
Collecting the estimates \eqref{esLin1}--\eqref{esLin4}, we arrive at the conclusion \eqref{esLIN}.

Next, differentiating \eqref{Lpa} with respect to time, we get
\begin{equation}
\left\{
\begin{aligned}
&\phi_{tt}-\Delta \widetilde{\mu}_t=0,
\quad \text{with}\ \ \widetilde{\mu}_t=-\Delta\phi_t+\alpha\phi_{tt}+(h_1)_t,
&\text{in}\ \Omega\times(0,T),\\
&\partial_\mathbf{n}\widetilde{\mu}_t=0,\quad \phi_t|_\Gamma=\psi_t,
&\text{on}\ \Gamma\times(0,T),\\
&\psi_{tt}-\Delta_\Gamma (\widetilde{\mu}_\Gamma)_t=0,
&\text{on}\ \Gamma\times(0,T),\\
&\quad \text{with}\quad (\widetilde{\mu}_\Gamma)_t=-\kappa \Delta_\Gamma\psi_t+\psi_t+\alpha \psi_{tt}+\partial_\mathbf{n}\phi_t+(h_2)_t,
&\text{on}\ \Gamma\times(0,T),\\
&\phi_t|_{t=0}=\Delta \widetilde{\mu}(0),
&\text{in}\ \Omega,\\
&\psi_t|_{t=0}=\Delta_\Gamma \widetilde{\mu}_\Gamma(0),
&\text{on}\ \Gamma.
\end{aligned}
\label{Lpat}
\right.
\end{equation}
Testing the first equation by $(A^0_\Omega)^{-1} \phi_t$, the third equation by $(A_\Gamma^0)^{-1}\psi_t$, adding the resultants together, we obtain
\begin{align}
&\frac{1}{2}\frac{\d}{\d t}\Big(\|\phi_t\|_{(H^1(\Omega))^*}^2
+ \alpha \|\phi_t\|_{L^2(\Omega)}^2+\|\psi_t\|_{(H^1(\Gamma))^*}^2
+\alpha \|\psi_t\|_{L^2(\Gamma)}^2\Big)\nonumber\\
&\qquad + \|\nabla \phi_t\|_{L^2(\Omega)}^2
 +\kappa\|\nabla_\Gamma \psi_t\|_{L^2(\Gamma)}^2+\|\psi_t\|_{L^2(\Gamma)}^2\nonumber\\
&\quad =-\int_\Omega (h_1)_t \phi_t\,\d x-\int_\Gamma (h_2)_t\psi_t\,\d S\nonumber\\
&\quad \leq \frac12\|\phi_t\|_{L^2(\Omega)}^2+ \frac12\|\psi_t\|_{L^2(\Gamma)}^2
+\frac12\|(h_1)_t\|_{L^2(\Omega)}^2+\frac12\|(h_2)_t\|_{L^2(\Gamma)}^2.
\label{highap}
\end{align}
Using the facts $\langle\phi_t\rangle_\Omega=\langle\psi_t\rangle_\Gamma=0$, we infer from Poincar\'e's inequality that the following interpolations hold
\begin{align*}
&\|\phi_t\|_{L^2(\Omega)}^2\leq C\|\phi_t\|_{(H^1(\Omega))^*}\|\nabla \phi_t\|_{L^2(\Omega)},\\
&\|\psi_t\|_{L^2(\Gamma)}^2\leq C\|\psi_t\|_{(H^1(\Gamma))^*}\|\nabla_\Gamma \psi_t\|_{L^2(\Gamma)}.
\end{align*}
Using Young's inequality and \eqref{highap}, we infer from Gronwall's lemma that
\begin{align}
&\|\phi_t(t)\|_{(H^1(\Omega))^*}^2+ \alpha \|\phi_t(t)\|_{L^2(\Omega)}^2
+\|\psi_t(t)\|_{(H^1(\Gamma))^*}^2+\alpha \|\psi_t(t)\|_{L^2(\Gamma)}^2\nonumber\\
&\qquad +\int_0^t \Big(\|\nabla \phi_t\|_{L^2(\Omega)}^2
+ \kappa\|\nabla_\Gamma \psi_t\|_{L^2(\Gamma)}^2+\|\psi_t\|_{L^2(\Gamma)}^2\Big)\,\d \tau\nonumber\\
&\quad \leq Ce^{Ct}\Big(\|\nabla\widetilde{\mu}(0)\|_{L^2(\Omega)}^2
+ \alpha \|\Delta \widetilde{\mu}(0)\|_{L^2(\Omega)}^2
\Big)\nonumber\\
&\qquad +Ce^{Ct}\Big(\|\nabla_\Gamma \widetilde{\mu}_\Gamma(0)\|_{L^2(\Gamma)}^2
+\alpha \|\Delta_\Gamma \widetilde{\mu}_\Gamma(0)\|_{L^2(\Gamma)}^2\Big)\nonumber\\
&\qquad +C e^{Ct} \Big(\|(h_1)_t\|_{L^2(Q_T)}^2+\|(h_2)_t\|_{L^2(\Sigma_T)}^2\Big),\quad\ \forall\, t\in [0,T],
\label{highapt}
\end{align}
where the constant $C$ is independent of $\alpha$.
From the following relations
\begin{align*}
&\widetilde{\mu}(0)-\alpha\Delta \widetilde{\mu}(0)=-\Delta \phi_0+ h_1(0),\quad \partial_\mathbf{n} \widetilde{\mu}(0)|_\Gamma=0,\\
&\widetilde{\mu}_\Gamma(0)
-\alpha\Delta_\Gamma\widetilde{\mu}_\Gamma(0)
=-\kappa\Delta_\Gamma \psi_0+\psi_0+\partial_\mathbf{n}\phi_0+h_2(0)
\end{align*}
we infer that
\begin{align*}
&\|\widetilde{\mu}(0)\|_{L^2(\Omega)}
+\|\widetilde{\mu}_\Gamma(0)\|_{L^2(\Gamma)}
+\alpha^\frac12\|\nabla\widetilde{\mu}(0)\|_{L^2(\Omega)}
+\alpha^\frac12 \|\nabla_\Gamma\widetilde{\mu}_\Gamma(0)\|_{L^2(\Gamma)}\\
&\qquad +\alpha\|\Delta\widetilde{\mu}(0)\|_{L^2(\Omega)}
+\alpha\|\Delta_\Gamma\widetilde{\mu}_\Gamma(0)\|_{L^2(\Gamma)}\\
&\quad \leq C\Big(\|\phi_0\|_{H^2(\Omega)}+ \kappa\|\Delta_\Gamma \psi_0\|_{L^2(\Gamma)}+\|h_1(0)\|_{L^2(\Omega)}+\|h_2(0)\|_{L^2(\Gamma)}\Big),
\end{align*}
where the constant $C$ is independent of $\alpha$. Recalling \eqref{highapt}, we obtain for $t\in [0,T]$,
\begin{align}
&\|\phi_t(t)\|_{(H^1(\Omega))^*}^2+ \alpha \|\phi_t(t)\|_{L^2(\Omega)}^2
+\|\psi_t(t)\|_{(H^1(\Gamma))^*}^2+\alpha \|\psi_t(t)\|_{L^2(\Gamma)}^2\nonumber\\
&\qquad +\int_0^t \Big(\|\nabla \phi_t\|_{L^2(\Omega)}^2+ \kappa\|\nabla_\Gamma \psi_t\|_{L^2(\Gamma)}^2+\|\psi_t\|_{L^2(\Gamma)}^2\Big)\, \d \tau\nonumber\\
&\quad \leq Ce^{Ct}\alpha^{-1}\Big(\|(\phi_0, \psi_0)\|_{\mathcal{V}^2}^2+\|h_1\|_{H^1(0,T; L^2(\Omega))}^2+\|h_2\|_{H^1(0,T; L^2(\Gamma))}^2\Big),
\label{esLint1}
\end{align}
where the constant $C$ is independent of $\alpha$.
For the elliptic problem \eqref{esLinelli}, we infer from Lemma \ref{esLepH2} that
\begin{align}
\|(\phi,\psi)\|_{\mathcal{V}^3}
&\leq C(\alpha \|(\phi_t, \psi_t)\|_{\mathcal{V}^1}+\|(\phi_t,\psi_t)\|_{H^1(\Omega))^*\times(H^1(\Gamma))^*})\nonumber\\
&\quad +C(|\langle \partial_\mathbf{n}\phi\rangle_\Gamma|+|\langle\psi_0\rangle_\Gamma|+\|h_1\|_{H^1(\Omega)}+\|h_2\|_{H^1(\Gamma)}),
\label{esLinV3}
\end{align}
for a.e. $t\in [0,T]$, where the constant $C$ is independent of $\alpha$. Then from \eqref{esmmuo}--\eqref{esmmug}, it follows that
\begin{align}
\|\widetilde{\mu}\|_{H^2(\Omega)}
   &\leq C(\|\Delta \widetilde{\mu}\|_{L^2(\Omega)}+|\langle \widetilde{\mu}\rangle_\Omega|)\nonumber\\
   &\leq C(\|\phi_t\|_{L^2(\Omega)}+\|\phi\|_{H^2(\Omega)}+\|h_1\|_{L^2(\Omega)}),\label{esLin5}
   \end{align}
   \begin{align}
\|\widetilde{\mu}_\Gamma\|_{H^2(\Gamma)}
    &\leq  C(\|\Delta_\Gamma \widetilde{\mu}_\Gamma\|_{L^2(\Gamma)}+|\langle \widetilde{\mu}_\Gamma \rangle_\Gamma|)\nonumber\\
&  \leq  C(\|\psi_t\|_{L^2(\Gamma)}+|\langle\psi_0\rangle_\Gamma|
+\|\phi\|_{H^2(\Omega)}+\|h_2\|_{L^2(\Gamma)}).
\label{esLin6}
\end{align}
Collecting the estimates \eqref{esLinV2} and \eqref{esLint1}--\eqref{esLin6}, we arrive at the conclusion \eqref{esLINh}.

 The proof of Lemma \ref{exLpa} is complete.


\section{Acknowledgements}
The authors would like to thank the anonymous referees for their careful reading
of an initial version of this paper and for several helpful comments that allowed
us to improve the presentation. The authors also want to thank Professors P. Colli, T. Fukao, C. Gal, H. Garcke, T.-Z. Qian and U. Stefanelli for helpful discussions.

\section{Compliance with Ethical Standards}
\begin{itemize}
\item \textbf{Funding}: C. Liu is partially supported by NSF grants DMS-1714401, DMS-1412005.
H. Wu is partially supported by NNSFC grant No. 11631011 and the Shanghai Center for Mathematical Sciences at Fudan University.
\item \textbf{Conflict of Interest}: The authors declare that they have no conflict of interest.
\end{itemize}


\end{document}